\documentclass[a4paper,12pt]{article}

\makeatletter
\newif\if@check@engine  \@check@enginetrue 
\makeatother

 \usepackage[usenames,dvipsnames]{pstricks}
 \usepackage{epsfig}

\usepackage{hyperref}

 \usepackage{pstricks}

\usepackage{amsmath}
\usepackage{amsthm}
\usepackage{amssymb}
\usepackage{comment}
\usepackage{latexsym}
\usepackage{array}
\usepackage{enumerate}

\numberwithin{equation}{section}

\newcommand{\skipfig}[1]{#1}

\newcommand{\CC}{\mathbb{C}}

\newcommand{\RR}{\mathbb{R}}
\newcommand{\QQ}{\mathbb{Q}}
\newcommand{\ZZ}{\mathbb{Z}}

\newcommand{\R}{\mathcal{R}}

\newcommand{\PP}{{\mathbb P}}

\newcommand{\vvspace}{\vspace{1.5ex}}

\newcommand{\mbbo}{{\bf 1}}

\renewcommand{\dim}{{\rm dim}}
\newcommand{\codim}{{\rm codim}}
\newcommand{\Vol}{{\rm Vol}}

\newcommand{\e}{\varepsilon}

\newcommand{\relint}{{\rm rel.int}}

\newcommand{\simto}{\overset{\sim}{\longrightarrow}}

\newcommand{\dint}{\displaystyle \int}

\newtheorem{theorem}{Theorem}[section]
\newtheorem{proposition}[theorem]{Proposition}
\newtheorem{conjecture}[theorem]{Conjecture}
\newtheorem{lemma}[theorem]{Lemma}
\newtheorem{corollary}[theorem]{Corollary}

\theoremstyle{definition}
\newtheorem{definition}[theorem]{Definition}
\newtheorem*{definition*}{Definition}
\newtheorem*{example*}{Example}
\newtheorem{remark}[theorem]{Remark}
\newtheorem{example}[theorem]{Example}

\title{On the monodromy conjecture for 
non-degenerate hypersurfaces 
\footnote{{\bf 2010 Mathematics 
Subject Classification: }14M25, 
32S40}}

\author{Alexander ESTEROV 
\footnote{HSE University\newline The article was prepared within the framework of the Academic Fund Program at the National Research University Higher School of Economics (HSE University) in 2018 -- 2019 (grant 18-01-0029) and by the Russian Academic Excellence Project ``5-100''.}, 
Ann LEMAHIEU 
\footnote{This research is partially supported by the MCI-Spain grant
MTM2010-21740-C02, the ANR `SUSI' project and the ANR `Defigeo' project.} 
and Kiyoshi TAKEUCHI }

\begin{document}

\maketitle

\begin{abstract}
The monodromy conjecture is an umbrella term for several conjectured relationships
between poles of zeta functions, monodromy eigenvalues and roots of
Bernstein-Sato polynomials in arithmetic geometry and singularity theory.
Even the weakest of these relations --- the Denef--Loeser conjecture on
topological zeta functions --- is open for surface singularities.

We prove it for a wide class of multidimensional singularities that are
non-degenerate with respect to their Newton polyhedra, including all such
singularities of functions of four variables.

A crucial difference from the known case of three variables is the existence of
degenerate singularities arbitrarily close to a non-degenerate one. Thus, even aiming
at the study of non-degenerate singularities, we have to go beyond this
setting.

We develop new tools to deal with such multidimensional phenomena, and conjecture how the proof for
non-degenerate singularities of arbitrarily many variables might look like.

\vspace{1ex}

{\sc Sur la conjecture de monodromie pour les singularités d'hypersurfaces nondégénérées.}

\vspace{1ex}

La conjecture de monodromie est un terme parapluie pour plusieurs relations conjecturées entre les pôles de fonctions zêta, les valeurs propres de monodromie et les zéros de polynômes de Bernstein en géométrie arithmétique et en théorie des singularités. 
Même la plus faible de ces relations - la conjecture de Denef-Loeser pour la fonction zêta topologique - est ouverte pour les singularités de surfaces. 

Nous démontrons cette conjecture pour une grande classe de singularités multidimensionnelles qui sont nondégénérées pour leur polyèdre de Newton, y incluses toutes celles de fonctions de quatre variables.

Une différence cruciale avec le cas de trois variables est l'existence de singularités dégénérées proche d'une singularité nondégénérée. Ainsi, même pour l'étude des singularités nondégénérées, il faut surmonter le contexte des polyèdres de Newton et des résolutions toriques. 

Finalement nous esquissons l'idée d'une preuve pour les singularités nondégénérées en un nombre quelconque de variables.
\end{abstract}

\tableofcontents

\section{Introduction}\label{sec:1}
Over the fields $\RR$ and $\CC$ it is well-known 
that the poles of the local zeta function associated to 
a polynomial $f$ are contained in the set of roots 
of the Bernstein-Sato polynomial and 
integer shifts of them. By a celebrated theorem 
of Kashiwara and Malgrange, this implies that for any 
such pole $s_0 \in \QQ$ the complex number 
$\exp (2 \pi i s_0) \in \CC$ is an eigenvalue of the 
monodromies of the complex hypersurface defined by $f$. 
Igusa predicted a similar beautiful relationship 
between the poles of $p$-adic integrals and the 
complex monodromies. This is now called the monodromy 
conjecture (see the papers of Denef \cite{D}, 
Nicaise \cite{N} and Denef and Loeser \cite{D-L-3} for excellent reviews on 
this subject). 
Later in \cite{D-L-1}, Denef and Loeser introduced 
the local topological zeta function $Z_{ {\rm top}, f}(s)$ 
associated to $f$ and proposed a weaker version of 
the monodromy conjecture. However, even this weaker version, proposed thirty years ago, is proved so far without any restrictions in dimension two only (\cite{L-1}). 

For other important contributions to 
this Denef-Loeser conjecture, 
see for example \cite{L-2} by Loeser, where he studies the monodromy 
conjecture for non-degenerate singularities satisfying some 
non-resonance conditions. The result holds for  isolated singularities. 
In \cite{L-V}, the second author and Van Proeyen restrict 
to non-degenerate surface singularities and then prove 
the monodromy conjecture without further conditions. 
In the works \cite{A-B-C-M} and \cite{A-B-C-M2}, 
Artal Bartolo--Cassou-Nogu\`es--Luengo--Melle 
Hern\'andez prove the monodromy conjecture for some 
particular classes of singularities for 
which they establish an explicit formula for 
the topological zeta function. It 
was also proved for hyperplane arrangements 
by Budur, Mustata and Teitler \cite{B-M-T}. 
In \cite{V-1} and \cite{V-2}, 
Veys obtained various results in dimension 3. 
Moreover, the case of homogeneous and 
isolated quasihomogeneous singularities was proved by 
Blanco--Budur--van der Veer \cite{B-B-V} 
and Rodrigues--Veys \cite{R-V}. 
Other important contributions to the 
Denef-Loeser conjecture and related results include 
\cite{B-V}, \cite{G-L}, \cite{L-Ve}, \cite{N-V}, 
\cite{Veys}. 

Generalizing the Igusa zeta function to an ideal and using the notion of Verdier monodromy, one can similarly formulate the monodromy conjecture for ideals. At the level of ideals, the conjecture has only been proven in full generality for ideals in two variables (\cite{V-Ve}). Very recently Musta\c t\u a showed that the monodromy conjecture for polynomials implies the monodromy conjecture for ideals (\cite{M}).

The aim of the present paper is to explore to what extent 
the results of \cite{L-V} hold true in higher dimensions, and what we are missing to step one dimension higher for non-degenerate singularities. 

A crucial difference that we observe in dimension four is the existence of degenerate singularities arbitrarily close to a non-isolated non-degenerate singularity. So, even aiming at the study of non-degenerate singularities, 
we have to go beyond the setting of Newton polyhedra and toric resolutions at some point. 

This is in contrast to all preceding results on non-isolated singularities: in the three-dimensional setting of \cite{L-V}, all singularities close to a non-degenerate one are non-degenerate, and, in the setting of \cite{A-B-C-M2}, all singularities close to a quasi-ordinary one are quasi-ordinary.

The paper consists of three parts. In Sections \ref{sec:3}--\ref{sec:4new}, we study configurations of facets of the Newton polyhedron that do not assure the existence of the corresponding pole of the topological zeta function. In Sections \ref{sec:4}--\ref{sec:6}, we study configurations of faces that, on the contrary, always non-trivially contribute to the multiplicity of the expected monodromy eigenvalue. Finally, in the last section, we use these results to prove the monodromy conjecture for singularities of non-degenerate functions of four variables.

\begin{theorem} The Denef--Loeser monodromy conjecture (and moreover its polyhedral version from \cite{E}) holds true for all non-degenerate hypersurface singularities of four variables.
\end{theorem}

Our proof of this theorem admits the following conjectural generalization to arbitrary dimension.
%\begin{definition}\label{defbfacet0}
%1) A lattice simplex in ${\mathbb R}^n$ with the standard coordinate system $v_1,\ldots,v_n$ is called a $B_1$-simplex with respect to the $i$-th coordinate if one of its vertices lies in the plane $v_i=1$ and the others in the plane $v_i=0$.
%2) A lattice polytope in ${\mathbb R}^n$ is called a $B$-polytope, if every lattice simplex that it contains is a $B_1$-simplex.
%\end{definition}
Let $f:({\mathbb C}^n,0)\to({\mathbb C},0)$
be a germ of a holomorphic function non-degenerate with
respect to its Newton polyhedron $\Gamma_+(f)$.

\begin{definition}
1) A \emph{bounded facet $F$ of the Newton
polyhedron $\Gamma_+(f)$ produces
the number $s_0$}, if the affine span of
$F$ is given by an equation $a_1v_1+\cdots+a_nv_n=q$
with coprime coefficients $a_i$
such that $s_0=-(a_1+\cdots+a_n)/q$.

2) A \emph{bounded face $F$ of the Newton polyhedron $\Gamma_+(f)$
is said to produce the number $s_0$}, if every
bounded facet of $\Gamma_+(f)$ containing $F$
produces $s_0$.
\end{definition}
We say that a polytope $P$ is inscribed into a polytope $Q$ if $\dim\, P=\dim\, Q$, $P\subset Q$ and ${\rm vert}\, P\subset {\rm vert}\, Q$ (where ${\rm vert}$ denotes the set of vertices).
\begin{conjecture} 
%1) If all faces of the Newton polyhedron $\Gamma_+(f)$, contributing to the number $s_0$, are $B$-polytopes, then $s_0$ is not a pole of the topological zeta function of $f$.
%2) Otherwise, $\exp(2\pi i s_0)$ is an eigenvalue of the monodromy of $f$ at some point near the origin (and, moreover, a nearby tropical monodromy eigenvalue of the polyhedron $\Gamma_+(f)$ in the sense of \cite{E}).
%===============
Let $\mathcal{F}$ be the set of all faces of the Newton
polyhedron $\Gamma_+(f)\subset\RR^n$, producing the number $s_0$, and let $\mathcal{V}$ be the set of all of their vertices having at least one coordinate equal to 1.

1) Assume there exists a function $b:\mathcal{V}\to\{1,\ldots,n\}$, assigning to every vertex $v\in\mathcal{V}$ the index of one of its unit coordinates, such that every simplex $\Delta$ inscribed into a face from $\mathcal{F}$ has some vertex $v\in\mathcal{V}$ for which the other vertices ${\rm vert}\, \Delta\setminus\{v\}$ belong to the $b(v)$-th coordinate hyperplane.
Then $s_0$ is not a pole of the
topological zeta function of $f$.

2) If such function $b$ does not exist, then $\exp(2\pi i s_0)$ is an eigenvalue
of the monodromy of $f$ at some point near the origin
(and, moreover, a nearby tropical monodromy eigenvalue of the polyhedron
$\Gamma_+(f)$ in the sense of \cite{E}).
\end{conjecture}

\begin{comment}
\begin{definition}\label{defbfacet0}
1) A lattice simplex in ${\mathbb R}^n$ with the standard coordinate system $v_1,\ldots,v_n$ is called a $B_1$-simplex with respect to the $i$-th coordinate if one of its vertices lies in the plane $v_i=1$ and the others in the plane $v_i=0$.

2) A lattice polytope in ${\mathbb R}^n$ is called a $B$-polytope, if every lattuce simplex that it contains is a $B_1$-simplex.
\end{definition}
Let $f:({\mathbb C}^n,0)\to({\mathbb C},0)$ be a germ of a holomorphic function non-degenerate with respect to its Newton polyhedron $N$.
\begin{definition}
1) A bounded facet $F$ of the Newton polyhedron $N$ is said to contribute the number $s_0$, if the affine span of $F$ is given by an equation $a_1v_1+\ldots+a_nv_n=q$ with coprime coefficients such that $s_0=-(a_1+\ldots+a_n)/q$.

2) A bounded face $F$ of the Newton polyhedron $N$ is said to contribute the number $s_0$, if every bounded facet of $N$ containing $F$ contributes $s_0$.
\end{definition}
\begin{conjecture} 1) If all faces of the Newton polyhedron $N$, contributing the number $s_0$, are $B$-polytopes, then $s_0$ is not a pole of the topological zeta function of $f$.

2) Otherwise, $\exp(2\pi i s_0)$ is an eigenvalue of the monodromy of $f$ at some point near the origin (and, moreover, a nearby tropical monodromy eigenvalue of the polyhedron $N$ in the sense of \cite{E}).
\end{conjecture}
\end{comment}
Both of these statements belong to polyhedral geometry, and together they imply the monodromy conjecture for non-degenerate singularities in arbitrary dimension. When proving Theorem 1.1, we actually prove both parts of this conjecture for $n=4$ in Sections 5 and 8 respectively. 

A key step in proving the first part would be to combinatorially classify faces that can appear in the aforementioned family $\mathcal{F}$. We call them $B$-faces.
\begin{definition}\label{defbfacet0}
1) A lattice simplex in ${\mathbb R}^n$ with the standard coordinate system $v_1,\ldots,v_n$ is called a \emph{$B_1$-simplex} with respect to the $i$-th coordinate if one of its vertices lies in the plane $v_i=1$ and the others in the plane $v_i=0$.

2) A lattice polytope in ${\mathbb R}^n$ is called a \emph{$B$-polytope}, if every lattice simplex that it contains is a $B_1$-simplex.
\end{definition}
For $n=4$, $B$-faces are classified in Lemma \ref{ALEX}: besides $B_1$-pyramids that are well known from the three-dimensional case, we detect another combinatorial type, which we call $B_2$-faces. As soon as the classification is done in any dimension, we believe that the general technique from Section 4 can be extended to arbitrary dimension, though this is still non-trivial, because the fact that the family $\mathcal{F}$ from the conjecture entirely consists of $B$-faces does not assure the existence of the function $b$, see e.g. the phenomenon of $B$-borders (Definition \ref{defborder}) in dimension 4. As to the second part of the conjecture, one step in its proof for $n=4$ is already done for arbitrary dimension, see Section 6.

\vspace{2ex}

The structure of the paper is as follows. In Section \ref{sec:2} we recall the exact statement of the monodromy conjecture and the notion of non-degenerate singularity.

In Section \ref{sec:3}, as a generalization 
of the notion of $B_1$-facets in \cite{L-V}, we introduce 
$B_1$-facets of the Newton polyhedron 
$\Gamma_+(f)$ (Definition \ref{B-F}) and discover so called $B_2$-facets (Definition \ref{B2-F}) that behave similarly, although do not exist in the lower-dimensional setting.

In Section \ref{sec:35}, we show that 
many configurations of $B_1$- and $B_2$-facets alone never assure the existence of 
the corresponding pole of the topological zeta function (Theorem \ref{ADJJJ}). In the course of the proof we introduce an important notion of a critical face of the Newton polyhedron (Definition \ref{defcritn}). Its role in the proof indicates that it might be possible to find a similarly important notion of a critical stratum of the exceptional divisor in the context of arbitrary singularities and their non-toric resolutions.

In Section \ref{sec:4new}, we apply the tools from the preceding two sections to completely classify configurations of facets of four-dimensional Newton polyhedra that never assure the existence of the corresponding pole of the topological zeta function (Theorem \ref{mainB}). Besides the previously found configurations, we discover so called $B^2$-borders (Definition \ref{defborder}). %It would be important for the further research to detect a proper multidimensional generalization of this object.

In Section \ref{sec:4}, following the strategy of 
\cite{L-V}, we prove  
that the candidate poles of $Z_{ {\rm top}, f}(s)$   
contributed by certain non-$B_1$-facets of 
$\Gamma_+(f)$ always yield monodromy eigenvalues. 
Most notably, we obtain 
Theorem \ref{Key-5}. 

Its proof 
relies upon the new notion of a hypermodular function  (Definition \ref{FSM}), which is 
inspired by supermodular functions 
in convex geometry and analysis, and may be of independent interest. 

As a corollary, we can confirm the monodromy conjecture 
of Denef-Loeser for many 
non-degenerate hypersurfaces in 
higher dimensions, see e.g. Theorem \ref{APPT}. 

In Section \ref{sec:6} we prove that singularities adjacent to a 
Newton non-degenerate singularity along a 
coordinate line are themselves Newton non-degenerate (see Proposition \ref{NDPP}). 

However, we notice that starting 
from dimension $4$, not all singularities adjacent to 
Newton non-degenerate singularities are non-degenerate themselves (see Example \ref{exquitnondegenerate}). 
In particular, even in dimension $4$ it is not possible
to prove the Denef-Loeser conjecture for Newton non-degenerate singularities within the framework of non-degenerate singularities.

To this end, the first author introduced in \cite{E}
the notion of tropical nearby monodromy eigenvalues and the corresponding monodromy conjecture, which implies the Denef--Loeser conjecture and (in contrast to the latter) turns into a purely combinatorial statement on the Newton polyhedron for non-degenerate singularities. This tool helps to study monodromy eigenvalues of singularities that are adjacent to a singularity with a given resolution. 

In particular, it allows us to prove in Section \ref{sec:7} the monodromy conjecture for non-degenerate functions of four variables: if the sought monodromy eigenvalue is not a tropical monodromy eigenvalue outside the origin, this imposes lots of restrictions on the combinatorial structure of the Newton polyhedron, and a detailed study of this structure shows that the sought monodromy eigenvalue is a root of the monodromy zeta function at the origin. 

\begin{remark}
Note that the order of this reasoning is opposite to the one usually seen in the literature: first try to find the sought monodromy eigenvalue at the origin, then in the case of a trouble switch to nearby singularities. We proceed in the following order:

1) try to find the sought monodromy eigenvalue outside the origin;

2) notice that we can fail only if the Newton polyhedron has certain combinatorial properties allowing to triangulate it naturally (in a sense);

3) if this occurs, then the resulting natural triangulation allows to find the sought monodromy eigenvalue at the origin.

One could speculate that it is reasonable to expect the same in the general (non-toric) setting: the absence of the sought monodromy eigenvalue outside the origin assures the existence of a certain geometric/deformation-theoretic structure on the resolution space of the singularity, whose combinatorial counterpart is the aforementioned triangulation, and which likewise allows to find the sought monodromy eigenvalue at the origin.
\end{remark}

\medskip \par 
\noindent {\bf Aknowledgements:} The authors 
are very grateful to the University of Nice for 
its hospitality. They also thank
the anonymous referees whose comments
improved this paper substantially.

\section{The monodromy conjecture for the topological 
zeta function}\label{sec:2}

In this section, we recall the monodromy conjecture for 
the topological zeta function and 
related results. 

\subsection{The conjecture}

Let $f: ( \CC^n, 0) \longrightarrow ( \CC, 0)$ 
be a germ of a non-trivial analytic function. We assume that $f$ is 
defined on an open neighborhood $X$ of the origin $0 \in \CC^n$. 
Let $\pi : Y \longrightarrow X$ be an embedded resolution of 
the complex hypersurface $f^{-1}(0) \subset X$ and $E_j$ 
$(j \in J)$ the irreducible components of the 
normal crossing divisor $\pi^{-1}(f^{-1}(0)) \subset Y$. 
For $j \in J$ we denote by $N_j$ (resp. $\nu_j-1$) the 
multiplicity of the divisor associated to $f \circ \pi$ 
(resp. $\pi^* (dx_1 \wedge \cdots \wedge dx_n)$) 
along $E_j \subset Y$. For a non-empty subset $I \subset J$ 
we set 
\begin{equation*}
E_I= \bigcap_{i \in I} E_i, \ \quad \ E_I^{\circ}=E_I \setminus 
\Bigl( \bigcup_{j \notin I}E_j \Bigr). 
\end{equation*}
In \cite{D-L-1} Denef and Loeser defined the local 
topological zeta function $Z_{ {\rm top}, f}(s) 
\in \CC (s)$ associated to $f$ (at the origin) by 
\begin{equation*}
Z_{ {\rm top}, f}(s) = \sum_{I \not= \emptyset} 
\chi ( E_I^{\circ} \cap \pi^{-1}(0) )  \prod_{i \in I} 
\frac{1}{N_i s + \nu_i}, 
\end{equation*}
where $\chi ( \cdot )$ denotes the topological Euler 
characteristic. More precisely, they introduced 
$Z_{ {\rm top}, f}(s)$ by $p$-adic 
integrals and showed that it does not depend 
on the choice of the embedded resolution 
$\pi : Y \longrightarrow X$ by algebraic 
methods. Later in \cite{D-L-2} and 
\cite{D-L-3}, they redefined $Z_{ {\rm top}, f}(s)$ by 
using the motivic zeta function of $f$ and reproved 
this independence of $\pi$ more elegantly. For a 
point $x \in f^{-1}(0) \cap X$ let $F_x \subset X \setminus 
f^{-1}(0)$ be the Milnor fiber of $f$ at $x$ and 
$\Phi_{j,x}: H^j(F_x; \CC ) \simto H^j(F_x; \CC )$ 
$(j \in \ZZ_+ := \{ m \in \ZZ \ | \ m \geq 0 \} )$ 
the Milnor monodromies associated to 
it. Then the monodromy conjecture of Denef-Loeser 
for the local topological 
zeta function $Z_{ {\rm top}, f}(s)$ is stated 
as follows. 

\bigskip \noindent
{\bf Monodromy Conjecture} (Denef-Loeser \cite[Conjecture 3.3.2]{D-L-1}):
Assume that $s_0 \in \CC$ is a pole of $Z_{ {\rm top}, f}(s)$.
Then $\exp (2 \pi i s_0) \in \CC$ is an eigenvalue of
the monodromy $\Phi_{j,x}: H^j(F_x; \CC ) \simto H^j(F_x; \CC )$
for some point $x \in f^{-1}(0) \cap X$
in a neighborhood of the origin $0 \in \CC^n$
and some $j \geq 0$.
\bigskip

\begin{comment}
\bigskip \noindent 
{\bf Monodromy Conjecture:} (Denef-Loeser \cite[Conjecture 3.3.2]{D-L-1}) 
Assume that $s_0 \in \CC$ is a pole of $Z_{ {\rm top}, f}(s)$. 
Then $\exp (2 \pi i s_0) \in \CC$ is an eigenvalue of 
the monodromy $\Phi_{j,x}: H^j(F_x; \CC ) \simto H^j(F_x; \CC )$ 
for some point $x \in f^{-1}(0) \cap X$ 
in a neighborhood of the origin $0 \in \CC^n$ 
and some $j \in \geq 0$. 
\bigskip 
\end{comment}

In \cite{D-L-1} the authors also formulated an even stronger 
conjecture concerning the Bernstein-Sato polynomial $b_f(s)$ 
of $f$. Namely they conjectured that the poles of 
$Z_{ {\rm top}, f}(s)$ are roots of $b_f(s)$. 

From now on, we assume that $f$ is a non-trivial polynomial 
on $\CC^n$ such that $f(0)=0$ and recall the results of 
Denef-Loeser \cite[Section 5]{D-L-1} and Varchenko \cite{V}. 
For $f(x)= \sum_{v \in \ZZ_+^n} c_v x^v \in \CC [ x_1, \ldots, 
x_n]$ ($\ZZ_+ = \{ m \in \ZZ \ | \ m \geq 0 \}$), its support ${\rm supp} f \subset \ZZ_+^n$ is the subset
\begin{equation*}
{\rm supp} f = \{ v \in \ZZ_+^n \ | \ c_v \not= 0 \} 
\subset \ZZ_+^n. 
\end{equation*}
We denote by $\Gamma_+(f) \subset \RR_+^n$ the convex 
hull of $\cup_{v \in {\rm supp} f} (v + \RR_+^n)$ in 
$\RR_+^n$. It is called the Newton polyhedron of $f$ 
at the origin $0 \in \CC^n$. The polynomial $f$ 
such that $f(0)=0$ is called convenient if 
$\Gamma_+(f)$ intersects the positive part of 
any coordinate axis of $\RR^n$. 

\begin{definition}[Kouchnirenko \cite{K}]\label{Non-deg} 
The \emph{polynomial $f$ is non-degenerate 
(at the origin $0 \in \CC^n$)} if for any compact 
face $\tau \prec \Gamma_+(f)$ the 
complex hypersurface 
\begin{equation*}
 \{ x \in ( \CC^*)^n  \ | \ f_{\tau}(x)=0 \} 
\subset ( \CC^*)^n 
\end{equation*} 
in $( \CC^*)^n$ is smooth, where we set 
\begin{equation*}
f_{\tau}(x)= \sum_{v \in \tau \cap \ZZ_+^n} 
c_v x^v \in \CC [ x_1, \ldots, x_n]. 
\end{equation*} 
\end{definition} 
It is well-known that generic polynomials having a 
fixed Newton polyhedron are 
non-degenerate (see for example 
\cite[Chapter V, paragraph 2]{Oka}). 

\subsection{The topological zeta function and  Newton polyhedra}\label{ssj}

In what follows, we assume that the reader is familiar with basic facts and notions of integer lattice geometry, see Appendix at the end of the paper for some digest. For 
$u=(u_1, \ldots, u_n) \in \RR_+^n$ we set 
\begin{equation*}
N( u )= \min_{v \in \Gamma_+(f)} 
\langle u, v \rangle , \quad  
\nu ( u )= | u | = \sum_{i=1}^n u_i 
\end{equation*}
and 
\begin{equation*}
F(u)= \{ v \in \Gamma_+(f) \ | \ \langle u, v \rangle =N(u) \} 
\prec \Gamma_+(f). 
\end{equation*} 
We call $F(u)$ the supporting face of the vector 
$u \in \RR_+^n$ on $\Gamma_+(f)$. 
To a face $\tau \prec \Gamma_+(f)$ one can associate a dual cone 
\begin{equation*}
\tau^{\circ} = 
\overline{  \{ u \in \RR_+^n \ | \ F(u)= \tau \} } 
\subset \RR_+^n. 
\end{equation*} 
Note that $\tau^{\circ}$ is an 
$(n- \dim \tau )$-dimensional rational polyhedral 
convex cone in $\RR_+^n$. The subdivision of $\RR_+^n$ 
by the cones $\tau^{\circ}$ 
($\tau \prec \Gamma_+(f)$) satisfies the axiom 
of fans (for the definition, see \cite{Fulton} and \cite{Oda}) 
and is called the dual fan of $\Gamma_+(f)$. 
Let $\Delta = \RR_+ a(1) + \cdots 
+ \RR_+ a(l)$ $( a(i) \in \ZZ_+^n)$ be a rational 
simplicial cone in $\RR_+^n$, where the 
lattice vectors $a(i) \in \ZZ^n_+$ are 
linearly independent over $\RR$ and primitive. 
Let ${\rm aff} ( \Delta ) \simeq \RR^{l}$ be the 
affine span of $\Delta$ in $\RR^n$ 
and $s( \Delta ) \subset \Delta$ the 
$l$-dimensional lattice simplex whose vertices are 
$a(1), \ldots, a(l)$ and the origin $0 \in \Delta 
\subset \RR_+^n$. We denote by ${\rm mult} ( \Delta ) \in 
\ZZ_{>0}$ the $l$-dimensional normalized 
volume $\Vol_{\ZZ}( s( \Delta ))$ of $s( \Delta )$, 
i.e.  $l!$ times the usual volume of $s( \Delta )$ 
with respect to the affine lattice 
${\rm aff} ( \Delta ) \cap \ZZ^n 
\simeq \ZZ^l$ in ${\rm aff} ( \Delta )$. By using this 
integer ${\rm mult} ( \Delta )$ we set 
\begin{equation*}
J_{\Delta}(s)= 
\frac{{\rm mult} ( \Delta )}{\prod_{i=1}^l 
\{ N(a(i)) s + \nu (a(i)) \} } 
\in \CC (s). 
\end{equation*}
For a face $\tau \prec \Gamma_+(f)$ 
we choose a decomposition $\tau^{\circ} = 
\cup_{1 \leq i \leq r} \Delta_i$ of its dual cone 
$\tau^{\circ}$ into rational simplicial cones 
$\Delta_i$ of dimension $l= \dim \tau^{\circ}$ 
such that $\dim ( \Delta_i \cap \Delta_j ) <l$ 
$(i \not= j)$ and set 
\begin{equation*}
J_{\tau}(s)= \sum_{i=1}^r J_{\Delta_i}(s)
\in \CC (s). 
\end{equation*}
By the following result of Denef-Loeser 
\cite{D-L-1}, this rational 
function $J_{\tau}(s)$ does not depend on the choice 
of the decomposition of $\tau^{\circ}$. Let us set 
\begin{equation*} 
\mbbo = 
 \left(  \begin{array}{c}
      1 \\
      1 \\
      \vdots \\
      1 
    \end{array}  \right) \in \RR_+^n. 
\end{equation*} 

\begin{lemma}[{see 
the proof of \cite[Lemme 5.1.1]{D-L-1}}]\label{DLLE} 
We have an equality
\begin{equation*}
J_{\tau}(s)= \dint_{\tau^{\circ}}
\exp \Bigl( - 
\langle u , P \rangle s - 
\langle  u , \mbbo \rangle 
\Bigr) du, 
\end{equation*}
where $P$ is a point in $\tau$ and 
$du$ is the $l$-dimensional volume form 
on the affine span ${\rm aff} ( \tau^{\circ} ) 
\simeq \RR^{l}$ for which the volume of 
the parallelepiped spanned by a basis of 
the affine lattice 
${\rm aff} ( \tau^{\circ} ) \cap \ZZ^n 
\simeq \ZZ^l$ is equal to $1$. 
\end{lemma}

It is also well-known 
that one can decompose $\tau^{\circ}$ into 
rational simplicial cones without adding new edges. 
Then we have the following formula for 
$Z_{ {\rm top}, f}(s)$. 

\begin{theorem}
[{Denef-Loeser \cite[Th\'eor\`eme 5.3 (ii)]{D-L-1}}]\label{TZF} 
Assume that $f(x) \in \CC [ x_1, \ldots, x_n]$ is 
non-degenerate. Then  
\begin{equation*}
Z_{ {\rm top}, f}(s) = \sum_{\gamma} J_{\gamma}(s) 
+ \frac{s}{s+1} \sum_{\tau} (-1)^{\dim \tau} 
\Vol_{\ZZ}( \tau ) \cdot J_{\tau}(s),  
\end{equation*}
where in the sum $\sum_{\gamma}$ (resp. 
$\sum_{\tau}$) the face $\gamma \prec \Gamma_+(f)$ 
(resp. $\tau \prec \Gamma_+(f)$) ranges through 
the vertices of $\Gamma_+(f)$ (resp. the compact 
ones such that $\dim \tau \geq 1$) and 
$\Vol_{\ZZ}( \tau ) \in \ZZ_{>0}$ is the 
$(\dim \tau )$-dimensional normalized volume 
of $\tau$ with respect to the affine lattice 
${\rm aff} ( \tau ) \cap \ZZ^n 
\simeq \ZZ^{\dim \tau}$ in ${\rm aff} ( \tau )
\simeq \RR^{\dim \tau}$. 
\end{theorem} 

Recall that a face $\tau$ of $\Gamma_+(f)$ 
is called a facet if $\dim \tau =n-1$. 
For a facet $\tau \prec \Gamma_+(f)$ let 
$a( \tau ) =( a( \tau )_1 , \ldots, a( \tau )_n ) 
\in \tau^{\circ} \cap \ZZ_+^n$ be its primitive 
conormal vector and set 
\begin{equation*}
N( \tau )= \min_{v \in \Gamma_+(f)} 
\langle a( \tau ) , v \rangle , 
\end{equation*}
\begin{equation*}
\nu ( \tau )= | a( \tau ) | 
= \sum_{i=1}^n a( \tau )_i = 
\langle a( \tau ) , \mbbo \rangle . 
\end{equation*}
We call $N( \tau )$ the lattice distance 
of $\tau$ from the origin $0 \in \RR^n$. 
It follows from Theorem \ref{TZF} that any pole 
$s_0 \not= -1$ of 
$Z_{ {\rm top}, f}(s)$ is 
contained in the finite set 
\begin{equation*}
\left\{ - \frac{\nu ( \tau )}{N ( \tau )} \ | \  
\tau \prec \Gamma_+(f) \quad \text{is a facet 
not lying in a coordinate hyperplane} 
\right\} \subset \QQ. 
\end{equation*}
Its elements are called candidate poles of 
$Z_{ {\rm top}, f}(s)$. 

If $\tau$ is a simplicial facet, the normalized volume $\Vol_{\ZZ}( \tau )$ is equal to the multiplicity of the cone spanned by the vertices divided by $N(\tau)$. 

\subsection{The monodromy zeta function and  Newton polyhedra}

Finally we recall the result of Varchenko \cite{V}. 
For a polynomial $f(x) \in \CC [ x_1, \ldots, x_n]$ 
such that $f(0)=0$, its monodromy 
zeta function $\zeta_{f,0} (t) \in \CC (t)$ at the origin 
$0 \in \CC^n$ is defined by 
\begin{equation*}
\zeta_{f,0} (t) = \prod_{j \in \ZZ_+} 
\left\{ {\rm det} \Bigl( {\rm id} - t \Phi_{j,0} \Bigr) 
\right\}^{(-1)^j} \in \CC (t). 
\end{equation*}
Similarly one can define also $\zeta_{f,x} (t) \in \CC (t)$ 
for any point $x \in f^{-1}(0)$. Then 
by considering the decomposition of the 
nearby cycle perverse sheaf $\psi_f( \CC_{\CC^n})[n-1]$ 
with respect to the monodromy eigenvalues of $f$ 
and the concentrations of its components 
at generic points $x \in f^{-1}(0)$ 
(see e.g. \cite{Dimca} and 
\cite{H-T-T}), in order to prove the monodromy 
conjecture, it suffices to show that for any 
pole $s_0 \in \CC$ of $Z_{ {\rm top}, f}(s)$ 
the complex number $\exp (2 \pi i s_0) \in \CC$ 
is a root or a pole of $\zeta_{f,x} (t)$ 
for some point $x \in f^{-1}(0)$ 
in a neighborhood of the origin $0 \in \CC^n$ 
(see Denef \cite[Lemma 4.6]{D-1}). 

For a subset $S \subset \{ 1,2, \ldots, n \}$ 
we define a coordinate subspace $\RR^S \simeq 
\RR^{|S|}$ of $\RR^n$ by 
\begin{equation*}
\RR^S = \{ v=(v_1, \ldots, v_n) \in \RR^n \ | \ 
v_i=0 \quad \text{for any} \quad i \notin S \} 
\end{equation*}
and set 
\begin{equation*}
\RR_+^S = \RR^S \cap \RR_+^n \simeq \RR_+^{|S|}. 
\end{equation*}
For a compact face $\tau \prec \Gamma_+(f)$ 
we take the minimal coordinate subspace $\RR^S$ 
of $\RR^n$ containing $\tau$ and set $s_{\tau}= |S|$. 
If $\tau$ satisfies the condition $\dim \tau = 
s_{\tau} -1$ we set 
\begin{equation*}
\zeta_{\tau} (t) = 
\Bigl( 1-t^{N( \tau )} \Bigr)^{\Vol_{\ZZ}( \tau )} 
\in \CC [t],   
\end{equation*}
where $N( \tau ) \in \ZZ_{>0}$ is the lattice distance 
(for the definition, see Appendix) 
of the affine hyperplane ${\rm aff} ( \tau )
\simeq \RR^{\dim \tau}$ in $\RR^S$ from the origin 
$0 \in \RR^S$. Let $a( \tau ) \in \tau^{\circ} \cap 
\ZZ_+^S$ be the primitive conormal vector of 
$\tau \subset \RR^S$ whose value on $\tau$ 
is equal to $N( \tau )>0$. 

\begin{lemma}\label{RESON} 
Let $\tau \prec \Gamma_+(f)$, $S$ and $a( \tau )$ be 
as above and $\alpha \in \QQ$ a rational number. 
For $\beta \in \QQ$ we define a hyperplane 
$L( \beta )$ in $\RR^S$ by 
\begin{equation*}
L( \beta )= \{ v \in \RR^S \ | \ 
\langle a( \tau ) , v \rangle = \beta \cdot N ( \tau ) \}. 
\end{equation*}
Then the complex number $\lambda = \exp (2 \pi i \alpha ) \in \CC$ 
is a root of the polynomial $\zeta_{\tau} (t)$ if and only if 
the hyperplane $L( \alpha ) \subset \RR^S$ 
is rational i.e. $L( \alpha ) \cap \ZZ^S \not= 
\emptyset$. 
\end{lemma} 

\begin{proof} 
Note that $0 \in L(0)$, $\tau \subset L(1) 
={\rm aff} ( \tau )$ and hyperplanes 
$L( \beta )$'s ($\beta \in \QQ$) are parallel to each other. 
The lattice distance $N( \tau )>0$ 
is equal to the number of (mutually parallel) ``rational" 
hyperplanes $L( \beta )$'s 
($\beta \in \QQ$, $0< \beta <1$) between $L(0)$ and $L(1)$ 
plus one. Then the assertion immediately follows 
from this geometric interpretation of $N( \tau )$. 
\end{proof} 

\begin{theorem}\label{TOV} 
(Varchenko \cite{V}) 
Assume that $f(x) \in \CC [ x_1, \ldots, x_n]$ is 
non-degenerate. Then one has 
\begin{equation*}
\zeta_{f,0} (t) = \prod_{\tau} 
\Bigl\{ \zeta_{\tau} (t) \Bigr\}^{(-1)^{\dim \tau}},  
\end{equation*}
where in the product $\prod_{\tau}$ 
the face $\tau \prec \Gamma_+(f)$ ranges through 
the compact ones satisfying the condition $\dim \tau = 
s_{\tau} -1$. 
\end{theorem} 

\begin{definition}\label{V-FA} 
We say that a face $\tau$ of $\Gamma_+(f)$ 
is a \emph{$V$-face} (or a Varchenko face) 
if it is compact and satisfies the condition 
$\dim \tau = s_{\tau} -1$. 
\end{definition} 

\begin{definition}\label{CONTR} 
\begin{enumerate}
\item 
We say that \emph{a candidate pole $s_0 \in \CC$ of 
$Z_{ {\rm top}, f}(s)$ is contributed by a 
facet $\tau \prec \Gamma_+(f)$} or that \emph{$\tau$ contributes $s_0$}
if we have $s_0=-\nu ( \tau )/N ( \tau )$.  
\item
Let $\sigma$ be a $V$-face in 
$\Gamma_+(f)$. We say that \emph{$\sigma$ 
contributes to (the multiplicity of) $t_0 \in \CC$} 
if $t_0$ is a root of 
the polynomial $\zeta_{\sigma}(t)$. 
\end{enumerate}
\end{definition} 

\section{Candidate poles of the topological 
zeta function and $B$-facets}\label{sec:3} 

In this section, we develop new tools 
(Lemmas \ref{magiccancel} and \ref{regintegr}) 
to detect configurations of facets contributing 
fake poles of $Z_{ {\rm top}, f}(s)$ ---
so that once a candidate pole is contributed only by facets from this 
configuration, then it is definitely fake. 

As a first example, 
we define $B_1$-pyramid facets 
(Definition \ref{Pyramid}) of the Newton polyhedron 
$\Gamma_+(f)$. Our definition is a 
straightforward generalization of that of 
Lemahieu-Van Proeyen \cite{L-V}. However, 
starting from dimension
$n=4$, there exist many other combinatorial types of 
facets and configurations that may 
contribute fake poles. In particular, we introduce so 
called $B_2$-facets (Definition \ref{B2-F}) 
and detect some non-contributing configurations of $B_1$ 
and $B_2$-facets, see Propositions \ref{Key-2}, \ref{ADJ} 
and \ref{Key-21}. The proofs of these facts are intended to 
motivate general constructions in the next section, where we 
prove a more general Theorem \ref{ADJJJ}.

From now on, we introduce the following convention on figures 
in this paper: whenever we depict some configuration of cones in 
$\RR^n$, we draw its projectivization, 
resulting in an $(n-1)$-dimensional figure.

\subsection{$B_1$-faces}
 
For a subset $S \subset \{ 1,2, \ldots, n \}$ 
let $\pi_S: \RR_+^n \longrightarrow 
\RR_+^{S^c} \simeq \RR_+^{n-|S|}$ be the 
natural projection. We say that a 
polyhedron 
$\tau$ in $\RR_+^n$ is non-compact for 
$S \subset \{ 1,2, \ldots, n \}$ if the 
Minkowski sum $\tau + \RR_+^S$ is contained 
in $\tau$.

\begin{definition}\label{Pyramid} 
(cf. Lemahieu-Van Proeyen \cite{L-V}) 
Let $\tau$ be a 
polyhedron in $\RR_+^n$. 
\begin{enumerate}
\item 
We say that $\tau$ is a \emph{$B_1$-pyramid 
of compact type for the 
variable $v_i$} if $\tau$ is a compact pyramid 
over the base $\gamma = \tau \cap \{ v_i=0 \}$ 
and its unique vertex $P \prec \tau$ such that 
$P \notin \gamma$ has height one from the 
hyperplane $\{ v_i=0 \} \subset \RR^n_+$. 
\item 
We say that $\tau$ is a 
\emph{$B_1$-pyramid of non-compact 
type} if there exists a non-empty subset 
$S \subset \{ 1,2, \ldots, n \}$ such that 
$\tau$ is non-compact for $S$ and 
$\pi_S ( \tau ) \subset  
\RR_+^{S^c} \simeq \RR_+^{n-|S|}$ is a 
$B_1$-pyramid of compact 
type for some variable $v_i$ 
$(i \notin S)$. 
\item We say that $\tau$ is a \emph{$B_1$-pyramid} if 
it is a $B_1$-pyramid of compact or 
non-compact type. 
\item
We say that a face $\tau$ of the Newton 
polyhedron $\Gamma_+(f)$ is a \emph{$B_1$-face} if 
it is a $B_1$-pyramid. In particular, 
$B_1$-faces 
of dimension $n-1$ and $1$ will be called 
\emph{$B_1$-facets} and \emph{$B_1$-segments} 
respectively.
\end{enumerate}
\end{definition} 

We shall see later in this section that $B_1$-facets alone 
tend not to contribute poles to the topological zeta function.

\begin{remark}
The fact that $B_1$-facets might not give rise to eigenvalues of monodromy 
was already discovered by Loeser 
(see \cite[Remark 6.3]{L-2}). The condition he 
requires on the facets expels among 
others all $B_1$-facets. Let us recall this condition.

For two distinct facets $\tau$ and $\tau^{\prime}$ 
of $\Gamma_+(f)$ let $\beta ( \tau, \tau^{\prime} ) 
\in \ZZ$ be the greatest common divisor of 
the $2 \times 2$ minors of the matrix 
$(a( \tau ), a( \tau^{\prime} )) \in M(n,2; \ZZ )$.
Recall that $a(\tau)\in\tau^\circ\cap\ZZ^n$ is the 
primitive conormal vector of $\tau$, and $\beta ( \tau, \tau^{\prime} ) 
\in \ZZ$ is equal to the lattice area of the 
triangle spanned by $a( \tau )$ and $a( \tau^{\prime} )$.
 
If $N( \tau ) \not= 0$ (e.g. if $\tau$ is 
compact) we set 
\begin{equation*}
\lambda ( \tau, \tau^{\prime} ) = 
\nu ( \tau^{\prime} ) - 
\frac{\nu ( \tau )}{N ( \tau )} 
N( \tau^{\prime} ), \quad 
\mu ( \tau, \tau^{\prime} ) = \frac{ 
\lambda ( \tau, \tau^{\prime} )}{ 
\beta ( \tau, \tau^{\prime} )}
 \quad \in \QQ.  
\end{equation*}
In \cite{L-2} the author considered only 
compact facets $\tau$ of $\Gamma_+(f)$ 
which satisfy the following 
technical condition: 

\medskip \par 
\indent ``For any facet $\tau^{\prime} 
\prec \Gamma_+(f)$ such that $\tau^{\prime} 
\not= \tau$ and $\tau^{\prime} \cap \tau 
\not= \emptyset$ we have 
$\mu ( \tau, \tau^{\prime} ) \notin \ZZ$."

\medskip \par 
\noindent He showed that 
if $f$ is non-degenerate, the candidate pole 
of $Z_{ {\rm top}, f}(s)$ associated to such a compact 
facet $\tau$ is a root of the local 
Bernstein-Sato polynomial of $f$. Now let $\tau \prec \Gamma_+(f)$ be a 
facet containing a 
$B_1$-pyramid of compact type for the 
variable $v_i$ and set 
$\gamma = \tau \cap \{ v_i=0 \}$. 
Let $\tau_0 \prec \Gamma_+(f)$ 
be the unique (non-compact) 
facet such that 
$\gamma \prec \tau_0$, $\tau_0 \not= \tau$ 
and $\tau_0 \subset 
\{ v_i=0 \}$. Then we can easily show that 
$\beta ( \tau, \tau_0 )= 1$. 
Indeed, by a rotation of $\RR^n$ which 
preserves the hyperplane $\{ v_i=0 \} \simeq \RR^{n-1}$ in it, 
we can reduce the problem to the case $n=2$. 
Moreover, by $\nu ( \tau_0 )=1$, $N( \tau_0 )=0$ 
and $\lambda ( \tau, \tau_0 )=1$ we obtain 
$\mu ( \tau, \tau_0 )=1$. Namely 
such a facet $\tau$ does not 
satisfy the above-mentioned condition of \cite{L-2}.
\end{remark} 

The atypical behavior of candidate 
poles of $Z_{ {\rm top}, f}(s)$ 
associated to $B_1$-facets 
essentially arises from the 
following simple computation 
(cf. Lemma \ref{DLLE}). For a subcone $C$ of the dual cone $\tau^\circ$ to a 
$k$-dimensional face $\tau$ of the Newton polyhedron $\Gamma_+(f)$, define 
the contribution of $C$ to the topological $\zeta$-function $Z_{\rm top, f}$ 
as $$\int_{C}\exp\bigl(-N(u) s- \langle u,\mbbo\rangle\bigr)\, du^{(n)}$$
for $k=0$ (see Lemma \ref{DLLE} for the details) and otherwise 
$$(-1)^{k}\Vol_\ZZ(\tau)\frac{s}{s+1} \int_{C}\exp\bigl(-N(u) 
s- \langle u, \mbbo \rangle\bigr)\, du^{(n-k)},$$
where $N(\cdot)$ is the support function of the Newton polyhedron, 
and $du^{(m)}$ is the $m$-dimensional lattice volume form. This 
definition is chosen so that the topological $\zeta$-function of $f$ 
equals the sum of contributions 
of the dual cones to all bounded faces of $\Gamma_+(f)$.

\begin{lemma}\label{magiccancel} 
Assume that a $B_1$-face 
$\tau$ of $\Gamma_+(f)$ is the convex hull 
of its base $\gamma$ 
in the coordinate hyperplane $\{ v_n=0 \}$ and its 
apex $P=(*,\ldots,*, 1)$. 
Furthermore, assume that $C \subset \gamma^\circ$ 
is the convex hull of a rational polyhedral 
subcone $C'\subset \tau^\circ$ and 
the $n$-th coordinate axis 
$O_n= \RR_+ (0, \ldots, 0, 1) \subset \RR^n_+$. 
Then the sum of the contributions from the cones 
$C \subset \gamma^\circ$ and $C'\subset \tau^\circ$ 
to $Z_{ {\rm top}, f}(s)$ is equal to 
\begin{equation*} 
\dint_{C}
\exp \Bigl( - 
\langle u , P \rangle s - 
\langle  u , \mbbo \rangle 
\Bigr) du_1 \cdots du_n
\end{equation*}
if $\tau$ is a 
$B_1$-segment, and is $0$ otherwise.
\end{lemma}

\begin{proof} 
In the second case, the contributions 
from $C$ and $C'$ are 
equal up to sign and cancel each other. 
Indeed, if we decompose $C'$ into simplicial 
cones $\Delta_i^{\prime} \subset C'$ and 
take the convex hulls 
$\Delta_i \subset C$ of them 
and the coordinate axis 
$O_n=\RR_+ \cdot (0, \ldots, 0, 1)$, then 
by using the condition $P=(*,\ldots,*, 1)$ 
we can easily show that 
${\rm mult} ( \Delta_i )={\rm mult} ( \Delta_i^{\prime} )$. 
In the first case, we may assume 
that $C'$ is simplicial and 
${\rm mult} ( C )={\rm mult} ( C^{\prime} )$.  
Let $a(i) \in C' \cap \ZZ_+^n$ 
$(1 \leq i \leq n-1)$ be the primitive 
vectors on the edges of the $(n-1)$-dimensional 
cone $C'$. Then 
the sum of the contributions is equal to 
\begin{align*}
\frac{{\rm mult} ( C )}{\prod_{i=1}^{n-1} 
\{ N(a(i))s+\nu(a(i)) \}}  &  - 
\frac{s}{s+1} \frac{{\rm mult} ( C )}{\prod_{i=1}^{n-1} \{ N(a(i))s+
\nu(a(i)) \} }
\\
  = & 
\frac{{\rm mult} ( C )}{ \{ \langle e_n ,P \rangle s+ 
\langle e_n, \mbbo \rangle \}
\cdot \prod_{i=1}^{n-1} 
\{  \langle a(i),P \rangle s+ \langle a(i),  
\mbbo \rangle \} }, 
\end{align*}
where $e_n=(0,\ldots,0, 1)$. 
The right hand side is equal to the sought 
integral by the proof of Lemma \ref{DLLE}.
\end{proof}

\subsection{Critical edges}

We now introduce our main tool to prove that a given 
number is not a pole of the topological zeta function.

\begin{lemma}\label{nontri} 
Assume that $s_0 \ne-1$. Then for the points $P$ in 
the summit of a $B_1$-facet of $\Gamma_+(f)$ the equation 
$\langle u,P \rangle s_0+ \langle u,\mbbo \rangle =0$ 
is non-trivial. Namely it defines a 
hyperplane $L_P$ in $\RR^n$. 
\end{lemma}
{\it Proof.} 
If it is trivial, then we have $P=- \mbbo /s_0$. 
Since $s_0 \ne -1$, none of the coordinates of $P$ is 
equal to 1, so it is not in the summit. $\hfill \quad\Box$

\begin{definition}\label{newdefi} 
We say that a closed set $C\subset \RR^n$ is an 
\emph{$n$-dimensional 
polyhedral cone} if it is a union of 
finitely many $n$-dimensional 
closed convex polyhedral cones. 
A ray $R$ on the boundary $\partial C$ of 
an $n$-dimensional polyhedral cone 
$C$ in $\RR^n$ is called an \emph{edge of $C$}, 
if, in an arbitrarily small neighborhood 
of a point of $\relint R= R \setminus \{ 0 \}$, 
the cone $C$ is not affinely isomorphic 
to a product $\RR^2\times C'$ for 
some subset $C'\subset\RR^{n-2}$. Moreover 
for $s_0 \in \RR$ and a point $P \in \RR^n$ 
we say that \emph{a ray $R$ in $\RR^n$ 
is critical 
with respect to the pair $(s_0,P)$} 
if for its generator $u \in R$ we have 
$\langle u,P \rangle s_0+ \langle u,\mbbo \rangle =0$, 
i.e. $u \perp (s_0 P+ \mbbo )$.
\end{definition}

\begin{lemma}\label{regintegr}
Let $C \subset \RR^n$ be an $n$-dimensional  
polyhedral cone in $\RR^n$. 
Assume that for $s_0 \in \RR$ and a point $P \in \RR^n$ 
no edge of $C$ is critical 
with respect to $(s_0,P)$. 
Then the integral 
\begin{equation}
\dint_{C}
\exp \Bigl( - 
\langle u , P \rangle s - 
\langle  u , \mbbo \rangle 
\Bigr) du_1 \cdots du_n
\end{equation}
is a rational function 
of $s$ holomorphic at $s=s_0 \in \CC$. 
\end{lemma}

\begin{proof} Subdivide $C$ into simplicial cones. 
This subdivision is combinatorially stable under small 
perturbation of each ray $R$ within its ambient face 
of $C$. Since no edge in 
this ambient face is critical with respect to 
$(s_0,P)$, the same is true also for almost all 
rays in the face. We can thus perturb the rays in the 
subdivision in their ambient 
faces so that all of them become 
non-critical with respect to 
$(s_0,P)$. Then by integrating $\exp ( - 
\langle u , P \rangle s - 
\langle  u , \mbbo \rangle )$ 
over each of the simplicial cones in the resulting subdivision, 
we obtain a rational function holomorphic at $s=s_0 \in \CC$ 
(see Lemma \ref{DLLE}). 
\end{proof}

\subsection{Some non-contributing configurations of $B_1$-facets}

We first show that a candidate pole 
contributed by a unique facet is always 
fake, once this facet is $B_1$. Then we 
discuss what happens in other cases 
(when the same candidate pole is contributed 
by several $B_1$-facets or by a non-$B_1$ facet).

The following result is not used in the 
sequel and, on the contrary, is a special 
case of the subsequent Theorem \ref{ADJJJ}. 
Nevertheless, we prefer to give it an 
independent proof, keeping things as simple 
and explicit as possible. This proof is a 
good illustration of a more general 
construction (of so called sprouts) 
used later on to prove Theorem \ref{mainB} 
leading to the monodromy conjecture in dimension 4.

\begin{proposition}\label{Key-2} 
Assume that $f$ is non-degenerate and let 
$\tau \prec \Gamma_+(f)$ 
be a $B_1$-facet.  Assume also 
that the candidate pole 
\begin{equation*}
s_0= - \frac{\nu 
( \tau )}{N ( \tau )} \not= -1  
\end{equation*}
of $Z_{ {\rm top}, f}(s)$ is 
contributed only by 
$\tau$. Then $s_0$ is fake, i.e. not an actual 
pole of $Z_{ {\rm top}, f}(s)$. 
\end{proposition} 

\begin{proof} 
Since the proof for 
$B_1$-pyramids of non-compact 
type is similar, we prove the assertion only 
for $B_1$-pyramids of 
compact type. Without loss of generality 
we may assume that 
$\tau$ is a compact pyramid 
over the base $\gamma = 
\tau \cap \{ v_n=0 \}$ 
and its unique vertex 
$P \prec \tau$ such that 
$P \notin \gamma$ has height one from the 
hyperplane $\{ v_n=0 \} \subset \RR^n$. 
Let $A_1, A_2, \ldots, A_m$ $(m \geq n-1)$ be the 
vertices of the $(n-2)$-dimensional polytope 
$\gamma$. For $1 \leq i \leq m$ 
we denote the dual cone $A_i^{\circ}$ of 
$A_i \prec \Gamma_+(f)$ by $C_{A_i}$. 
Similarly we set $C_P=P^{\circ}$. 
Let $a ( \tau ) \in \ZZ_+^n$ be the 
primitive vector on the ray $\tau^{\circ}$. 
Then we have 
\begin{equation*}
a( \tau ) \in {\rm Int} 
(C_P \cup C_{A_1} \cup 
\cdots \cup C_{A_m}). 
\end{equation*}
Note that $C_P \cup C_{A_1} \cup 
\cdots \cup C_{A_m}$ is an $n$-dimensional 
polyhedral cone in the sense of 
Definition \ref{newdefi}. 
In order to construct a nice $n$-dimensional 
polyhedral 
subcone $\square$ of $C_P \cup C_{A_1} \cup 
\cdots \cup C_{A_m}$ such that 
$a( \tau ) \in {\rm Int}  \square$, 
we shall introduce a new dummy vector 
$b \in {\rm Int} C_P \cap \ZZ_+^n$ satisfying 
the condition 
\begin{equation}\label{condb} 
- \frac{\nu (b)}{N(b)} \not= s_0 
\end{equation}
in the following way. First, by our 
assumption $s_0 \not= -1$ and 
Lemma \ref{nontri}, for the summit 
$P=(*,*, \ldots, *,1) 
\in \ZZ_+^n$ of the $B_1$-pyramid 
$\tau$ the equation 
$\langle u,P \rangle s_0+ \langle u,\mbbo \rangle =0$ 
is non-trivial. It thus defines 
a hyperplane $L_P$ in $\RR^n$. 
Then by taking a primitive 
vector $b \in {\rm Int} 
C_P \cap \ZZ_+^n$ such that 
$b \notin L_P$ we get the desired condition 
$N(b)s_0+ \nu (b) \not= 0$. 
Let $\tau_0 \prec \Gamma_+(f)$ 
be the unique facet such that 
$\gamma \prec \tau_0$, 
$\tau_0 \not= \tau$ and $\tau_0 \subset 
\{ v_n=0 \}$. Then 
the primitive vector $a ( \tau_0 ) \in \ZZ_+^n$ 
on the ray $\tau_0^{\circ}$ is 
given by $a ( \tau_0 )=(0,0,\ldots, 0,1)$. 
For $1 \leq i \leq m$ let $\sigma_i \prec \tau$ 
be the edge of $\tau$ connecting the two 
points $P$ and $A_i$ and $F_i \prec C_P$ 
the corresponding facet of the cone $C_P$ containing 
$\tau^{\circ} = \RR_+ a( \tau ) \prec C_P$. 
All the facets of $C_P$ containing 
$\tau^{\circ}$ are obtained in this way. 
Since the point $A_i \prec \sigma_i$ is a 
vertex of $\tau_0$, its dual cone 
$C_{A_i}=A_i^{\circ}$ contains not only 
$F_i$ but also the ray 
$\tau_0^{\circ} = \RR_+ a( \tau_0 )$. 
For $1 \leq i \leq m$ set 
\begin{equation*} 
F_i^{\sharp}= \RR_+ a( \tau_0 )+F_i, \quad 
\quad F_i^{\flat}= \RR_+ b +F_i.  
\end{equation*}
In Figure \ref{picproof1} below we presented the transversal 
hyperplane sections of the cones $F_i$, 
$F_i^{\sharp}$ and $F_i^{\flat}$.

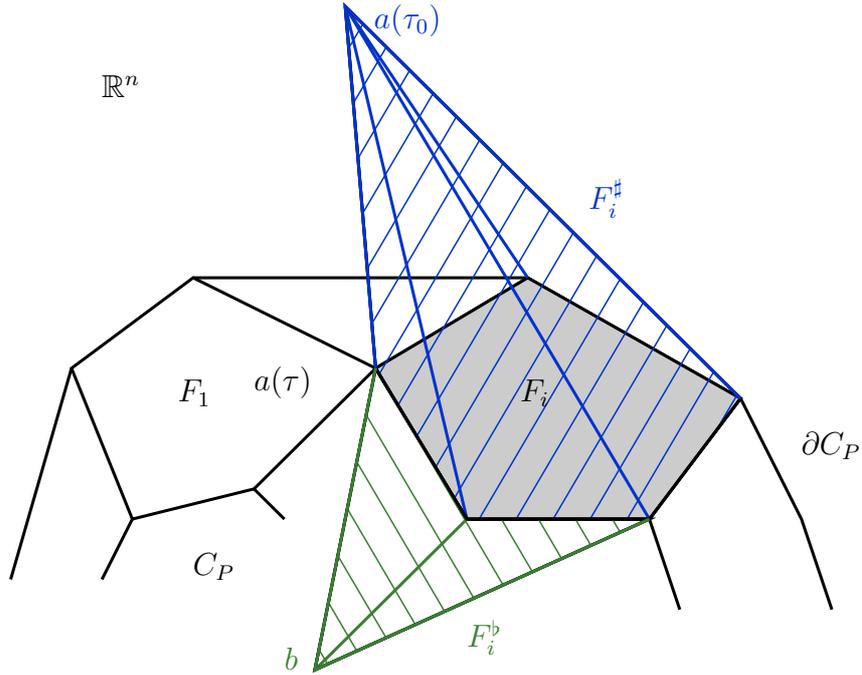
\begin{figure}[h]
\begin{center}
\skipfig{\psscalebox{0.5 0.5} 
\Large{
\begin{pspicture}(0,-4.7921977)(11.029225,4.7921977)
\definecolor{colour0}{rgb}{0.8,0.8,0.8}
\definecolor{colour2}{rgb}{0.0,0.2,0.8}
\psline[linecolor=black, linewidth=0.04](6.819226,0.4421979)(4.819226,-0.7578021)(6.0192256,-2.757802)(8.419226,-2.757802)(9.6192255,-1.1578021)(6.819226,0.4421979)
\psline[linecolor=black, linewidth=0.04](9.6192255,-1.1578021)(10.419226,-2.757802)(10.819225,-3.957802)(10.819225,-3.957802)
\psline[linecolor=black, linewidth=0.04](8.419226,-2.757802)(8.819225,-3.957802)(8.819225,-3.957802)
\psline[linecolor=black, linewidth=0.04](4.819226,-0.7578021)(3.2192256,-2.3578022)(1.6192256,-2.757802)(0.8192256,-0.7578021)(2.4192257,0.4421979)(4.819226,-0.7578021)(4.819226,-0.7578021)
\psline[linecolor=black, linewidth=0.04](2.4192257,0.4421979)(6.819226,0.4421979)
\psline[linecolor=black, linewidth=0.04](0.8192256,-0.7578021)(0.019225609,-3.5578022)
\psline[linecolor=black, linewidth=0.04](3.2192256,-2.3578022)(3.6192255,-2.757802)
\psline[linecolor=black, linewidth=0.04](1.6192256,-2.757802)(1.2192256,-3.5578022)
\pspolygon[linecolor=black, linewidth=0.04, fillstyle=solid,fillcolor=colour0](4.819226,-0.7578021)(6.0192256,-2.757802)(8.419226,-2.757802)(9.6192255,-1.1578021)(6.819226,0.4421979)
\psline[linecolor=colour2, linewidth=0.04](8.419226,-2.757802)(4.4192257,4.0421977)(4.4192257,4.0421977)
\psline[linecolor=colour2, linewidth=0.04](4.819226,-0.7578021)(4.4192257,4.0421977)
\psline[linecolor=colour2, linewidth=0.04](6.0192256,-2.757802)(4.4192257,4.0421977)
\psline[linecolor=colour2, linewidth=0.04](9.6192255,-1.1578021)(4.4192257,4.0421977)
\psline[linecolor=colour2, linewidth=0.04](6.819226,0.4421979)(4.4192257,4.0421977)
\pspolygon[linecolor=black, linewidth=0.04, fillstyle=vlines, hatchwidth=0.02, hatchangle=330, hatchsep=0.4, hatchcolor=colour2](4.4192257,4.0421977)(4.819226,-0.7578021)(6.0192256,-2.757802)(8.419226,-2.757802)(9.6192255,-1.1578021)
\psline[linecolor=black, linewidth=0.04](4.819226,-0.7578021)(4.0192256,-4.757802)(8.419226,-2.757802)
\psline[linecolor=OliveGreen, linewidth=0.04](4.0192256,-4.757802)(6.0192256,-2.757802)
\pspolygon[linecolor=black, linewidth=0.04, fillstyle=vlines, hatchwidth=0.02, hatchangle=30, hatchsep=0.4, hatchcolor=OliveGreen](4.819226,-0.7578021)(4.0192256,-4.757802)(8.419226,-2.757802)(6.0192256,-2.757802)
\rput[bl](1.2192256,2.842198){${\mathbb R}^n$}
\rput[bl](4.7992257,3.642198){\textcolor{colour2}{$a(\tau_0)$}}
\rput[bl](6.719226,-1.2578021){$F_i$}
\rput[bl](3.219226,-1.1578021){$a(\tau)$}
\psline[linecolor=colour2, linewidth=0.04](4.819226,-0.7578021)(4.4192257,4.0421977)(9.6192255,-1.1578021)
\rput[bl](7.6192255,1.2421979){\textcolor{colour2}{$F_i^\sharp$}}
\psline[linecolor=OliveGreen, linewidth=0.04](4.819226,-0.7578021)(4.0192256,-4.757802)(8.419226,-2.757802)
\rput[bl](6.0192256,-4.557802){\textcolor{OliveGreen}{$F_i^\flat$}}
\rput[bl](2.2192257,-1.2578021){$F_1$}
\rput[bl](2.4192257,-3.5578022){$C_P$}
\rput[bl](10.419226,-1.9578022){$\partial C_P$}
\rput[bl](3.6192256,-4.757802){\textcolor{OliveGreen}{$b$}}
\end{pspicture}}}
\end{center}
\caption{The proof of Proposition \ref{Key-2}}\label{picproof1}

\end{figure}

Then by our construction, $\square = 
\cup_{i=1}^m (F_i^{\sharp} \cup F_i^{\flat})$ 
is an $n$-dimensional polyhedral cone in $\RR^n$ 
and satisfies the desired condition 
$a( \tau ) \in {\rm Int}  \square$. 

By Lemma \ref{magiccancel} and the argument 
in Case $1$ of the proof 
of \cite[Proposition 14]{L-V}, 
the contribution to $Z_{ {\rm top}, f}$ 
from the dual cones $C_P, C_{A_1}, \ldots, 
C_{A_m}$, $F_1, \ldots, F_m$ of 
$P, A_1, \ldots, A_m$, $PA_1, \ldots, PA_m$ 
respectively is equal to $Z_{ {\rm top}, f}$ 
itself modulo holomorphic functions at 
$s=s_0 \in \CC$. 
We thus obtain an equality 
\begin{equation*}
Z_{ {\rm top}, f} (s) \equiv  
\dint_{\square}
\exp \Bigl( - 
\langle u , P \rangle s - 
\langle  u , \mbbo \rangle 
\Bigr) du_1 \cdots du_n 
\end{equation*}
modulo holomorphic functions at 
$s=s_0 \in \CC$. 
By our assumption and 
the condition \eqref{condb} no edge of the 
polyhedral cone 
$\square$ is critical with respect to 
$(s_0, P)$ in the sense of 
Definition \ref{newdefi}. 
Then by Lemma \ref{regintegr} 
the rational function
\begin{equation*}
\dint_{\square}
\exp \Bigl( - 
\langle u , P \rangle s - 
\langle  u , \mbbo  \rangle 
\Bigr) du_1 \cdots du_n
\end{equation*}
of $s$ is holomorphic at $s=s_0 \in \CC$. 
This implies that also $Z_{ {\rm top}, f} (s)$ 
is holomorphic there. 
\end{proof} 

We now discuss what happens when the same 
candidate pole is contributed by several $B_1$-facets.

If two $B_1$-facets for different variables 
are adjacent (i.\ e.\ have a common 
$(n-2)$-dimensional face) and contribute 
the same candidate pole, this may happen 
to be an actual pole of the topological 
$\zeta$-function. This is always so for 
$n=3$ (see \cite{L-V}), but may fail 
starting from $n=4$ (see $B^2$-borders in Theorem \ref{mainB}). 

On the contrary, two adjacent $B_1$-faces for 
the same variable cannot alone yield  
an actual pole (similarly to \cite[Proposition 14]{L-V} for $n=3$). 
The following result is a higher-dimensional 
analogue of the one in 
the proof of \cite[Proposition 14]{L-V}. 

\begin{proposition}\label{ADJ} 
Assume that $f$ is non-degenerate and let 
$\tau_1, \ldots, \tau_k \prec \Gamma_+(f)$ 
be $B_1$-facets such that 
\begin{equation*}
s_0:=   - \frac{\nu ( \tau_1 )}{N ( \tau_1 )} 
= \cdots \cdots = 
- \frac{\nu ( \tau_k )}{N ( \tau_k )} \not= -1 
\end{equation*}
and their common candidate pole $s_0 \in \QQ$ 
of $Z_{ {\rm top}, f}(s)$ is 
contributed only by them. Assume also that 
if $\tau_i$ and $\tau_j$ ($i \not= j$) 
have a common facet they are 
$B_1$-pyramids for the same variable. 
Then $s_0$ is fake i.e. not an actual 
pole of $Z_{ {\rm top}, f}(s)$. 
\end{proposition} 

\begin{proof} 
If $\tau_i$ and $\tau_j$ ($i \not= j$) do 
not have a common facet, then by the proof 
of Proposition \ref{Key-2} after a suitable 
subdivision of the dual fan of $\Gamma_+(f)$ 
into rational simplicial cones we can 
calculate their contributions to 
$Z_{ {\rm top}, f}(s)$ separately. So 
we may assume that the $B_1$-facets 
$\tau_1, \ldots, \tau_k$ 
have the common summit $P \in \tau_1 \cap 
\cdots \cap \tau_k$. For the sake of 
simplicity, here we shall treat only the case 
where $k=2$, $\tau_1$ (resp. $\tau_2$) 
is a compact $B_1$-pyramid over the base $\gamma_1 = 
\tau_1 \cap \{ v_n=0 \}$ 
(resp. $\gamma_2 = 
\tau_2 \cap \{ v_n=0 \}$) and 
$\tau_1 \cap \tau_2$ is the (unique) 
common facet of $\tau_1$ and $\tau_2$. 
The proofs for the other cases are similar. 
Let $\tau_0 \prec \Gamma_+(f)$ 
be the unique facet of $\Gamma_+(f)$ such that 
$\gamma_1, \gamma_2 \prec \tau_0$, 
$\tau_0 \not= \tau_i$ ($i=1,2$) 
and $\tau_0 \subset 
\{ v_n=0 \}$. We denote by $P$ the common 
summit of $\tau_1$ and $\tau_2$. Since 
$\tau_1 \cap \tau_2$ is a 
common facet of $\tau_1$ and $\tau_2$, 
there exists a $2$-dimensional face 
of the dual cone $C_P$ 
of $P \prec \Gamma_+(f)$ containing both 
$\tau_1^{\circ} =\RR_+ a( \tau_1 )$ and 
$\tau_2^{\circ} =\RR_+ a( \tau_2 )$. 
As in the proof of Proposition \ref{Key-2}, 
let $F_1,F_2, \ldots, F_m$ be 
the facets of $C_P$ containing 
the ray $\tau_1^{\circ}$ or 
$\tau_2^{\circ}$ 
and subdivide $F_1 \cup \cdots \cup F_m \subset 
\partial C_P$ into rational simplicial cones 
without adding new edges. 
Let $\Delta_1, \ldots, \Delta_r$ be the 
$(n-1)$-dimensional simplicial cones thus obtained 
in $F_1 \cup \cdots \cup F_m \subset 
\partial C_P$ and containing 
$\tau_1^{\circ}$ or $\tau_2^{\circ}$. 
As in the proof of Proposition \ref{Key-2}, 
we take a new primitive vector 
$b \in {\rm Int} C_P \cap \ZZ_+^n$ such that 
\begin{equation}\label{condnb} 
- \frac{\nu (b)}{N(b)} \not= s_0. 
\end{equation}
For $1 \leq i \leq r$ set 
\begin{equation*} 
\Delta_i^{\sharp}= 
\RR_+ a( \tau_0 )+ \Delta_i, \quad 
\quad \Delta_i^{\flat}= \RR_+ b + \Delta_i.  
\end{equation*}
Then $\square := \cup_{i=1}^r 
( \Delta_i^{\sharp} \cup \Delta_i^{\flat})$ 
is an $n$-dimensional polyhedral cone in $\RR^n$ 
such that  
\begin{equation*} 
a( \tau_1 ), a( \tau_2 ) \in 
{\rm Int} \square .   
\end{equation*}
By Lemma \ref{magiccancel} 
(or the argument 
in Case $1$ of the proof 
of \cite[Proposition 14]{L-V}) 
we obtain an equality 
\begin{equation*}
Z_{ {\rm top}, f} (s) \equiv  
\dint_{\square}
\exp \Bigl( - 
\langle u , P \rangle s - 
\langle  u , \mbbo \rangle 
\Bigr) du_1 \cdots du_n
\end{equation*}
modulo holomorphic functions at 
$s=s_0 \in \CC$. By our assumption and 
the condition \eqref{condnb} no edge of the 
polyhedral cone 
$\square$ is critical with respect to 
$(s_0, P)$ in the sense of 
Definition \ref{newdefi}. 
Then by Lemma \ref{regintegr} 
the rational function
\begin{equation*}
\dint_{\square}
\exp \Bigl( - 
\langle u , P \rangle s - 
\langle  u , \mbbo \rangle 
\Bigr) du_1 \cdots du_n
\end{equation*}
of $s$ is holomorphic at $s=s_0 \in \CC$. 
\end{proof} 

\subsection{$B_2$-facets}

We now discuss what happens when a 
candidate pole is contributed by a 
non-$B_1$-facet. Such pole $s_0$ is always 
allowed not to be fake for $n=3$, 
because $\exp(2\pi i s_0)$ is always 
a nearby monodromy eigenvalue (see \cite{L-V}). 
However, this is not the case starting 
from $n=4$ for some non-$B_1$ facets. In 
this subsection, we introduce one such example.

\begin{definition}\label{B2-F} 
For all 
$n \geq 4$ we define \emph{$B_2$-facets} 
$\tau \prec \Gamma_+(f)$ to be non-$B_1$ 
compact facets whose projection to a certain 
$(n-2)$-dimensional coordinate plane 
coincides with the standard 
$(n-2)$-dimensional simplex. 

In particular, for $n=4$, a facet $\tau$ of $\Gamma_+(f)$  
is a $B_2$-facet if and only if, up to reordering 
the coordinates, it has the vertices 
$A,B,P,Q,X,Y$ of the form 
\begin{equation*}
\begin{cases}
A= & (1,0,*,*) \\
B= & (1,0,*,*) \\
P= & (0,1,*,*) \\
Q= & (0,1,*,*) \\
X= & (0,0, * ,*) \\
Y= & (0,0, * ,*) 
\end{cases}
\end{equation*}
as in Figure \ref{picb2} below (it can be degenerated 
so that $X=Y$). 
\end{definition}

\begin{figure}[h]
\begin{center}
\skipfig{
\psscalebox{0.7 0.7} 
{ \Large
\begin{pspicture}(0,-3.1)(13.068935,3.1)
\psline[linecolor=black, linewidth=0.04, arrowsize=0.05291666666666667cm 4.0,arrowlength=3.0,arrowinset=0.0]{->}(2.18831,-0.7)(13.3883095,-0.7)
\psline[linecolor=black, linewidth=0.04](3.78831,-0.7)(4.98831,1.3)(5.38831,-1.9)(3.78831,-0.7)
\psline[linecolor=black, linewidth=0.04](10.98831,-0.7)(9.3883095,1.3)(8.98831,-1.9)(10.98831,-0.7)
\psline[linecolor=black, linewidth=0.04](5.38831,-1.9)(8.98831,-1.9)
\psline[linecolor=black, linewidth=0.04](4.98831,1.3)(9.3883095,1.3)
\psline[linecolor=black, linewidth=0.04, arrowsize=0.05291666666666667cm 4.0,arrowlength=3.0,arrowinset=0.0]{->}(2.18831,-0.7)(2.18831,3.3)
\psline[linecolor=black, linewidth=0.04, arrowsize=0.05291666666666667cm 4.0,arrowlength=3.0,arrowinset=0.0]{->}(2.18831,-0.7)(-0.21169007,-3.1)
\psline[linecolor=black, linewidth=0.04, linestyle=dashed, dash=0.17638889cm 0.10583334cm](4.98831,1.3)(2.18831,1.3)
\psline[linecolor=black, linewidth=0.04, linestyle=dashed, dash=0.17638889cm 0.10583334cm](5.38831,-1.9)(0.9883099,-1.9)
\rput[bl](4.38831,-0.3){$\sigma_1$}
\rput[bl](7.38831,0.5){$\sigma_2$}
\rput[bl](9.3883095,-0.3){$\sigma_3$}
\rput[bl](6.18831,-0.3){$\tau$}
\rput[bl](4.98831,-2.7){$A$}
\rput[bl](8.98831,-2.7){$B$}
\rput[bl](10.98831,-1.5){$Y$}
\rput[bl](3.38831,-1.5){$X$}
\rput[bl](4.98831,1.7){$P$}
\rput[bl](9.3883095,1.7){$Q$}
\rput[bl](1.7883099,1.3){$1$}
\rput[bl](0.58830994,-1.9){$1$}
\rput[bl](2.58831,2.9){$x_2$}
\rput[bl](0.58830994,-3.1){$x_1$}
\rput[bl](12.18831,-0.3){$x_3,\, x_4$}
\rput[bl](0.18830994,1.7){${\mathbb R}^4$}
\end{pspicture}
}} \nopagebreak

\caption{a $B_2$-facet in dimension 4} \label{picb2}

\end{center}
\end{figure}
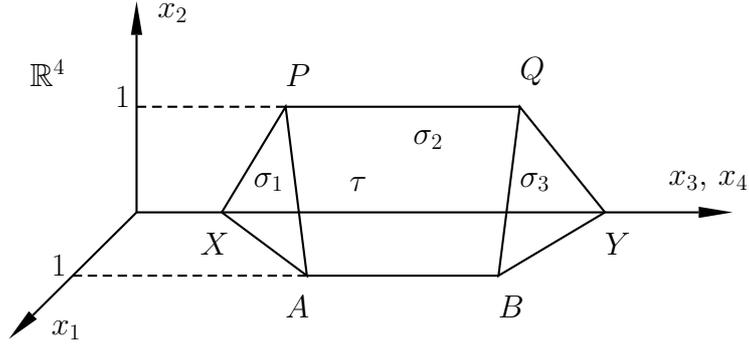

Note that this facet $\tau$ 
splits into two $B_1$-pyramids for different 
variables whose intersection does not contain 
any $1$-dimensional $V$-face. For example 
we have the decomposition 
$\tau =AXPQ \cup QXABY$. 

\begin{definition}\label{B-F}
A facet $\tau$ of $\Gamma_+(f)$  
is called a \emph{$B$-facet} if it is a $B_1$-facet or a 
$B_2$-facet.
\end{definition}

\begin{proposition}\label{Key-21} 
In the case $n \geq 4$ assume that $f$ 
is non-degenerate and let 
$\tau \prec \Gamma_+(f)$ be a $B_2$-facet. 
Assume also that the candidate pole 
\begin{equation*}
s_0= - \frac{\nu ( \tau )}{N ( \tau )} \not= -1  
\end{equation*}
of $Z_{ {\rm top}, f}(s)$ is contributed only by 
$\tau$. Then $s_0$ is fake. 
\end{proposition} 

\begin{proof} 
We prove the assertion only for $n=4$. The proof 
for the general case $n \geq 4$ is similar. 
In the notation of Figure \ref{picb2}, we 
define facets $\sigma_1, \sigma_2, \sigma_3$ 
of $\tau$ by $\sigma_1=XAP$, $\sigma_2=
PABQ$, $\sigma_3=YQB$ respectively. 
As in the proof of \cite[Proposition 14]{L-V} 
let $\tau_i \prec \Gamma_+(f)$ $(1 \leq i \leq 3)$ 
be the unique facet such that 
$\sigma_i \prec \tau_i$ 
and $\tau_i \not= \tau$. Moreover for $i=1,2$ 
let $\rho_i \prec \Gamma_+(f)$ 
be the unique facet such that 
$\rho_i \subset \{ v_i=0 \} \simeq \RR^3$, 
$\tau \cap \{ v_i=0 \}  \prec \rho_i$ 
and $\rho_i \not= \tau$.
Since the dimension of the $B_2$-facet 
$\tau$ is 3, the three segments $PQ$, $AB$ and $XY$ are 
parallel. This implies that their dual cones are on the 
same hyperplane $H \simeq \RR^3$ in $\RR^4$. 
By the four facets $\tau$, $\tau_1$, 
$\tau_2$, $\rho_1$ containing the vertex $P$ 
we define a $4$-dimensional simplicial 
cone $\Delta_P \subset \RR^4$ by 
\begin{equation*}
\Delta_P= \RR_+ a( \tau )+ \RR_+ a( \tau_1 )+
\RR_+ a( \tau_2 )+ \RR_+ a( \rho_1 ). 
\end{equation*}
Similarly we define $4$-dimensional simplicial 
cones $\Delta_A, \Delta_X, \Delta_Q, 
\Delta_B, \Delta_Y \subset \RR^4$ for the 
vertices $A,X,Q,B,Y$ and set 
\begin{equation*}
\square_P= \Delta_P \cup \Delta_A \cup \Delta_X, 
\ \ 
\square_Q= \Delta_Q \cup \Delta_B \cup \Delta_Y. 
\end{equation*}
Then by the above-mentioned property of $H$, 
$\square_P \cap H$ (resp. $\square_Q \cap H$) 
is a facet of $\square_P$ (resp. $\square_Q$) 
and $\square_P \cap H = \square_Q \cap H
= \square_P \cap \square_Q$. 
The dual ray $\tau^{\circ}$ of $\tau$ is 
contained in $\square_P \cap H = \square_Q \cap H
= \square_P \cap \square_Q$, 
but it is an edge of neither $\square_P$ nor $\square_Q$. 
Moreover $\square := \square_P \cup \square_Q$ 
is a $4$-dimensional polyhedral cone such that 
$a( \tau ) \in {\rm Int}  \square$. 
By Lemmas \ref{DLLE} and \ref{magiccancel} and the argument 
in Case $1$ of the proof 
of \cite[Proposition 14]{L-V}, 
the contribution to $Z_{ {\rm top}, f}$ 
from the dual cones of 
$P, A, X, PA, PX$ and $Q, B, Y, QB, QY$ 
is equal to $Z_{ {\rm top}, f}$ 
itself modulo holomorphic functions at 
$s=s_0 \in \CC$. 
For example, the sum of the contributions 
to $Z_{ {\rm top}, f}$ 
from the dual cones of $\sigma_1=XAP$ 
(resp. $\sigma_3=YQB$) and 
$XA$ (resp. $YB$) is zero. The same 
is true also for the dual cones of 
the three faces 
$\sigma_2=PABQ$, $PQ$ and $AB$. 
Here we used the fact that the normalized 
area of the quadrilateral face $\sigma_2=
PABQ$ of $\tau$ is equal to the sum 
of the lengths of the segments $PQ$ and $AB$ 
(see Lemma \ref{ldsproutn} for higher 
dimensional cases). Moreover by Lemma \ref{DLLE} 
it suffices to consider only the 
contribution from their subcones 
$\Delta_P, \Delta_A, \Delta_X, 
\Delta_P \cap \Delta_A, \Delta_P \cap \Delta_X$ and 
$\Delta_Q, \Delta_B, \Delta_Y, 
\Delta_Q \cap \Delta_B, \Delta_Q \cap \Delta_Y$. 
Then by applying Lemma \ref{magiccancel} 
to the pair of cones $\Delta_A \subset A^{\circ}$ 
and $\Delta_P \cap \Delta_A \subset (PA)^{\circ}$ 
(resp. $\Delta_X \subset X^{\circ}$ and 
$\Delta_P \cap \Delta_X \subset (PX)^{\circ}$) etc., 
we obtain an equality 
\begin{align*}
Z_{ {\rm top}, f} (s)  & \equiv 
\\
& \dint_{\square_P}
\exp \Bigl( - 
\langle u , P \rangle s - 
\langle  u , \mbbo \rangle 
\Bigr) du_1 \cdots du_4  + 
\dint_{\square_Q}
\exp \Bigl( - 
\langle u , Q \rangle s - 
\langle  u , \mbbo \rangle 
\Bigr) du_1 \cdots du_4
\end{align*}
modulo holomorphic functions at 
$s=s_0 \in \CC$. 
By Lemma \ref{regintegr} the rational function
\begin{equation*}
\dint_{\square_P}
\exp \Bigl( - 
\langle u , P \rangle s - 
\langle  u , \mbbo \rangle 
\Bigr) du_1 \cdots du_4
\end{equation*}
of $s$ is holomorphic at $s=s_0 \in \CC$. 
The same is true also for the above integral 
over $\square_Q$. 
Hence $Z_{ {\rm top}, f} (s)$ 
is holomorphic at $s=s_0 \in \CC$. 
\end{proof}

\section{Fake poles of the topological zeta function 
in arbitrary dimension}\label{sec:35}

In view of the observations from the preceding section, the following result does not look unexpected. From now on, by faces we mean faces of the Newton polyhedron $\Gamma_+(f)$, unless explicitly stated otherwise.
\begin{definition}
\begin{enumerate}
\item A one-element set $\{i\}\subset\{1,\ldots,n\}$ is called a \emph{base direction for a $B_1$-facet} $\tau$, if the $i$-th coordinate equals 1 for one vertex of $\tau$, and equals 0 for the other vertices.
\item An $(n-2)$-element set $I\subset\{1,\ldots,n\}$ is called a \emph{base direction for a $B_2$-facet}, if its projection to the $I$-th coordinate plane is the standard $(n-2)$-dimensional simplex.
\end{enumerate}
\end{definition}
Note that a $B_1$-facet may have more than one base direction.
\begin{definition}\label{defconsist}
\emph{A collection of $B$-facets is said to be consistent}, if their base directions can be chosen so that: 
\begin{center}
a pair of $B$-facets in the collection have a common facet $\Rightarrow$ 

the intersection of their base directions is non-empty.
\end{center}
\end{definition}
\begin{theorem}\label{ADJJJ} 
Assume that $f$ is non-degenerate and does not have a 
Morse singularity at the origin $0 \in \CC^n$. Let 
$\tau_1, \ldots, \tau_k \prec \Gamma_+(f)$ 
be $B$-facets such that 
\begin{equation*}
s_0=   - \frac{\nu ( \tau_1 )}{N ( \tau_1 )} 
= \cdots \cdots = 
- \frac{\nu ( \tau_k )}{N ( \tau_k )} \not= -1 
\end{equation*}
and their common candidate pole $s_0 \in \QQ$ 
of $Z_{ {\rm top}, f}(s)$ is 
contributed only by 
them. If we can choose
their base directions to be
consistent, then $s_0$ is fake i.e. not an actual 
pole of $Z_{ {\rm top}, f}(s)$. 
\end{theorem} 
We allow ourselves to exclude Morse singularities from consideration, because the monodromy conjecture for Morse singularities is clear.

This section is devoted to the proof of Theorem \ref{ADJJJ}.

In the course of the proof we introduce some new tools that will be used in the next section, which is devoted to a sharper version of Theorem \ref{ADJJJ}, completely classifying configurations of $B$-faces contributing fake poles for $n=4$. This classification is a key point in the proof of the monodromy conjecture for $n=4$.

\subsection{Contributions}

\begin{definition} 
Let $S\subset\RR^n_+$ be a polyhedral set, that is, a disjoint union of 
the relative interiors of some (finitely many) closed 
convex polyhedral cones 
in $\RR^n_+$. Then we define its \emph{contribution 
$Z(S)(s) \in \CC (s)$ (to the 
topological zeta function $Z_{\rm top, f}$)} by 
$$\int_{S}\exp\bigl(-N(u) s- \langle 
u,\mbbo\rangle\bigr)\, du_{\RR^n_+}+
\frac{s}{s+1}\sum_{\tau} (-1)^{\dim\, \tau}
\Vol_\ZZ(\tau)\int_{S\cap\tau^\circ}
\exp\bigl(-N(u) s
- \langle u,\mbbo\rangle\bigr)\, du_{\tau^\circ},$$
where $\tau$ ranges through all the 
positive-dimensional compact faces of 
$\Gamma_+(f)$, $N(\cdot)$ is the support 
function of $\Gamma_+(f)$, 
and $du_{C}$ is the lattice volume form on a 
rational polyhedral cone $C$ (so that all 
components of dimension smaller than $\dim\, C$ 
in the intersection $S\cap C$ do not affect the 
integral with respect to $du_{C}$). 
\end{definition}

\begin{remark} \begin{enumerate}
\item As a function of $S$, the contribution to the topological zeta function is an additive measure.
\item By Lemma \ref{DLLE}, the contribution $Z(\RR^n_+)$ of the open positive quadrant $\RR^n_+$ equals the topological zeta function of a generic $f$ with the given Newton polyhedron $\Gamma_+(f)$.
\item We do not assume the argument $S$ to be closed or open, because $Z(S)$ changes significantly when passing to the closure or the interior of $S$, and indeed we shall need sets $S$ that are neither closed nor open.
\end{enumerate}
\end{remark}

\subsection{The main theorem: the plan of the proof}
${}$
\indent I. Very loosely, the proof of Theorem \ref{ADJJJ} will consist of constructing a particular subdivision of $\RR^n_+$ into pieces $\sigma_i$ such that $Z(\sigma_i)$ has no pole at $s=s_0$. The boldest hope would be to choose $\sigma_i$ so that the key Lemma \ref{regintegr} is directly applicable to every $\sigma_i$, i.e.:

-- Every $\sigma_i$ is contained in the dual cone to some vertex $P$ of the Newton polyhedron;

-- No edge of $\sigma_i$ is critical with respect to $(s_0,P)$ in the sense of Lemma \ref{regintegr}.

\vspace{1ex}

II. Unfortunately, in general step (I) is not realistic as written, because some of the edges of the dual cone $P^\circ$ will be critical, and they also have to be edges for some $\sigma_i$. These critical edges are exactly the ones dual to the contributing facets $\tau_i$ of the Newton polyhedron. 

Fortunately, all such facets $\tau_i$ are $B$-facets in the assumptions of Theorem \ref{ADJJJ}. This will help us to surround every such critical edge $\tau_i^\circ$ by its personal conic neighborhood $\sigma'_i$ (so called {\it sprout}) such that Lemma \ref{regintegr} is still applicable to it:

-- The contribution $Z(\sigma'_i)$ can be written as the integral from Lemma \ref{regintegr} for some appropriate vertex $P_i$ (despite $\sigma'_i$ is not contained in any individual cone of the form $P^\circ$ anymore!)

-- No edge of $\sigma'_i$ is critical with respect to $(s_0,P_i)$.

\vspace{1ex}

III. Unfortunately, upon choosing neighborhoods 
in Step (II), we observe the next (and the last) obstacle: the complement to $\bigcup_i \sigma'_i$ in any cone $P^\circ$ may have new critical edges in the boundary of $P^\circ$ (different from the edges of $P^\circ$ itself). This happens because some cones (so called {\it critical cones}) in the dual fan $\Sigma_0$ entirely (!) consist of critical rays, so every 0-dimensional intersection of a face of $\sigma'_i$ with such critical cone $C$ will create a new critical edge of the complement of $\sigma'_i$ in any cone containing $C$.
\begin{example} Recall that here and in what follows we draw the projectivization of the fan $\Sigma_0$ rather than the fan itself. On the left picture of Figure \ref{picscheme1} below, the critical 1-dimensional cone $\tau_1^\circ$, whose dual facet $\tau_1$ is a $B_1$-pyramid with the apex $P$,
is surrounded with a conical neighborhood $\sigma'_1$ (shown in bold). Since the dashed segment is a critical 2-dimensional cone, its intersection point with the boundary of $\sigma'_1$ is a critical ray that is not an edge of $\sigma'_1$, but is a critical ray of the complement to $\sigma'_1$ in the 3-dimensional cone $P^\circ$. 
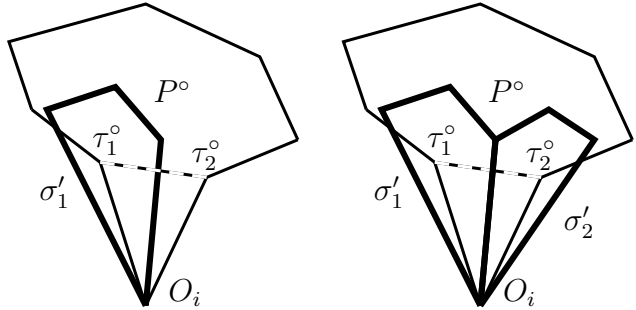
\begin{figure}[h]
\begin{center}
\skipfig{\psscalebox{1.0 1.0} 
{
\begin{pspicture}(0,-2.0218885)(8.243852,2.0218885)
\psline[linecolor=black, linewidth=0.04](1.8256423,-1.9993604)(1.2256424,-0.099360466)(2.6256423,-0.29936045)(1.8256423,-1.9993604)
\psline[linecolor=black, linewidth=0.04](1.2256424,-0.099360466)(0.3256424,0.6006395)
\psline[linecolor=black, linewidth=0.04](2.6256423,-0.29936045)(3.8256423,0.20063953)
\psline[linecolor=black, linewidth=0.08](1.8256423,-1.9993604)(0.5256424,0.6006395)(1.4256424,0.90063953)(2.0256424,0.20063953)(1.8256423,-1.9993604)
\psline[linecolor=white, linewidth=0.04, linestyle=dashed, dash=0.17638889cm 0.10583334cm](1.2256424,-0.099360466)(2.6256423,-0.29936045)
\rput[bl](1.1256424,0.0){$\tau_1^\circ$}
\rput[bl](2.4256425,-0.19936046){$\tau_2^\circ$}
\rput[bl](1.9256424,0.70063955){$P^\circ$}
\rput[bl](2.1256423,-1.9993604){$O_i$}
\psline[linecolor=black, linewidth=0.04](0.3256424,0.6006395)(0.025642395,1.5006396)(1.8256423,2.0006394)(3.3256423,1.3006395)(3.8256423,0.20063953)
\psline[linecolor=black, linewidth=0.04](6.225642,-1.9993604)(5.6256423,-0.099360466)(7.0256424,-0.29936045)(6.225642,-1.9993604)
\psline[linecolor=black, linewidth=0.04](5.6256423,-0.099360466)(4.725642,0.6006395)
\psline[linecolor=black, linewidth=0.04](7.0256424,-0.29936045)(8.225642,0.20063953)
\psline[linecolor=black, linewidth=0.08](6.225642,-1.9993604)(4.9256425,0.6006395)(5.8256426,0.90063953)(6.4256425,0.20063953)(6.225642,-1.9993604)
\psline[linecolor=white, linewidth=0.04, linestyle=dashed, dash=0.17638889cm 0.10583334cm](5.6256423,-0.099360466)(7.0256424,-0.29936045)
\rput[bl](5.5256424,0.0){$\tau_1^\circ$}
\rput[bl](6.8256426,-0.19936046){$\tau_2^\circ$}
\rput[bl](6.3256426,0.70063955){$P^\circ$}
\rput[bl](6.5256424,-1.9993604){$O_i$}
\psline[linecolor=black, linewidth=0.04](4.725642,0.6006395)(4.4256425,1.5006396)(6.225642,2.0006394)(7.725642,1.3006395)(8.225642,0.20063953)
\psline[linecolor=black, linewidth=0.08](6.225642,-1.9993604)(7.725642,0.20063953)(7.1256423,0.6006395)(6.4256425,0.20063953)(6.225642,-1.9993604)
\rput[bl](0.4256424,-0.6993605){$\sigma'_1$}
\rput[bl](4.8256426,-0.6993605){$\sigma'_1$}
\rput[bl](7.3256426,-1.0993605){$\sigma'_2$}
\end{pspicture}
}}
\end{center}
\caption{good neighborhoods of critical cones}\label{picscheme1}
\end{figure}
\end{example}

Fortunately, choosing the neighborhoods $\sigma'_i$  wisely, we can ensure that no new critical edge of the complement of an individual $\sigma'_i$ is an edge of the complement to the whole $\bigcup_i \sigma'_i$ (see the picture on the right of Figure \ref{picscheme1}). For instance, this is done in detail in the proof of Proposition \ref{ADJ} for the case of only two facets contributing the pole $s_0$. 

In the general setting, this will be done using the geometry of so called {\it delimiter planes} and will allow to literally apply Step (I) of our plan to the complement of $\bigcup_i \sigma'_i$. Warning: the resulting cones $\sigma'_i$ will not form a fan in the sense that $\sigma'_i\cap\sigma'_j$ may be not a face of $\sigma'_i$ and $\sigma'_j$.

\vspace{1ex}

We introduce all the aforementioned objects in the subsequent subsections.

\subsection{Intersections of $B$-facets}
\begin{lemma}\label{Noquadrn}
Assume that f is non-degenerate and does 
not have a Morse singularity at the origin $0 \in \CC^n$. 
Then, if two $B$-facets of the Newton polyhedron $\Gamma_+(f)$ have a common $(n-2)$-dimensional face, then it is a $B_1$-pyramid.
\end{lemma}

The proof of this lemma requires the following observation.

\begin{lemma}\label{NDGG} 
If a homogeneous polynomial $g$ of 
degree 2 on $\CC^4$ is 
non-degenerate with respect to its Newton polyhedron
and its support ${\rm supp}\; g$ contains 
the points $(1,1,0,0)$ and $(0,0,1,1)$, 
then $g$ is non-degenerate as a
quadratic form.
\end{lemma}
Upon publishing the first preprint version of this paper, this lemma was beautifully generalized to arbitrary dimension in \cite{yuran}.

\begin{proof}
The non-degeneracy of $g$ as a quadratic 
form is equivalent to the 
smoothness of the hypersurface $V= \{ g=0 \}$  
of $\PP^3$ defined by $g$. Indeed, 
the kernel of the symmetric matrix 
associated to the quadratic 
form corresponds to 
the singular locus of $V \subset \PP^3$. 
Recall that 
$\PP^3$ is naturally a toric variety 
on which the complex torus $T=( \CC^*)^3$ 
acts with 15 orbits. First of all, by 
the non-degeneracy of $g$ w.r.t. its 
Newton polyhedron, 
the hypersurface $V= \{ g=0 \}$  
of $\PP^3$ is smooth in 
$T \subset \PP^3$. So we have only to 
analyze at other points $x$ in 
$\PP^3 \setminus T$. We shall do it 
step by step, considering $x$ in each 
of the other 14 $T$-orbits in $\PP^3$.
We denote the Newton polytope of $g$ by $Q \subset \{ v_1+v_2+v_3+v_4=2 \}$. 

1) Assume that $x=(1:0:0:0)$. If the 
point $(2,0,0,0)$ is in ${\rm supp} g$, then
$g(x) \ne 0$. Otherwise, we have $dg(x) \ne 0$, 
because $(1,1,0,0)\in {\rm supp} g$ and hence 
$\partial g/\partial x_2 (x) \ne 0$. 
This implies that the hypersurface $V= \{ g=0 \}$  
of $\PP^3$ is smooth at $x \in V$.

2) Assume that $x$ is in one of the 
other three 0-dimensional $T$-orbits in 
$\PP^3$. Then the reasoning is the same as in (1).

3) Assume that $x=(s:t:0:0)$ 
($s, t \not= 0$). Then, since the 
face $F = Q \cap [ (2,0,0,0),
(0,2,0,0) ]$ of $Q$ is non-empty (containing 
at least the point $(1,1,0,0)$) and 
the restriction $g_F$ of $g$ to 
the face $F \prec Q$ defines a smooth 
hypersurface in the 
1-dimensional $T$-orbit in $\PP^3$ 
associated to $[ (2,0,0,0),
(0,2,0,0) ]$ (by the non-degeneracy of 
$g$), the hypersurface $V= \{ g=0 \}$  
of $\PP^3$ is smooth at $x$. 

4) Assume that $x=(0:0:s:t)$ ($s, t \not= 0$) 
or $x$ is in one of the four 2-dimensional
$T$-orbits in $\PP^3$. Then the reasoning 
is the same as in (3).

5) Assume that $x=(s:0:t:0)$ ($s,t \not= 0$). 
If ${\rm supp} g \cap [ (2,0,0,0), (0,0,2,0) ]$ is
not empty, then the reasoning is the same 
as in (3). Otherwise, we have $g(x)=0$. 
Assume that we have also $dg(x)=0$. Then in 
particular $\partial g/\partial
x_2(x)=\partial g/\partial x_4(x)=0$. 
From these identities, we see that the
restrictions $g_A$, $g_B$ of $g$ to  
the segments 
$A=[(1,1,0,0), (0,1,1,0)]$, 
$B=[(1,0,0,1), (0,0,1,1)]$
vanish at $x$, 
thus they are multiples of the 
same linear function. 
So the restriction $g_P$ 
of $g$ to the parallelogram
$P={\rm conv} (A \cup B)$ is a product of 
two linear functions. This implies that 
$g_P$ defines a singular hypersurface in 
$( \CC^*)^4$. This would contradict 
the non-degeneracy of $g$ with respect 
to its Newton polyhedron. 

6) Assume that $x$ is in one of the other 
three 1-dimensional $T$ orbits in 
$\PP^3$. Then the reasoning is the same as in (5).
\end{proof}

\begin{lemma}\label{Noquadr} 
Assume that $n=4$, and 
two $B_2$-facets of $\Gamma_+(f)$ have 
a common quadrilateral face. Then $f$ has a 
Morse singularity at the origin $0 \in \CC^4$. 
\end{lemma}
\begin{proof} 
Projectivizing the ambient space of  
$\Gamma_+(f)$, we see the positive octant as 
a tetrahedron, and the two 
$B_2$-facets in it as two polytopes from Figure \ref{picb2b2} with a 
common quadrilateral face. This is only possible if the 
common quadrilateral is a parallelogram, whose edges are 
parallel to two opposite edges of the tetrahedron, and 
whose vertices are contained in the four other edges, as shown on the picture.

\begin{figure}[h]
\begin{center}
\skipfig{\psscalebox{1.0 1.0} 
{
\begin{pspicture}(0,-1.6197056)(5.6372924,1.6197056)
\definecolor{colour0}{rgb}{0.8,0.8,0.8}
\pspolygon[linecolor=black, linewidth=0.034, linestyle=dashed, dash=0.17638889cm 0.10583334cm, fillstyle=solid,fillcolor=colour0](2.8194005,1.2000955)(2.4194007,-0.3999045)(2.8194005,-1.1999044)(5.6194005,-0.79990447)(5.2194004,0.0)
\psline[linecolor=black, linewidth=0.04](0.8194006,1.6000955)(0.019400634,-1.5999045)
\psline[linecolor=black, linewidth=0.04](4.819401,0.8000955)(5.6194005,-0.79990447)
\psline[linecolor=black, linewidth=0.04](0.019400634,-1.5999045)(5.6194005,-0.79990447)
\psline[linecolor=black, linewidth=0.04](0.8194006,1.6000955)(5.6194005,-0.79990447)
\psline[linecolor=black, linewidth=0.04](0.8194006,1.6000955)(4.819401,0.8000955)
\psline[linecolor=black, linewidth=0.04](2.8194005,1.2000955)(3.2194006,0.40009552)(2.8194005,-1.1999044)
\psline[linecolor=black, linewidth=0.04, linestyle=dashed, dash=0.17638889cm 0.10583334cm](2.8194005,1.2000955)(2.4194007,-0.3999045)(2.8194005,-1.1999044)
\psline[linecolor=black, linewidth=0.04, linestyle=dashed, dash=0.17638889cm 0.10583334cm](0.019400634,-1.5999045)(4.819401,0.8000955)
\psline[linecolor=black, linewidth=0.04, linestyle=dashed, dash=0.17638889cm 0.10583334cm](2.8194005,1.2000955)(0.8194006,1.2000955)
\psline[linecolor=black, linewidth=0.04, linestyle=dashed, dash=0.17638889cm 0.10583334cm](3.2194006,0.40009552)(0.8194006,1.2000955)
\psline[linecolor=black, linewidth=0.04, linestyle=dashed, dash=0.17638889cm 0.10583334cm](2.8194005,1.2000955)(5.2194004,0.0)
\psline[linecolor=black, linewidth=0.04, linestyle=dashed, dash=0.17638889cm 0.10583334cm](2.4194007,-0.3999045)(5.2194004,0.0)
\psline[linecolor=black, linewidth=0.034, linestyle=dashed, dash=0.17638889cm 0.10583334cm](2.8194005,-1.1477306)(0.384618,-0.10425231)
\psline[linecolor=black, linewidth=0.034, linestyle=dashed, dash=0.17638889cm 0.10583334cm](2.4715745,-0.4520784)(0.384618,-0.10425231)
\end{pspicture}
}}

\caption{$B_2$-facets with a common quadrilateral face}\label{picb2b2}

\end{center}
\end{figure}
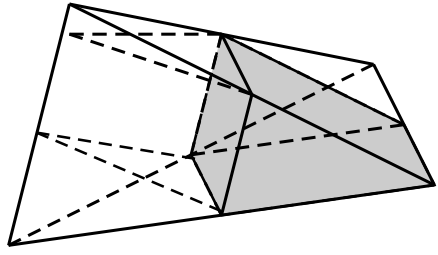

Thus, reordering coordinates if necessary, 
the vertices are of the form $(*,*,0,0),\, 
(0,*,*,0),\, (0,0,*,*)\, (*,0,0,*)$, and all the 
stars are equal to 1 by the definition of  
$B_2$-facets. By Lemma \ref{NDGG}, 
the quadratic part of $f$ 
is non-degenerate, and hence $f$ has a 
Morse singularity at the origin $0 \in \CC^4$.  
\end{proof}

{\it Proof of Lemma \ref{Noquadrn}.} 
Assume that an $(n-2)$-dimensional non-$B_1$-pyramid $F$ is contained in two $B$-facets $F_1$ and $F_2$. First, note that $F$ is not a $V$-face, otherwise one of $F_1$ and $F_2$ were contained in the boundary of $\RR^n_+$, which cannot happen for a $B$-facet.

Second, note that neither of $F_1$ and $F_2$ can be a $B_1$-facet, because every $(n-2)$-dimensional face of a $B_1$-facet is either a $V$-face, or a $B_1$-pyramid.

Finally, the only $(n-2)$-dimensional non-$B_1$ non-$V$-face of a $B_2$-facet $F_i$
is combinatorially isomorphic to the product of a segment and an $(n-3)$-simplex (let us call it the front face of the $B_2$-facet). So the only exception from the statement of the lemma could come from two $B_2$-facets with different base directions and the common front face. However, for $n>4$ this is impossible, because two products of a segment and an $(n-3)$-simplex cannot be non-trivially combinatorially isomorphic, and for $n=4$ this is excluded by Lemma \ref{Noquadr}. \hfill{} $\square$

\subsection{Bases and apices}\label{ssbases}

\begin{definition}
The \emph{star 
of a cone $C$ in the dual fan of $\Gamma_+(f)$} is the set of all cones containing $C$. 

\end{definition}
\begin{comment}
For each $B_1$-face 
$\tau$ contributing the pole $s_0$
we can choose its apex 
and preferred base 
to be its vertex $P$ and a number $i\in\{1,\ldots,n\}$ respectively, such that the $i$-th coordinate of $P$ equals 1 and the $i$-th coordinates of the other vertices of $\tau$ equals 0. Note that we may have several options for this choice.

We now fix once and for all the choice of apices $P_\tau$ and preferred bases $b_\tau$

-- for all $B_1$-facets $\tau$, and

-- for all $B_1$-facets $\tau$ of $B_2$-facets (by a $B_1$-facet of a facet $\sigma$ we mean a codimension 2 face of the Newton polyhedron that belongs to $\sigma$ and is a $B_1$-pyramid).

Moreover, in the setting of Theorem \ref{ADJJJ}, we can choose the apices and the preferred bases consistently, so that the preferred base of every aforementioned face $\tau$ belongs to the base direction of every $B$-facet containing $\tau$. This in particular, ensures that
 
 -- every two $B_1$-facets intersecting by a codimension 2 $B_1$-face have the same apex and preferred base, and
 
 -- if a codimension 2 $B_1$-face $\tau'$ is at the same time a facet of a $B_1$-facet $\tau$ and of a $B_2$-facet (so that we have chosen the preferred base and apex for it), then $\tau'$ and $\tau$ have the same apex and preferred base.
\end{comment} 
For each
$B_1$-face $\tau$ contributing to the
candidate pole $s_0$ we can choose its apex
and preferred base
to be its vertex $P$ and a number
$i\in\{1,\ldots,n\}$ respectively, such
that the $i$-th coordinate of $P$ equals 1
and the $i$-th coordinates of the other
vertices of $\tau$ equals 0. Note that we
may have several options for this choice.

We now fix once and for all the choice of
apices $P_\tau$ and preferred bases $b_\tau$

-- for all $B_1$-facets $\tau$ contributing to the
candidate pole $s_0$, and

-- for all $B_1$-facets $\tau$ of $B_2$-facets
contributing to the candidate pole $s_0$
(by a $B_1$-facet of a facet $\sigma$ we mean
a codimension 2 face of the Newton polyhedron
that belongs to $\sigma$ and is a $B_1$-pyramid).

Moreover, in the setting of Theorem \ref{ADJJJ},
we can choose the apices and the preferred bases
consistently, so that the preferred base of
every aforementioned face $\tau$ belongs to
the base direction of every $B$-facet
containing $\tau$. This in particular, ensures that
 
 -- every two $B_1$-facets intersecting
by a codimension 2 $B_1$-face have the same
apex and preferred base, and
 
 -- if a codimension 2 $B_1$-face $\tau'$
is at the same time a facet of a $B_1$-facet
$\tau$ and of a $B_2$-facet (so that we have
chosen the preferred base and apex for it),
then $\tau'$ and $\tau$ have the same
apex and preferred base.

\begin{definition}\label{defdelimb2n}
Let $L$ be the 2-dimensional coordinate plane along which the projection of a $B_2$-facet $\tau$ equals the standard simplex, and let $l\subset\RR^n$ be the 1-dimensional vector space parallel to the intersection of $L$ with the affine span of $\tau$.

The dual hyperplane to $l$ will be denoted by $D_\tau=D_{\tau^\circ}$ and called the \emph{delimiter of $\tau$}. 
\end{definition}
In particular, if $n=4$, then, 
in the notation of Figure \ref{picb2}, the delimiter is the 3-dimensional plane normal to the three parallel segments. Furthermore, the hyperplane H in the proof of Proposition 3.11 is nothing but the delimiter of $\tau$. 
\begin{remark}
Most $B_2$-facets have a unique $V$-edge normal to the delimiter. However, we allow this edge to degenerate into a vertex (such $B_2$-facets are said to be degenerate). However, various attributes of this $V$-edge make natural sense even for degenerate $B_2$-facets. For instance, ``the length of the $V$-edge of $\tau$'' and ``the dual cone of the $V$-edge of $\tau$'' refer to 0 and $D_\tau\cap V^\circ$ respectively for a degenerate $B_2$-facet $\tau$ with a vertex $V$ instead of the $V$-edge. In what follows, this small abuse of terminology never causes confusion.
\end{remark}
\begin{remark}
Every $B_2$-facet (including the degenerate ones) has three distinguished facets (i.e. $(n-2)$-dimensional faces): namely, it has exactly one non-simplicial non-$V$-facet and exactly two $B_1$-facets. 

For instance, in the 4-dimensional setting of Figure \ref{picb2}, they are denoted by $\sigma_2, \sigma_1$ and $\sigma_3$ respectively.
\end{remark}

\begin{definition} \label{defontheside} \begin{enumerate}
\item We shall say that a vertex $P$ of a $B_2$-facet $\tau$ and a number $i\in\{1,\ldots,n\}$
are an \emph{apex} and a \emph{preferred base of $\tau$ on the side of a face $F \prec \tau$}, if $P$ and $i$ are the apex and the preferred base of a $B_1$-facet $\sigma$ of $\tau$, such that $\sigma\cap F\ne\varnothing$. 
If $\tau$ has a unique preferred base on the side of $F$ (that is, if $\sigma$ is uniquely defined by the condition $\sigma\cap F\ne\varnothing$), this preferred base will be denoted by $b_\tau^F=b_{\tau^\circ}^F$.

\item Similarly, if a cone $C$ in the star of $\tau^\circ$ does not intersect the delimiter $D_\tau$, then the apex and the preferred base of the $B_2$-facet $\tau$ on the side of $C$ are defined as the apex and the preferred base of the $B_1$-facet of $\tau$, whose dual cone is not separated from $C$ by the delimiter.
\item For conformity, we shall say that a vertex $P$ of a $B_1$-facet $\tau$ and a number $i$
are its apex and preferred base on the side of a face $F \prec \tau$ (or a cone $C$ in the star of $\tau^\circ$), if $P$ and $i$ are the apex and the preferred base of $\tau$ (independently of $F$ and $C$).
\end{enumerate}
\end{definition}
\begin{example}\label{exaontheside} For instance, let $n=4$, consider the $B_2$-facet $\tau$ on Figure \ref{picb2}, and assume that (in the notation of this figure) the preferred base and apex for the $B_1$-triangle $\sigma_1$ are $2$ and $P$, while those for $\sigma_3$ are $1$ and $B$. 
Then 

-- $P$ is the apex of $\tau$ on the side of $\sigma_1$ and all of its faces,

-- $B$ is the apex of $\tau$ on the side of $\sigma_3$ and all of its faces, 

-- both $B$ and $P$ are apices of $\tau$ on the side of its other 7 faces.

Further, the delimiter $D_\tau$ is the hyperplane normal to the segment $XY$. It divides the dual space into two half-spaces, containing the dual cones to the triangles $\sigma_1$ and $\sigma_3$, let us call them ``left'' and ``right'' respectively. If a cone $C$ is in the left (respectively right)  half-space, then $P$ (respectively $B$) is the apex of $\tau$ on the side of $C$.
\end{example}
\begin{lemma}\label{onthesiden} If a face $F$ is contained in a $B$-facet $\tau$ and not contained in a coordinate hyperplane, then every apex of $\tau$ on the side of $F$ is contained in $F$.
\end{lemma}

\subsection{Sprouts and cancellation of contributions}

Let $C$ be a (not necessarily convex) polyhedral cone in the star of $\tau^\circ$, where $\tau$ is a $B$-facet. In the case of a $B_2$-facet, we additionally assume that $C$ is not intersected by the delimiter $D_\tau$.
Let $P$ and $i$ be the apex and the preferred base of $\tau$ on the side of $C$. Note that they uniquely determine each other, and that we have chosen them once and for all in the preceding subsection. Recall that the standard basis in $\RR^n$ is denoted by $e_1,\ldots,e_n$.
\begin{definition} \label{defsprout} The union of $C\cap\relint P^\circ$ and all two-dimensional cones, generated by $e_i$ and a point of $C\cap \partial P^\circ$, is denoted by $S_{C,\tau}$ and is called \emph{the sprout of $C$}. The vertex $P$ is denoted by $R_{C,\tau}$ and is called \emph{the root of $C$}. \end{definition}
For example, if $\tau$ is a $B_1$-facet, then
the cone $\square$ in the proof of Proposition 3.7
is a sprout of some cone.

\begin{remark}
By the definition, if $C$ is a closed $n$-dimensional polyhedral cone (in the sense of Definition \ref{newdefi}), containing $\tau^\circ$ in its interior, then so is the sprout. On the other hand, convexity of $C$ does not imply convexity of the sprout, see Figure \ref{picsprout}.
\end{remark}
\begin{example}
Figure \ref{picsprout} gives the simplest example of a cone $C$ and its sprout in the star of the ray $\tau^\circ$.
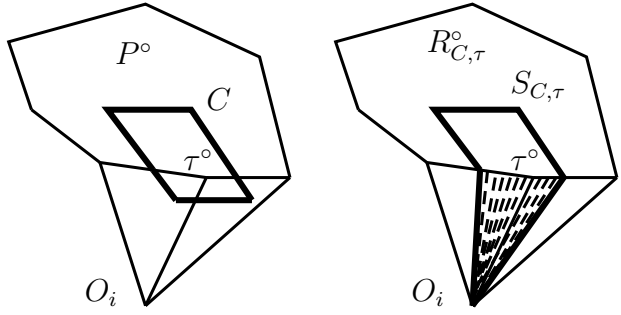
\begin{figure}[h]\label{fig5rem}
\begin{center}
\skipfig{\psscalebox{1.0 1.0} % Change this value to rescale the drawing.
{
\begin{pspicture}(0,-2.0501754)(18.002663,2.0501754)
\psline[linecolor=black, linewidth=0.04](1.8256421,-1.9710739)(1.2256422,-0.07107376)(2.625642,-0.27107376)(1.8256421,-1.9710739)
\psline[linecolor=black, linewidth=0.04](1.2256422,-0.07107376)(0.32564226,0.6289262)
\psline[linecolor=black, linewidth=0.04](2.625642,-0.27107376)(3.7256422,-0.27107376)
\psline[linecolor=black, linewidth=0.08](1.9256423,-0.17107375)(1.3256421,0.6289262)(2.4256423,0.6289262)(3.0256422,-0.27107376)(3.2256422,-0.5710737)
\rput[bl](2.325642,-0.17107375){$\tau^\circ$}
\rput[bl](1.4256423,1.2289263){$P^\circ$}
\rput[bl](1.0256423,-1.9710739){$O_i$}
\psline[linecolor=black, linewidth=0.04](0.32564226,0.6289262)(0.025642255,1.5289261)(1.8256421,2.0289264)(3.325642,1.3289263)(3.7256422,-0.27107376)
\psline[linecolor=black, linewidth=0.08](2.2256422,-0.5710737)(3.2256422,-0.5710737)
\rput[bl](2.625642,0.6289262){$C$}
\psline[linecolor=black, linewidth=0.04](1.8256421,-1.9710739)(3.7256422,-0.27107376)
\psline[linecolor=black, linewidth=0.08](1.9256423,-0.17107375)(2.2256422,-0.5710737)
\psline[linecolor=black, linewidth=0.04](6.1256423,-1.9710739)(5.5256424,-0.07107376)(6.9256425,-0.27107376)(6.1256423,-1.9710739)
\psline[linecolor=black, linewidth=0.04](5.5256424,-0.07107376)(4.6256423,0.6289262)
\psline[linecolor=black, linewidth=0.04](6.9256425,-0.27107376)(8.025641,-0.27107376)
\psline[linecolor=black, linewidth=0.08](6.2256417,-0.17107375)(5.6256423,0.6289262)(6.7256417,0.6289262)(7.3256426,-0.27107376)(6.1256423,-1.9710739)
\rput[bl](6.6256423,-0.17107375){$\tau^\circ$}
\rput[bl](5.5256424,1.2289263){$R_{C,\tau}^\circ$}
\rput[bl](5.3256426,-1.9710739){$O_i$}
\psline[linecolor=black, linewidth=0.04](4.6256423,0.6289262)(4.3256426,1.5289261)(6.1256423,2.0289264)(7.6256423,1.3289263)(8.025641,-0.27107376)
\rput[bl](6.6256423,0.7289263){$S_{C,\tau}$}
\psline[linecolor=black, linewidth=0.04](6.1256423,-1.9710739)(8.025641,-0.27107376)
\psline[linecolor=black, linewidth=0.08](6.2256417,-0.17107375)(6.1256423,-1.9710739)
\psline[linecolor=black, linewidth=0.04, linestyle=dashed, dash=0.17638889cm 0.10583334cm](6.3256426,-0.17107375)(6.1256423,-1.9710739)(6.4256425,-0.27107376)
\psline[linecolor=black, linewidth=0.04, linestyle=dashed, dash=0.17638889cm 0.10583334cm](6.5256424,-0.27107376)(6.1256423,-1.8710737)
\psline[linecolor=black, linewidth=0.04, linestyle=dashed, dash=0.17638889cm 0.10583334cm](6.6256423,-0.27107376)(6.1256423,-1.8710737)(6.7256417,-0.27107376)
\psline[linecolor=black, linewidth=0.04, linestyle=dashed, dash=0.17638889cm 0.10583334cm](6.8256426,-0.27107376)(6.1256423,-1.8710737)(7.0256424,-0.27107376)
\psline[linecolor=black, linewidth=0.04, linestyle=dashed, dash=0.17638889cm 0.10583334cm](7.1256423,-0.27107376)(6.1256423,-1.9710739)(7.2256417,-0.27107376)
\psline[linecolor=black, linewidth=0.04](10.985642,-1.9710739)(10.385642,-0.07107376)(11.785643,-0.27107376)(10.985642,-1.9710739)
\psline[linecolor=black, linewidth=0.04](10.385642,-0.07107376)(9.485642,0.6289262)
\psline[linecolor=black, linewidth=0.08](11.085642,-0.17107375)(10.485642,0.6289262)(11.585643,0.6289262)(11.947345,0.12892623)(11.628196,-0.57958436)
\rput[bl](11.8516,-0.44341418){$\tau^\circ$}
\rput[bl](10.585642,1.2289263){$P^\circ$}
\rput[bl](10.185642,-1.9710739){$O_i$}
\psline[linecolor=black, linewidth=0.04](9.485642,0.6289262)(9.185642,1.5289261)(10.985642,2.0289264)(12.485642,1.3289263)(11.789143,-0.2622886)
\psline[linecolor=black, linewidth=0.08](11.385642,-0.5710737)(11.641647,-0.5837024)
\rput[bl](11.709046,0.63743687){$C$}
\psline[linecolor=black, linewidth=0.08](11.085642,-0.17107375)(11.385642,-0.5710737)
\psline[linecolor=black, linewidth=0.04](15.285643,-1.9710739)(14.685642,-0.07107376)(16.085642,-0.27107376)(15.285643,-1.9710739)
\psline[linecolor=black, linewidth=0.04](14.685642,-0.07107376)(13.785643,0.6289262)
\psline[linecolor=black, linewidth=0.08](15.385642,-0.17107375)(14.785643,0.6289262)(15.885642,0.6289262)(16.264366,0.1204156)(15.285643,-1.9710739)
\rput[bl](16.202663,-0.46043545){$\tau^\circ$}
\rput[bl](14.685642,1.2289263){$R_{C,\tau}^\circ$}
\rput[bl](14.485642,-1.9710739){$O_i$}
\psline[linecolor=black, linewidth=0.04](13.785643,0.6289262)(13.485642,1.5289261)(15.285643,2.0289264)(16.785643,1.3289263)(16.079258,-0.28809503)
\rput[bl](15.530323,0.7289263){$S_{C,\tau}$}
\psline[linecolor=black, linewidth=0.08](15.385642,-0.17107375)(15.285643,-1.9710739)
\psline[linecolor=black, linewidth=0.04, linestyle=dashed, dash=0.17638889cm 0.10583334cm](15.485642,-0.17107375)(15.285643,-1.9710739)(15.585643,-0.27107376)
\psline[linecolor=black, linewidth=0.04, linestyle=dashed, dash=0.17638889cm 0.10583334cm](15.685642,-0.27107376)(15.285643,-1.8710737)
\psline[linecolor=black, linewidth=0.04, linestyle=dashed, dash=0.17638889cm 0.10583334cm](15.785643,-0.27107376)(15.285643,-1.8710737)(15.885642,-0.27107376)
\psline[linecolor=black, linewidth=0.04, linestyle=dashed, dash=0.17638889cm 0.10583334cm](15.985642,-0.27107376)(15.285643,-1.8710737)(15.3345785,-1.8795844)
\end{pspicture}
}}
\caption{the sprout of the cone $C$ (the case of compact $\tau$ on the left and non-compact $\tau$ on the right)}\label{picsprout}
\end{center} 
\end{figure}

\end{example}
\begin{lemma} \label{lsprout} The contribution of a sprout is simple: $$Z(S_{C,\tau})=\int_{S_{C,\tau}} \exp\bigl(-\langle u, R_{C,\tau}\rangle s-\langle u, \mbbo\rangle\bigr)\, du.$$ 
\end{lemma}
\begin{proof} The contribution of $C\cap\relint P^\circ$ equals the integral of $\exp\bigl(-\langle u, P\rangle s-\langle u, \mbbo\rangle\bigr)$ by definition. Triangulating the set $C\cap \partial P^\circ$ and, for every simplicial cone $T$ of this triangulation, applying Lemma \ref{magiccancel} to the contribution of the cone generated by $T$ and $e_i$, we conclude that the rest of $S_{C,\tau}$ also contributes the integral of $\exp\bigl(-\langle u, P\rangle s-\langle u, \mbbo\rangle\bigr)$.
\end{proof}
Let $\tau$ be a $B_2$-facet, 
and assume without loss of generality that the 
projection of $\tau$ along the $(e_{n-1}, e_n)$-coordinate 
plane is the $(n-2)$-dimensional 
standard simplex $S$ with the vertices $0,e_1,\ldots,e_{n-2}$. The preimage of the facet $(e_1\ldots e_{n-2})$ of this simplex under this projection intersects $\tau$ by its unique non-simplicial non-$V$-facet 
$\tau_0$. Choose a ray $r$ in 
the (relatively open) dual cone to $\tau_0$. 
\begin{definition} The 
$(n-1)$-dimensional (relatively open) 
cone $S_{r,\tau}$ generated by 
$r, e_1,\ldots,e_{n-2}$ is called 
\emph{the delimiter sprout of $r$}.
\end{definition}
\begin{lemma} \label{ldsproutn}  The contribution of the delimiter sprout $Z(S_{r,\tau})$
is equal to $0$ (independently of the choice of the ray $r$). 
\end{lemma}
\begin{proof}
For a subset $I\subset\{1,\ldots,n-2 \}$, introduce the following notation:

-- $C_I$ is the cone generated by $\tau^\circ$ 
and $e_i$ $( i\in I)$, 

-- $\widetilde C_I$ is the cone spanned by $C_I$ and $r$,

-- $\tau_I$ is the preimage of the simplex with the vertices $e_i,\, i\in I$, under the projection $\tau\to S$.

 -- $\tilde\tau_I$ is the preimage of the simplex with the vertices 0 and $e_i,\, i\in I$, under the projection $\tau\to S$.

Then the contribution of the delimiter sprout splits into those of the cones $C_I$ 
for $I\subset\{1,\ldots,n-2 \}$ and 
the cones $\widetilde C_I$ 
for $I\subsetneq\{1,\ldots,n-2 \}$.

In order to evaluate them, denote the lattice length of the segment $\tau_{i}$ by $g_i$ and notice that the lattice volume of $\tau_I$ equals $\sum_{i\in I} g_i$.
Then the sum of 
the contributions 
$$Z( \widetilde C_I )(s) 
= (-1)^{n-|I|-2} \frac{s}{s+1} 
( \sum_{i \notin I} g_i) 
J_{\widetilde C_I}(s) 
$$
of the cones $\widetilde C_I$ 
over $I\subsetneq\{1,\ldots,n-2 \}$ 
is equal to $0$, because the functions $J_{\widetilde C_I}(s)$ 
for all such $I$ are equal to each other by definition (c.f. a similar calculation in Lemma \ref{magiccancel}). 

In the same way we can show 
that the corresponding sum for the cones $C_I$ over $I\subset\{1,\ldots,n-2 \}$ is $0$, because the lattice volume of $\tilde\tau_I$ equals the lattice volume of $\tau_I$ plus the lattice length of the segment $\tilde \tau_\varnothing$, and all $J_{C_I}(s)$ are equal to each other as well (more specifically, they are $p$ times smaller than $J_{\widetilde C_I}(s)$, where $p$ is the coefficient from the decomposition  of the primitive generator of the ray $r$ into the linear combination $p \cdot (e_1+\ldots+e_{n-2}) + q\cdot($the primitive generator of $\tau^\circ)$). 
\begin{comment}
For $i\in\{1,\ldots,n-2 \}$ we denote the 
length of the edge of $\tau$ in the preimage 
of $e_i$ by $g_i$. 
For a subset $I\subset\{1,\ldots,n-2 \}$, denote 
by $C_I$ the cone generated by $\tau^\circ$ 
and $e_i$ $( i\in I)$, and by $\widetilde C_I$ 
the one spanned by $C_I$ and $r$. 
Then the contribution of the delimiter sprout splits into those of the cones $C_I$ 
for $I\subset\{1,\ldots,n-2 \}$ and 
the cones $\widetilde C_I$ 
for $I\subsetneq\{1,\ldots,n-2 \}$. 
Note that the normalized volume of the 
non-simplicial non-$V$-facet $\tau_0$ 
of $\tau$ is equal to $\sum_{i=1}^{n-2} g_i$. 
We have also a similar description for 
the normalized volume of any face of $\tau_0$. 
Then the sum of 
the contributions 
of the cones $\widetilde C_I$ 
over $I\subsetneq\{1,\ldots,n-2 \}$ 
is equal to $0$ (see Lemma \ref{magiccancel} 
for a similar calculation). 
Similarly we can show 
that the corresponding sum for the cones $C_I$ 
over $I\subset\{1,\ldots,n-2 \}$ is $0$. y
\end{comment}
\end{proof}

\begin{lemma} \label{sproutrays} The edges of a sprout or a delimiter sprout $S_{C,\tau}$ are either edges of the cone 
$(\partial C)\cap\sigma^\circ$ for some face $\sigma\prec\tau$, 
or coordinate rays. 
\end{lemma}
This directly follows from the definition of a sprout.

\subsection{Critical cones}

We now introduce and study the key notion of critical 
faces for a candidate pole $s_0 \neq -1$ 
of the zeta function $Z_{top, f}(s)$. Similarly to Lemma \ref{nontri}, we have the following lemma. 

\begin{lemma} \label{nontrivn} For $s_0\ne -1$, any vertex 
$P$ of 
a $B$-facet, and any hyperplane $\{v_i=0\}$ which is either base for this $B$-facet, or contains $P$, 
the equation 
$\langle u,P \rangle s_0+ \langle u,\mbbo \rangle =0$ 
is not satisfied by the coordinate vector $u=e_i$. In particular, this equation is non-trivial and thus defines a 
hyperplane $L_P$ in $\RR^n$. 
\end{lemma}
\begin{proof}
$\langle e_i,P \rangle s_0+ \langle e_i,\mbbo \rangle$ equals $1$ or $1+s_0$ in this setting.
\end{proof}
Starting from the following definition, the things depend on the choice of preferred bases and apices of $B$-faces, which we have fixed in Section 4.4 for the rest of the paper.

\begin{definition} For a face $F$, we define the 
\emph{critical set} $K_F \subset F^{\circ}$ 
to be the closure of $\cup_P ( \relint F^\circ \cap L_P)$, 
where in the union $\cup_P$ the vertex 
$P \prec \Gamma_+(f)$ ranges through the apices on the side of $F$ for 
$B$-facets containing $F$ (see Definition \ref{defontheside}). 
\end{definition}

\begin{lemma} \label{lcritsetn} \begin{enumerate}
\item If a face $F$ is contained in a coordinate plane, then $K_F$ is a finite union of hyperplanes in $F^\circ$.
\item If a face $F$ is not contained in a coordinate plane, then $K_F$ is given by one equation $\{u\in F^\circ\,|\, \langle u,Q \rangle s_0+ \langle u,\mbbo \rangle =0\}$, 
%in $F^{\circ}$, 
where $Q$ is an arbitrary point of $F$. In particular, either $K_F$ is a hyperplane in $F^\circ$, or $K_F=F^\circ$.
\end{enumerate}
\end{lemma}
\begin{proof} 1) If $F$ is in the coordinate plane $v_i=0$, then $e_i\in F^\circ$. 
If $F$ is in a $B$-facet $\tau$ contributing to $s_0$, then either $F$ is in the preferred base of $\tau$ (so that $v_i=0$ may be chosen to be this preferred base), or $F$ contains the apex $P\in\tau$ on its side (so that $P$ is in $v_i=0$). In both cases, the set $L_P$ does not contain $e_i\in F^\circ$ by Lemma \ref{nontrivn}, so no $L_P$ can contain the whole $F^\circ$.

2) If $F$ is not in a coordinate plane, then any $P$ in the definition of the critical set
is contained in $F$ by Lemma \ref{onthesiden}, so $K_F$ is given by the equation 
$\langle u,P \rangle s_0+ \langle u,\mbbo \rangle =0$. The set defined by this equation in $F^\circ$  does not change if we substitute $P$ with any other point $Q\in F$, because $\langle Q-P, u\rangle=0$ for $u\in F^\circ$.
\end{proof}

\begin{definition}\label{defcritn}
We say that \emph{a face $F \prec \Gamma_+(f)$ and its 
dual cone $F^{\circ} \in \Sigma_0$ are 
critical (for the candidate pole $s_0$)} if $K_F=F^\circ$.
\end{definition} 

\begin{lemma}\label{lcritfacen} A critical face is not contained in a coordinate plane. 
\end{lemma}
This follows from Lemma \ref{lcritsetn}(1).

\begin{proposition}\label{pcritface}
Assume that a candidate pole $s_0 \not= -1$ is contributed
only by $B$-facets.
In this case, for a face $F$
the following conditions are equivalent:
\begin{enumerate}
\item The face $F$ is critical, that is,
$F$ is contained in a $B$-facet contributing to $s_0$, and, for its apex $P$
on the side of $F$, the plane $L_P$ contains the cone $F^\circ$.
\item Every facet containing $F$ is a $B$-facet contributing to $s_0$,
and, for every apex $P$ of it on the side of $F$,
the plane $L_P$ contains the cone $F^\circ$.
\item Every facet containing $F$ contributes
to
the candidate pole $s_0$ to the topological $\zeta$-function.
\end{enumerate}
\end{proposition}

The second condition will be mostly used in practice,
and the third one is especially simple (and relates the first two).

\begin{proof}

\par \noindent
{\bf ((1) $\Rightarrow$ (3))}
By Lemma \ref{lcritfacen} the critical face $F$
is not contained in a coordinate hyperplane.
Then by Lemma \ref{lcritsetn} (2), the equation
of $K_F$ in $F^{\circ}$ is given by
$$\langle u,Q \rangle s_0+ \langle u,\mbbo \rangle =0,$$
where $Q$ is an arbitrary point of $F$.
Since $F$ is critical i.e. $K_F=F^{\circ}$,
it holds for any $u \in F^{\circ}$. In particular,
for every facet $\tau$ containing $F$, its
conormal vector $a( \tau ) \in \tau^{\circ} \prec F^{\circ}$
satisfies the condition
$$\langle a( \tau ),Q \rangle s_0+ \langle a( \tau ),
\mbbo \rangle =0.$$
Thus $\tau$ contributes the candidate pole $s_0$.
\vvspace
\par \noindent
{\bf ((3) $\Rightarrow$ (2))}
By the assumptions of the proposition, every facet
$\tau$ containing $F$ is a $B$-facet.
In particular, $F$ is not contained in a coordinate hyperplane.
By Lemma \ref{onthesiden}, every apex $P \in \tau$
(of every facet $\tau$ containing $F$) on the side of $F$
is contained in $F$. Then by the condition (3),
for the conormal vector $a( \tau ) \in \tau^{\circ} \prec F^{\circ}$
we have
$$\langle a( \tau ),P \rangle s_0+ \langle a( \tau ),
\mbbo \rangle =0.$$
Since such conormal vectors $a( \tau )$ generate the cone
$F^{\circ}$, for any $u \in F^{\circ}$ we have
$$\langle u,P \rangle s_0+ \langle u,
\mbbo \rangle =0.$$
\vvspace
\par \noindent
{\bf ((2) $\Rightarrow$ (1))}
This is evident.
\end{proof}

\begin{corollary} \label{critclosed} A face of a critical cone in the dual fan $\Sigma_0$ is a critical cone. Equivalently, a face containing a critical face is critical itself.
\end{corollary}

Note that, however, for non-critical faces $F\subset G$, it is not in general true that $K_G
\subset 
K_F$.

It is now a crucial observation that, under the assumptions of Theorem \ref{ADJJJ}, every critical face $F$ is a $B_2$-facet or $B_1$-pyramid (by Lemma \ref{Noquadrn} and Proposition \ref{pcritface}). In the latter case, 
using the assumption
of Theorem 4.3 we can choose apices of
$B_1$-facets and $B$-facets of $B_2$-facets as in Section 4.4, so
that 
all $B$-facets, containing $F$, have the same preferred base $b_F$ and the same apex $P_F$ on the side of $F$.

\subsection{A tubular neighborhood of the critical subfan}\label{ssgeomn}

We are now ready to prove Theorem \ref{ADJJJ}.  When referring to $B$-facets or critical faces in the course of the proof, we always mean only the faces contributing to the candidate pole $s_0\ne-1$, for which we are proving Theorem \ref{ADJJJ} (thus all the choices and objects that we introduce for the proof completely depend on the choice of $s_0$). Recall that by this time we have 

-- chosen once and for all a preferred base of every $B_1$-facet and every $B_1$-facet of every $B_2$-facet in the Newton polyhedron $\Gamma_+(f)$, 

-- depending on this choice, called some cones critical in the dual fan $\Sigma_0$,

-- defined the delimiter hyperplane $D_\sigma$ for every cone $\sigma$, dual to a $B_2$-facet (see Definition \ref{defdelimb2n}). 

Choose once and for all an affine structure on the projectivization ${\mathbb P}\RR^n_+$. For every cone $C$ and fan $\Sigma$, denote their projectivizations by ${\mathbb P}C$ and ${\mathbb P}\Sigma$ respectively. 
In this subsection, we refer to cones and their projectivizations interchangeably whenever it causes no confusion.

According to Corollary \ref{critclosed}, the set of projectivized critical cones in the dual fan $\Sigma_0$ is a polyhedral complex $\Sigma_c$, closed with respect to taking faces and intersections. 
We shall construct a generic piecewise linear tubular neighborhood of $\Sigma_c$ in ${\mathbb P}\RR^n_+$, whose boundary is transversal to the projectivized critical sets of non-critical cones.

In differential geometry, the standard way to construct stratified tubular neighborhood starts with fixing a metric. We shall mimick the same approach in our PL setting.

Recall that two polyhedra in $\RR^n$ are said to be transversal in $U\subset\RR^n$, if they have no common points in $U$, or their union is not contained in an affine hyperplane. More generally, two piecewise linear sets in $\RR^n$ are said to be transversal in $U\subset\RR^n$, if they can be subdivided into relatively open polyhedra so that the polyhedra from these two subdivisions are pairwise transversal in $U$.

Recall that the corner locus of a continuous piecewise linear function is the (piecewise linear) set of all points at which the function is not smooth, and that for convex PL functions it has the natural structure of a polyhedral complex (defined by the projections of the faces of the subgraph of the function). When discussing the transversality to the corner locus, we always imply transversality in the sense of this polyhedral structure. 

To every polyhedral complex $M$ in ${\mathbb P}\RR^n_+$, assign its tangent bundle $TM$: define the tangent plane $T_xM$ at a point $x\in M$ as the maximal vector space $L$ 
(lying in the $(n-1)$-dimensional vector space $V$, underlying the ambient affine space of ${\mathbb P}\RR^n_+$), such that for every $\tilde x\in M$ sufficiently close to $x$ the affine plane $\tilde x+L$ is contained in $M$ in a small neighborhood of $\tilde x$. 
The tangent bundle $TM$ is the (finite) set of all tangent spaces to $M$.

We shall say that $h:V\to{\mathbb R}$ is a piecewise linear norm, transversal to a collection of subspaces $V_1,\ldots,V_M\subset V$, if:

1) $h(w)=\max_{i=1}^N h_i(w)$, where $h_1,\ldots,h_N$ are linear functions, 

2) $h(w)>0$ for $w\ne 0$,

3) For every pair of subsets $I\subset\{1,\ldots,N\}$ and $J\subset\{1,\ldots,M\}$ and every cone $\sigma\in\Sigma_0$, 
the projections of the spaces $H_I:=\{w\,|\,h_i(w)=h_k(w)$ for $i,k\in I\}$ and  $\cap_{j\in J} V_j$ along the vector space parallel to ${\mathbb P}\sigma$ are transversal outside 0.

For $N\geqslant n$, condition (2) is satisfied for all tuples of linear functions $h_1,\ldots,h_N$ from a non-empty open cone in the space of all tuples. In this cone, condition (3) is satisfied for almost all tuples of linear functions $h_i$.

Thanks to this, we can choose once and for all a norm $h$ transversal to the following collection of subspaces:

-- the tangent bundle of the critical set $T{\mathbb P}K_{\sigma'}$ for every non-critical cone $\sigma'\in\Sigma_0$;

-- the 
tangent bundle of the delimiter set $T{\mathbb P}(D_{\sigma''}\cap\sigma')$ for every non-critical cone $\sigma'\in\Sigma_0$ and its edge $\sigma''$ dual to a $B_2$-facet.

-- the tangent bundle $T{\mathbb P}{\mathbb R}^n_+$.

For every projectivized closed cone $\sigma\in\Sigma_c$, consider the convex piecewise linear (in the sense of the selected affine structure) ``distance function'' $d_\sigma:{\mathbb P}\RR^n_+\to\RR_+$ in the sense of the norm $h$, i.e. $d_\sigma(u)=\min_{u'\in\sigma}h(u-u')$. Properties 2 and 3 of the norm $h$ above translate into the following properties of the distance function $d_\sigma$ (in order to see how the property 3 of $h$ translates into the properties 2-3 of the corner locus of $d_\sigma$, notice that the affine spans of the polyhedral cells of the corner locus of $h$ are among the planes $H_I$ from its property 3):

1) It vanishes on $\sigma$ and is strictly positive outside of it.

2) For every non-critical cone $\sigma'\in\Sigma_0$, the corner locus of $d_\sigma$ is transversal to the projectivized critical set ${\mathbb P}K_{\sigma'}$ in the complement to $\sigma$.

3) Moreover, for every non-critical cone $\sigma'\in\Sigma_0$ and its edge $\sigma''$ dual to a $B_2$-facet, the corner locus of $d_\sigma$ is transversal to the projectivized delimiter ${\mathbb P}(D_{\sigma''}\cap\sigma')$ and the critical set ${\mathbb P}(D_{\sigma''}\cap K_{\sigma'})$ in the complement to $\sigma$.

In the sense of this distance, we shall consider $\varepsilon$-neighborhoods $B_\sigma(\varepsilon)=\{u\,|\, d_\sigma(u)<\varepsilon\}$ and their boundaries $S_\sigma(\varepsilon)$, which are all piecewise-linear sets.

Associating positive numbers $\varepsilon_\sigma$ to all $\sigma\in\Sigma_c$, introduce the following sets:

-- the open neighborhood $U=\bigcup_{\sigma\in\Sigma_c} B_\sigma(\varepsilon_\sigma)$ of the critical complex $\Sigma_c$;

-- for every $\sigma\in\Sigma_c$, the set $U_\sigma=B_\sigma(\varepsilon_\sigma)\setminus\bigcup_{\sigma'\subsetneq\sigma} B_{\sigma'}(\varepsilon_{\sigma'})$;

-- for every $\sigma\in\Sigma_c$ dual to a $B_2$-facet, the {\it delimiter disk} $DD_\sigma=U_\sigma\cap {\mathbb P}D_\sigma$. 

Note that, by Lemmas
4.9 and 4.28, the cones $\sigma^{\prime} \ne \sigma$
in $DD_{\sigma}$ are not critical.

\begin{lemma}\label{ltubularn} One can choose the numbers $\varepsilon_\sigma$ so that \begin{enumerate}
\item for every non-critical projectivized cone $\sigma'\in{\mathbb P}\Sigma_0$, no vertex of the boundary of 
$U_\sigma\cap\sigma'$ and $DD_\sigma\cap\sigma'$ is contained in the critical set ${\mathbb P}K_{\sigma'}\cap\relint\sigma'$;

\item $U_\sigma$ is contained in the star of $\sigma$;
\item $U_\sigma\cap U_\delta=\varnothing$ unless $\sigma=\delta$, and $\bar U_\sigma\cap \bar U_\delta=\varnothing$ unless $\sigma\subset\delta$ or vice versa;
\item $DD_\sigma$ divides $U_\sigma$ into two connected components. 
\end{enumerate}
\end{lemma}
\begin{proof} Properties (2-4) are satisfied if the tuple $(\varepsilon_\sigma)$ is chosen 
rapidly decreasing (i.e.  $\varepsilon_\sigma\ll\varepsilon_\delta\ll 1$ for all $\delta\subset\sigma$), and even without genericity assumptions (2-3) on the distance function $d_\sigma$.

If the tuple is moreover chosen generically (i.e. avoiding finitely many hyperplanes in the space of all such tuples), then it satisfies (1) as well by the properties (2-3) of the corner locus of $d_\sigma$, because the vertices of $U_\sigma\cap\sigma'$ belong to $\sigma'\cap\bigl((n+1-\dim\sigma')$-dimensional skeleton of the corner locus of $d_\sigma\bigr)$, and the same for delimiter discs.
\end{proof}

\begin{definition}\label{defvf4} Under the assumptions of Theorem \ref{ADJJJ}, the {\it vertex function} $$v:U\setminus(\mbox{delimiter disks})\to(\mbox{vertices of }\Gamma_+(f))$$ is defined on every $U_\sigma$ as follows:

-- if $\sigma\in\Sigma_c$ is dual to a non-$B_2$ critical face $F$, then we define $v(\cdot)$ on $U_\sigma$ as the apex $P_F$. 

-- if $\sigma\in\Sigma_c$ is dual to a $B_2$-facet $\tau$, then we define $v(\cdot)$ on each of the two components of $U_\sigma\setminus DD_\sigma$ as the apex of $\tau$ on the side of this component.
\end{definition}
\begin{example}
If $\sigma$ is dual to the $B_2$-facet $\tau$ from Figure \ref{picb2}, then $U_\sigma$ is split with the delimiter hyperplane (perpendicular to the segment $XY$) into two components -- ``left'' and ``right''. Then, in the setting of Example \ref{exaontheside}, the vertex function $v$ equals $P$ on the left component and $B$ on the right one.
\end{example}

Note that the vertex function is locally constant on its domain.
We now use this observation to consistently substitute every piece of the neighborhood $U$ by an appropriate sprout. Recall that we refer to cones and their projectivizations interchangeably, and, in particular,  $S_{W,\tau}$ for the projectivization $W$ of a cone $C$ is another notation for the sprout $S_{C,\tau}$.

-- If $\sigma\in\Sigma_c$ is dual to a face $F$ that is not a $B_2$-facet, then the vertex function $v$ equals a constant $P$ on it, so we define $V_{\sigma,0}$ as the sprout $S_{U_\sigma,\tau}$ (see Definition \ref{defsprout}) for any $B$-facet $\tau\supset F$. Note that neither the sprout $V_{\sigma,0}:=S_{U_\sigma,\tau}$ nor its root $R_{\sigma,0}:=R_{U_\sigma,\tau}=P$ depend on the choice of $\tau$.

-- If $\sigma\in\Sigma_c$ is dual to a $B_2$-facet $\tau$ with the apices $P_{\pm 1}$, then we define $V_{\sigma,\pm 1}$ as the sprout $S_{\{v=P_{\pm 1}\}\cap U_\sigma,\tau}$ and the root $R_{\sigma,\pm 1}$ as $R_{\{v=P_{\pm 1}\}\cap U_\sigma,\tau}=P_{\pm 1}$. Also, choosing a ray $r$ in the dual cone $\sigma'$ of the non-simplicial non-$V$-facet of $\tau$ outside of $K_{\sigma'}$, we define $V_{\sigma,0}$ as $S_{r,\tau}$ (see Lemma \ref{ldsproutn}), leaving $R_{\sigma,0}$ undefined. 

Define the {\it sprouting} $S_P\subset\RR^n$ of a vertex $P$ as the union of all $V_{\sigma,\ast}\; (\ast=0,\pm 1)$, such that $R_{\sigma,*}=P$.

\begin{lemma} \label{lproofmainBn} \begin{enumerate}
\item The contribution $Z(S_P)(s)$ of the 
sprouting $S_P$ to the zeta function $Z_{top, f}(s)$ is equal to $$\int_{S_P} \exp\bigl(-\langle u, P\rangle s-\langle u, \mbbo\rangle\bigr)\, du$$ modulo a function that has no pole at $s_0$.
\item No edge of the boundary of $S_P$ is critical for $(s_0,P)$.
\end{enumerate}
\end{lemma}
\begin{proof} The part (1) follows from Lemmas \ref{lsprout} and 4.22. To deduce (2), it is enough (by Lemma \ref{sproutrays}) to show that every vertex of $(\partial U_\sigma)\cap\sigma'$ and $(\partial DD_\sigma)\cap\sigma'$ for every projectivized cone $\sigma'\supset\sigma$ is either not in ${\mathbb P}K_{\sigma'}$, or not a projectivized edge of $\partial S_P$. For non-critical $\sigma'$, this follows from Lemma \ref{ltubularn}(1), and every critical $\sigma'$ is either in the projectivized interior of $S_P$, or disjoint from its closure, or intersects its boundary at an interior point of its facet $DD_{\sigma'}$ (thus no vertex in $\sigma'$ can be a projectivized edge of $\partial S_P$).
\end{proof}

{\it Proof of Theorem \ref{ADJJJ}.} By the preceding lemma and Lemma \ref{regintegr}, the contribution of  every $S_P$ to the topological zeta function of $f$ has no pole at $s_0$. By Lemma \ref{ldsproutn}, the same is true for $V_{\sigma,0}$ for every $B_2$-facet $\sigma$. Since $$V=\bigcup_{\sigma,\; \ast=0,\pm 1} V_{\sigma,\ast}$$ contains all the critical cones in its interior, for every cone $\sigma'\in\Sigma_0$ the set $\sigma'\setminus V$ has edges of two types: either edges of $(\partial V)\cap\sigma'$ or non-critical 1-dimensional cones of $\Sigma_0$. The edges of the first kind are not in the critical set of any cone by Lemma \ref{lproofmainBn}(2), and for the second kind the same holds by definition. Thus the contribution of $\sigma'\setminus V$ to the topological zeta function has no pole at $s_0$ as well by Lemma \ref{regintegr}. We have subdivided $\RR^n_+$ into the pieces 

-- $S_P$ for some vertices $P$, 

-- $V_{\tau^\circ,0}$ for some $B_2$-facets $\tau$ ,

-- $\sigma'\setminus V$ for some non-critical cones $\sigma'$,

\noindent so that none of them contributes the pole $s_0$.
$\hfill \square$

\section{Fake poles of the topological zeta function 
in dimension 4}\label{sec:4new}

Throughout this section we work in dimension $n=4$. In the first subsection, we classify all non-contributing configurations of $B$-faces. The proof of the classification theorem occupies the subsequent two subsections and makes use of the tools introduced in the preceding two sections. In the last subsection, we show that every non-$B$-facet contains a non-$B$-simplex and discuss a possible general definition of $B$-facets in arbitrary dimension. Both of the mentioned results will be used in the last section to prove the monodromy conjecture for all non-degenerate singularities of functions of $4$ variables.

\subsection{The main theorem}

We shall prove that if a candidate 
pole of the topological zeta function 
$Z_{ {\rm top}, f}(s)$ is contributed only by $B$-facets, 
then it is fake, with one exception:

\begin{definition}\label{defborder} 1. A \emph{border} is a triangular face of $\Gamma_+(f)$, such that, 
up to a reordering of the 
coordinates, its vertices are of the form $A=(1,1,*,*)$, 
$B=(0,0,*,*)$ and  $D=(0,0,*,*)$, and the two facets containing 
it are $B$-facets with the vertex $A$ and the bases 
in the coordinate hyperplanes 
$\{ v_1=0 \}$ and $\{ v_2=0 \}$ respectively. 

In this definition we admit ``infinite triangles'' obtained by tending some of the starred coordinates to infinity. Namely, the notion of the border includes the Minkowski sum of the segment $AD$ with the aforementioned coordinates and the 4th coordinate ray, as well as the  Minkowski sum of the point $A=(1,1,*,*)$ and the 3rd and 4th coordinate rays.

2. The border $ABD$ is said to be a $B$-border unless (up to reordering $B$ and $D$ and the last two coordinates) we have $A=(1,1,0,a)$, $B=(0,0,0,b)$ and $D=(0,0,1,d)$. In the latter case  (implying, in particular, that the edge $BD$ is itself a $B$-edge in the coordinate plane $O_{34}$) the border $ABD$ is said to be  a $B^2$-border (see Figure \ref{picborder}). The vertex $A$ is called the apex of the border, and the first two coordinates are called its bases.

In this definition we admit ``infinite triangles'' obtained by tending $b$ to infinity, i.e. the Minkowski sum of the segment $AD$ with the aforementioned coordinates and the 4th coordinate ray.
\end{definition}

\begin{figure}
    \centering
\skipfig{\psscalebox{1.0 1.0} 
{
\begin{pspicture}(0,-1.62)(9.696568,1.62)
\psline[linecolor=black, linewidth=0.04](1.2482843,-0.4)(0.048284303,-1.6)(2.8482842,-1.6)(4.0482845,-0.4)(1.2482843,-0.4)(1.2482843,1.6)(4.0482845,1.6)(4.0482845,-0.4)
\rput[bl](0.2482843,0.0){$A$}
\rput[bl](1.2482843,-0.8){$B$}
\rput[bl](3.0482843,-0.2){$D$}
\psline[linecolor=black, linewidth=0.04](5.2482843,-0.4)(4.0482845,-1.6)(8.448284,-1.6)(9.648284,-0.4)(5.2482843,-0.4)(5.2482843,1.6)(9.648284,1.6)(9.648284,-0.4)
\psline[linecolor=black, linewidth=0.08](6.0482845,-0.4)(7.2482843,1.2)(6.8482842,0.4)
\psline[linecolor=black, linewidth=0.08](7.2482843,1.2)(8.448284,-0.4)
\rput[bl](7.0482845,0.4){$A$}
\rput[bl](5.8482842,-0.2){$B$}
\rput[bl](8.448284,-0.2){$D$}
\psline[linecolor=black, linewidth=0.04](1.2482843,1.6)(0.048284303,0.4)(0.048284303,-1.6)
\psline[linecolor=black, linewidth=0.08](1.2482843,1.0)(0.6482843,0.2)(0.4482843,-1.2)(3.0482843,-0.4)(0.6482843,0.2)
\psline[linecolor=black, linewidth=0.08](3.0482843,-0.4)(1.2482843,1.0)
\psline[linecolor=black, linewidth=0.08, linestyle=dashed, dash=0.17638889cm 0.10583334cm](1.2482843,-0.4)(0.6482843,0.2)
\psline[linecolor=white, linewidth=0.04, linestyle=dashed, dash=0.17638889cm 0.10583334cm](1.2482843,-0.4)(2.8482842,-0.4)
\psline[linecolor=white, linewidth=0.04, linestyle=dashed, dash=0.17638889cm 0.10583334cm](6.0482845,-0.4)(8.448284,-0.4)
\psline[linecolor=black, linewidth=0.08, dotsize=0.07055555cm 2.0,dotsize=0.07055555cm 2.0]{cc-cc}(6.8482842,0.4)(6.0482845,-0.4)(6.448284,-1.2)(6.8482842,0.4)(8.448284,-0.4)(6.448284,-1.2)
\end{pspicture}
}}
    \caption{a $B^2$-border and a $B$-border}
    \label{picborder}
\end{figure}
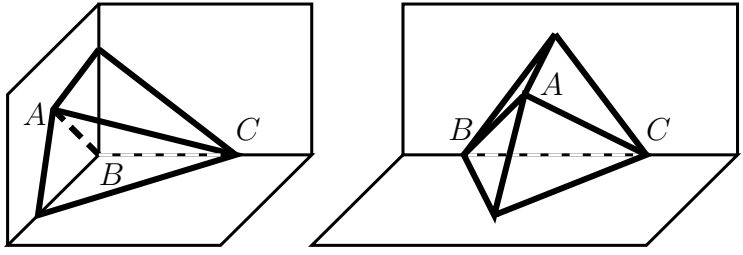

Let $f$ be a non-degenerate polynomial on $\CC^4$ with 
the Newton polyhedron $\Gamma_+(f) \subset \mathbb{R}^4$ 
and $\Sigma_0$ its dual fan. 

\begin{theorem} \label{mainB}
Assume that $f$ is non-degenerate and does not have a 
Morse singularity at the origin $0 \in \CC^4$. Let a candidate pole 
$s_0 \ne-1$  of the topological zeta function 
$Z_{ {\rm top}, f}(s)$ 
be contributed only by $B$-facets, 
and no two of them contain the same $B$-border (although they may contain the same $B^2$-border). 
Then $s_0$ is fake i.e. not an actual 
pole of $Z_{ {\rm top}, f}(s)$. 
\end{theorem}

The rest of this subsection is devoted to 
the proof of this theorem.

The first difference with the setting of Theorem \ref{ADJJJ} is that we have to prove that $B^2$-borders do not contribute to candidate poles. For this purpose, we shall  extend the notion of a sprout (Definition \ref{defsprout}) to them. Let $C$ be a (not necessarily convex) cone in the star of $\tau^\circ$, where $\tau$ is a border with bases $i$ and $j$ and the apex $P$, adjacent to two $B$-facets contributing the pole $s_0$. 

Note that we may assume w.l.o.g. throughout the rest of this section that $\tau$ is compact: indeed, if two $B$-facets sharing a common non-compact $B^2$-border contribute the same pole $s_0$, then $s_0=-2$.
\begin{definition}\label{defbordconv} The cone $C$ is said to be \emph{border-convex}, if every 3-dimensional plane through $e_i$ and $e_j$ intersects $C$ by a convex cone.
\end{definition}

\begin{definition} \label{defbsprout} The union of $C\cap\relint P^\circ$ and all three-dimensional cones, generated by the coordinate vectors $e_i, e_j$ and $u\in C\cap \partial P^\circ$, is denoted by $S_{C,\tau}$ and is called the \emph{border sprout of $C$}. The vertex $P$ is denoted by $R_{C,\tau}$ and is called the \emph{root of $C$}. \end{definition}
The definition implies the following.
\begin{lemma} \label{sproutrays4} The edges of a border sprout of a cone $C$ are either edges of $(\partial C)\cap\sigma^{\circ}$ for some faces $\sigma$ of $\tau$ outside of coordinate planes, or coordinate rays.
\end{lemma}
%\begin{proof}
%Let $S$ be the sprout, and let $A,B,D$ be the three vertices of the border. Then the edges of $S$ are among the edges of $S\cap A^\circ, S\cap D^\circ,$ and $\cap D^\circ$. The edges of $S\cap D^\circ$ have the sought form, because $S\cap D^\circ$ is generated by two coordinate rays and $C\cap DA^\circ$. The same for $S\cap B^\circ$, and the edges of $S\cap A^\circ$ are tautologically of the sought form.
%\end{proof}
\begin{lemma} \label{lbsprout} Let $C$ be a border-convex cone in the star of the dual cone to a $B^2$-border $\tau$, and assume that both of its adjacent $B$-facets contribute to the same candidate pole $s_0\ne -1$. 
Then the contribution of the border sprout $Z(S_{C,\tau})$ is equal to $$\int_{S_{C,\tau}} \exp\bigl(-\langle u, R_{C,\tau}\rangle s-\langle u, \mbbo\rangle\bigr)\, du$$ modulo a function that has no pole at $s=s_0$.
\end{lemma}
\begin{remark} Most of borders do not satisfy this property.
\end{remark}
\begin{proof} Let $\tau=ABD$ be a $B^2$-border with coordinates $A=(1,1,0,a)$, 
$B=(0,0,0,b)$ and $D=(0,0,1,d)$. The fact that the adjacent $B$-facets $\tau_1$ and $\tau_2$ contribute the same candidate pole means that the line $s\cdot A+{\bf 1}$ intersects the vector spans of $\tau_1-A$ and $\tau_2-A$ at the same point $s_0\cdot A+{\bf 1}$ (note that the
dimension of the intersection of these two
3-dimensional linear subspaces is 2).
Thus this point is a linear combination of the vectors $D-B$ and $A-B$: $$D-B=\beta(s_0\cdot A+{\bf 1})+\alpha\cdot (A-B).$$
Solving this system of equations for $\alpha, \beta$ and $s_0$, we find
$$s_0=(a+d-2b-1)/b,\;\; \alpha=-s_0-1,\;\; \beta=1.$$
At the same time, for the future reference, we interpret the latter equalities as the computation of the ratio of $D-B$ and $s_0\cdot A+{\bf 1}$ as linear functions on the dual space $(\RR^4)^*$ restricted on the hyperplanes $(A-D)^\perp$ and $(A-B)^\perp$:
$$ \frac{(D-B)|_{(A-B)^\perp}}{(s_0\cdot A+{\bf 1})|_{(A-B)^\perp}}=\beta=1,\quad \frac{(D-B)|_{(A-D)^\perp}}{(s_0\cdot A+{\bf 1})|_{(A-D)^\perp}}=\beta/(1-\alpha)=1/(2+s_0).\eqno{(*)}$$

We now first prove the lemma for a very special choice of the cone $C$. 

Choose primitive vectors $u_A$, $u_{B}$ and $u_{D}$ in the interior of the cones $ABD^\circ$, $AB^\circ$ and $AD^\circ$ respectively. Recall that the standard basis is denoted by $e_1,e_2,e_3,e_4$.
Denote 
the union of the relatively open simplicial cones $\langle u_A,u_{B}, e_1,e_2\rangle$,  $\langle u_A,u_{D}, e_1,e_2\rangle$ and $\langle u_A, e_1,e_2\rangle$ by $N_{ABD}$. 

Then the statement of the lemma is valid for $C=N_{ABD}$. To prove this, split
$Z(N_{ABD})(s)-\int_{N_{ABD}} \exp\bigl(-s\langle u, A\rangle-\langle u,
\mbbo\rangle\bigr)\, du$ into $$Z(\langle u_A,u_{B}, e_1,e_2\rangle)+Z(\langle u_A,u_{D}, e_1,e_2\rangle)+Z(\langle u_A, e_1,e_2\rangle)-$$ $$\int_{\langle u_A,u_{B}, e_1,e_2\rangle} \exp\bigl(-s\langle u, A\rangle-\langle u,
\mbbo\rangle\bigr)\, du-\int_{\langle u_A,u_{D}, e_1,e_2\rangle} \exp\bigl(-s\langle u, A\rangle-\langle u,
\mbbo\rangle\bigr)\, du,$$ and apply Lemma \ref{DLLE} to each of the five terms. The result consists of the five respective terms  

$$\frac{|\det(e_1,e_2,u_A,u_D)|}{(s\cdot A+{\bf 1})|_{u_A}\cdot (s\cdot A+{\bf 1})|_{u_D}}+\frac{|\det(e_1,e_2,u_A,u_B)|}{(s\cdot A+{\bf 1})|_{u_A}\cdot (s\cdot A+{\bf 1})|_{u_B}}-\frac{s\cdot |e_1\wedge e_2\wedge u_A|}{(s+1)\cdot(s\cdot A+{\bf 1})|_{u_A}}-$$
$$\frac{|\det(e_1,e_2,u_A,u_D)|}{(s+1)^2\cdot(s\cdot A+{\bf 1})|_{u_A}\cdot (s\cdot A+{\bf 1})|_{u_D}}-\frac{|\det(e_1,e_2,u_A,u_B)|}{(s+1)^2\cdot(s\cdot A+{\bf 1})|_{u_A}\cdot (s\cdot A+{\bf 1})|_{u_B}}, \eqno{(**)}
$$
where $|\cdot|$ in the third term stands for the lattice length. Collecting similar terms and taking into account the identities $\det(e_1,e_2,u_A,x)=(e_1 \wedge e_2 \wedge u_A) \cdot x$ and $u_A\cdot (D-B)=0$ (which in coordinates $u_A=(u_1,u_2,u_3,u_4)$ reads as $u_3=u_4\cdot(b-d)$), we can rewrite $(**)$ as the uninteresting factor $\frac{s\cdot u_4}{(s+1)^2\cdot(s\cdot A+{\bf 1})|_{u_A}}$ times

$$
\frac{(s+2)\bigl|(D-B)|_{u_D}\bigr|}{(s\cdot A+{\bf 1})|_{u_D}}+\frac{(s+2)\bigl|(D-B)|_{u_B}\bigr|}{(s\cdot A+{\bf 1})|_{u_B}}-(s+1) \eqno{(\star)}
$$

Since $u_B$ and $u_D$ are chosen to be support vectors of the edges $AB$ and $AD$ respectively, we have $u_B\cdot A=u_B\cdot B<u_B\cdot D$ and $u_D\cdot A=u_D\cdot D<u_D\cdot B$, so $\bigl|(D-B)|_{u_B}\bigr|=(D-B)|_{u_B}$ and $\bigl|(D-B)|_{u_D}\bigr|=-(D-B)|_{u_D}$.
Applying this and $(*)$, we conclude that $(\star)$ for $s=s_0$ equals $-(s_0+2)/(s_0+2)+(s_0+2)-(s_0+1)=0$. 

As a result, the sought difference $Z(N_{ABD})(s)-\int_{N_{ABD}} \exp\bigl(-s\langle u, A\rangle-\langle u,
\mbbo\rangle\bigr)\, du$ equals the product of the uninteresting factor that has a simple pole at $s=s_0$ and the rational (actually linear) function $(\star)$ that has a root at $s=s_0$. Thus the product is holomorphic at $s=s_0$, and we have proved the lemma for $C=N_{ABD}$.

Now, for an arbitrary border-convex $C$, the border sprout $S_{C,\tau}$, can be represented as the union (with disjoint interiors) of $N_{ABD}$ and the sprouts of the form $S_{C^{(k)},\tau_k},\, k=1,2$, for certain cones $C^{(k)}$. Thus the statement of the lemma is valid for it, applying the preceding computation to $N_{ABD}$ and Lemma \ref{lsprout} to $S_{C^{(k)},\tau_k}$. 
\end{proof}

We now comment on how one might arrive at considering the cone $N_{ABD}$ in this proof. We do not know how to prove the lemma directly for an arbitrary sprout $S$ of the border $\tau$, but we already know similar identities for (non-border) sprouts $S_i$ of the two adjacent $B$-facets $\tau_i\supset\tau$. So it is a natural idea to try to simplify the problem, subtracting from $S$ non-overlapping sprouts $S_i$, trying to choose them so that the difference is as small and simple as possible. Now it is an easy exercise of spatial thinking to see that the smallest and simplest difference has the form $N_{ABD}$.

\subsection{Very critical cones}

In contrast to the setting of Theorem \ref{ADJJJ}, we cannot in general choose preferred bases and apices of $B_1$-faces consistently, so we now choose them arbitrarily once and for all in this section. Recall that we fix the choice of apices $P_\tau$ and preferred bases $b_\tau$ for all $B_1$-facets $\tau$ and for all $B_1$-facets of $B_2$-facets $\tau$, and then use them to define preferred bases and apices of $B_2$-facets on the side of a given face or cone, see Subsection \ref{ssbases} for details.

As a result, in contrast to the setting of Theorem \ref{ADJJJ}, different $B_1$-facets, containing a given critical face, may have different preferred bases. Such critical faces are said to be very critical. They are studied in this subsection.

\begin{definition} \label{defbase} The \emph{preferred base 
$b_F=b_{F^{\circ}} \subset \{ 1,2,3,4 \}$ of a 
critical face $F \prec \Gamma_+(f)$ and its 
dual cone $F^{\circ} \in \Sigma_0$} is the 
set of the preferred bases on the side of $F$ for all 
$B$-facets containing $F$. 
We say that the face $F$ and its dual cone 
$F^\circ \in \Sigma_0$ 
are very critical if $|b_F| \geq 2$.  
\end{definition}

\begin{example} The projectivized pictures of the 
Newton polyhedron $\Gamma_+(f)$ on Figure \ref{picverycrit} below give some examples of 
($2$-dimensional) very critical faces 
(hatched), provided that all 
the facets on the pictures are $B$-facets contributing the 
same candidate pole $s_0$. The apices of $B_1$-faces are bold points, and the apices of the $B_2$-facet are the end points of the bold segment. \end{example}
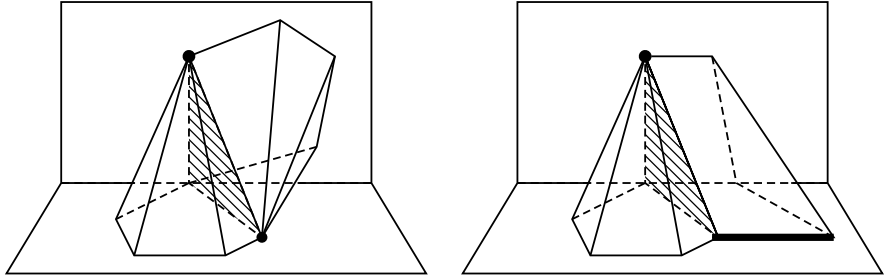
\begin{figure}[h]
\begin{center}
\skipfig{\psscalebox{0.6 0.6} 
{
\begin{pspicture}(0,-3.02)(19.27061,3.02)
\pspolygon[linecolor=black, linewidth=0.04, linestyle=dashed, dash=0.17638889cm 0.10583334cm, fillstyle=hlines, hatchwidth=0.028222222, hatchangle=-45.0, hatchsep=0.1411111](4.035305,1.8)(5.635305,-2.2)(4.035305,-1.0)
\psline[linecolor=black, linewidth=0.04, dotsize=0.2cm 2.0]{*-}(4.035305,1.8)(5.635305,-2.2)(4.8353047,-2.6)(2.835305,-2.6)(2.4353049,-1.8)(4.035305,1.8)(6.035305,2.6)(7.235305,1.8)(5.635305,-2.2)
\psline[linecolor=black, linewidth=0.04](4.035305,1.8)(2.835305,-2.6)
\psline[linecolor=black, linewidth=0.04](4.035305,1.8)(4.8353047,-2.6)
\psline[linecolor=black, linewidth=0.04, dotsize=0.2cm 1.0]{*-}(5.635305,-2.2)(6.035305,2.6)
\psline[linecolor=black, linewidth=0.04](7.235305,1.8)(6.8353047,-0.2)(5.635305,-2.2)
\psline[linecolor=black, linewidth=0.04, linestyle=dashed, dash=0.17638889cm 0.10583334cm](2.4353049,-1.8)(4.035305,-1.0)
\psline[linecolor=black, linewidth=0.04, linestyle=dashed, dash=0.17638889cm 0.10583334cm](4.035305,-1.0)(6.8353047,-0.2)
\psline[linecolor=black, linewidth=0.04, linestyle=dashed, dash=0.17638889cm 0.10583334cm](1.235305,-1.0)(8.035305,-1.0)
\psline[linecolor=black, linewidth=0.04](1.235305,-1.0)(2.835305,-1.0)
\psline[linecolor=black, linewidth=0.04](8.035305,-1.0)(6.435305,-1.0)
\pspolygon[linecolor=black, linewidth=0.04](1.235305,3.0)(1.235305,-1.0)(0.03530495,-3.0)(9.235305,-3.0)(8.035305,-1.0)(8.035305,3.0)
\pspolygon[linecolor=black, linewidth=0.04, linestyle=dashed, dash=0.17638889cm 0.10583334cm, fillstyle=hlines, hatchwidth=0.028222222, hatchangle=-45.0, hatchsep=0.1411111](14.035305,1.8)(15.635305,-2.2)(14.035305,-1.0)
\psline[linecolor=black, linewidth=0.04, dotsize=0.2cm 2.0]{*-}(14.035305,1.8)(15.635305,-2.2)(14.835305,-2.6)(12.835305,-2.6)(12.435305,-1.8)(14.035305,1.8)(14.035305,1.8)(15.635305,-2.2)(15.635305,-2.2)
\psline[linecolor=black, linewidth=0.04](14.035305,1.8)(12.835305,-2.6)
\psline[linecolor=black, linewidth=0.04](14.035305,1.8)(14.835305,-2.6)
\psline[linecolor=black, linewidth=0.04, linestyle=dashed, dash=0.17638889cm 0.10583334cm](12.435305,-1.8)(14.035305,-1.0)
\psline[linecolor=black, linewidth=0.04, linestyle=dashed, dash=0.17638889cm 0.10583334cm](11.235305,-1.0)(18.035305,-1.0)
\psline[linecolor=black, linewidth=0.04](11.235305,-1.0)(12.835305,-1.0)
\psline[linecolor=black, linewidth=0.04](18.035305,-1.0)(17.235306,-1.0)
\pspolygon[linecolor=black, linewidth=0.04](11.235305,3.0)(11.235305,-1.0)(10.035305,-3.0)(19.235306,-3.0)(18.035305,-1.0)(18.035305,3.0)
\psline[linecolor=black, linewidth=0.04](15.501971,-2.2)(18.168638,-2.2)(15.501971,1.8)(14.168638,1.8)
\psline[linecolor=black, linewidth=0.04, linestyle=dashed, dash=0.17638889cm 0.10583334cm](15.501971,1.8)(16.035305,-1.0)(18.168638,-2.2)
\psline[linecolor=black, linewidth=0.16](15.501971,-2.2)(18.168638,-2.2)
\end{pspicture}
}}

\caption{very critical faces}\label{picverycrit}

\end{center}
\end{figure}

This subsection is devoted to the study of very critical 
faces depending on their dimension. First of all, 
by Lemma \ref{nontrivn} there is no critical vertex, 
and a facet $\tau$ is critical if and only if the 
candidate pole $s_0$ is contributed by $\tau$ (and thus $\tau$ is a $B$-facet in the assumptions of Theorem \ref{mainB}). A 2-dimensional face 
$F \prec \Gamma_+(f)$ is critical 
if and only if it separates 
two $B$-facets $\tau_1$ and $\tau_2$, contributing $s_0$. 
This critical face $F$ is 
very critical, if moreover we have 
$b_{\tau_1}^F \ne b_{\tau_2}^F$. 

\begin{definition} \label{defdelim} 
Let $F$ be a $2$-dimensional 
very critical face separating two 
$B$-facets $\tau_1$ and $\tau_2$. Then we 
define the {\it $F$-delimiter} $D_F=D_{F^\circ} \subset \RR^4$ 
to be the vector subspace in $\RR^4$ generated by 
the $b_{\tau_1}^F$ and  $b_{\tau_2}^F$'th coordinate lines in
$\RR^4$
and a ray $R \subset \relint F^\circ$. 

The delimiter is said to be \emph{generic}, if, for every non-critical face $G\subset F$, the critical set $D_F\cap K_G$ is nowhere dense in $D_F\cap \relint G^\circ$.
\end{definition}

\begin{lemma} 
Under the assumptions of Theorem \ref{mainB}, 
any $2$-dimensional very critical face $F$ is a triangle. If it is a border, then the delimiter $D_F$ is the affine span of the dual cone of the $V$-edge of $F$. Otherwise the affine span ${\rm aff}(F^\circ)$ of $F^\circ$ is 
transversal to the delimiter $D_F$. 
In particular, in both cases $D_F$ is a hyperplane in $\RR^4$. 
\end{lemma}
\begin{proof} 
The face $F$ is a triangle by Lemma \ref{Noquadrn}. 
Assume that ${\rm aff}(F^\circ)$ 
is not transversal to the delimiter. 
Then for the coordinate plane 
$H= \{ v_{b_{\tau_1}^F}=v_{b_{\tau_2}^F}=0 \} \subset \RR^4$ 
we have $\dim \{ {\rm aff}(F) \cap H \} \geq 1$. 
This implies that ${\rm aff}(F)$ is not 
transversal to $H$, and hence 
their intersection is a $V$-edge 
in both $\tau_1$ and $\tau_2$. 
Then $F$ is a border. 
\end{proof}
\begin{corollary} \label{genericexists} If $F$ is a non-border very critical face, then the rays $R \subset \relint F^\circ$ in Definition \ref{defdelim} parameterize a one-dimensional family of $F$-delimiters. Among them, all but finitely many delimiters are generic.
\end{corollary}

An edge $F \prec \Gamma_+(f)$ is critical, 
if and only if all of the facets containing it 
are $B$-facets $\tau$ contributing to the 
candidate pole $s_0$. Since 
we have $\dim F^\circ =3$ in this case, moreover, for any 
apex $P$ of such $\tau$ on the side of $F$, 
the critical hyperplane $L_P$ coincides with 
${\rm aff}(F^\circ)$.

\begin{proposition}\label{NEVC} 
Under the assumption of Theorem \ref{mainB}, no edge 
of $\Gamma_+(f)$ is very critical. 
\end{proposition}

\begin{proof} 
Assume that an edge $P_1P_2$ is very critical. 
Then, by Proposition \ref{pcritface}(2), there exist $B$-facets $\tau_1$ and $\tau_2$ containing
it with the apices $Q_1$ and $Q_2$ on the side of $P_1P_2$,
and their critical hyperplanes
$L_{Q_i}:
\langle u,Q_i \rangle s_0+ \langle u,\mbbo \rangle =0$ 
both
coincide with the hyperplane
$\langle u ,P_1-P_2 \rangle =0$
generated by the dual cone $(P_1P_2)^{\circ}$. Moreover,
since $P_1P_2$ is not in a coordinate plane, we have $P_1=Q_1$ and
$P_2=Q_2$ or vice versa. Thus the vectors $P_1s_0+ \mbbo$ and
$P_2s_0+ \mbbo$ are parallel.

On the other hand, it cannot happen that one of the points $P_1$ and
$P_2$ is in the apex of $\tau_1$ and in the base of $\tau_2$, and the
other one is in the apex of $\tau_2$ and in the base of $\tau_1$.
Otherwise, reordering coordinates if necessary, we would have
$P_1=(1,0,*,*)$, $P_2=(0,1,*,*)$, $P_1s_0+ \mbbo =(1+s_0,1,*,*)$,
$P_2s_0+ \mbbo =(1,1+s_0,*,*)$. 
Since the last two vectors are parallel and 
$s_0 \ne 0$, we have $s_0=-2$, 
$P_1s_0+ \mbbo = -(P_2s_0+ \mbbo )$ and hence 
$P_1+P_2= \mbbo$.  Thus, up to
reordering the coordinates and 
$P_i$'s, we have $P_1=(1,0,0,0)$ (i.e. $f$
has no singularity at the origin) or 
$P_1=(1,0,1,0)$ and $P_2=(0,1,0,1)$ and 
hence $f$ has a Morse singularity at the 
origin by Lemma \ref{NDGG}. 

If one of $P_i$'s is in the apex of all $B$-facets 
$\tau_j$'s containing $P_1P_2$, and one of 
$\tau_j$'s is  $B_2$, then the other 
facet containing its quadrilateral face 
$F$ also contains $P_1P_2$,
and thus is also $B_2$ (it cannot be $B_1$, 
because it has a quadrilateral face $F$ 
outside the coordinate hyperplanes). Then we 
would have two $B_2$-facets with a common 
quadrilateral face $F$ (see Lemma \ref{Noquadrn} for a contradiction). 

Thus all of $\tau_j$'s are $B_1$-facets. 
As we have seen above, one of $P_1$ and $P_2$  
is equal to the apices 
of all $\tau_j$'s, and the other one is 
in the bases of all $\tau_j$'s. Then,
among $\tau_j$'s, we can find at least two pairs  of
$B_1$-facets with a common apex, a
common triangular face and different bases. 

These two pairs surround two borders, and,
since there are no $B$-borders by the assumptions 
of Theorem \ref{mainB}, they are $B^2$-borders.
Two $B^2$-borders with a common apex $A$, intersecting in a common edge $AC$ in the interior of $\RR^4_+$,
by their definition have $C=(0,0,1,1)$ and $A=(1,1,0,0)$, which by Lemma \ref{NDGG} implies that the singularity $f$ is Morse non-degenerate.
\end{proof}

\subsection{A tubular neighborhood of the critical subfan revisited}

We are now ready to prove Theorem \ref{mainB}. When referring to $B$-facets or critical faces in the course of the proof, we always imply only the faces contributing to the candidate pole $s_0\ne-1$, for which we are proving Theorem \ref{mainB} (thus all the choices and objects that we introduce for the proof completely depend on the choice of $s_0$). Recall that by this time we have 

-- chosen once and for all a preferred base of every $B_1$-facet and every $B_1$-facet of a $B_2$-facet in the Newton polyhedron $\Gamma_+(f)$, 

-- depending on this choice, called some cones critical in the dual fan $\Sigma_0$ (see Definitions \ref{defcritn} and \ref{defbase} respectively),

-- chosen once and for all a generic delimiter hyperplane $D_\sigma$ for every 2-dimensional very critical cone $\sigma$, dual to a non-border (see Definition \ref{defdelim}), and defined the delimiter hyperplane $D_\sigma$ for every 1-dimensional very critical cone $\sigma$, dual to a $B_2$-facet (see Definition \ref{defdelimb2n}). 
\begin{definition}\label{delimited}
These two kinds of faces and their dual cones will be called {\it delimited}.\end{definition}

Recall that, similarly to Subsection \ref{ssgeomn}, we choose once and for all an affine structure on the projectivization ${\mathbb P}\RR^n_+$ and refer to cones and their projectivizations interchangeably whenever it causes no confusion.

According to Corollary \ref{critclosed}, the set of projectivized critical cones in the dual fan $\Sigma_0$ is a closed polyhedral complex $\Sigma_c$. 
We shall construct a generic piecewise linear tubular neighborhood of $\Sigma_c$, whose boundary is transversal to the critical sets of non-critical cones and delimiters of very critical cones.

For every projectivized closed cone $\sigma\in\Sigma_c$, choose a convex piecewise linear (with respect to the selected affine structure) ``distance function'' $d_\sigma:{\mathbb P}\RR^n_+\to\RR_+$, satisfying the following properties:

1) It vanishes on $\sigma$ and is strictly positive outside of it.

2) For every non-critical cone $\sigma'\in\Sigma_0$, the corner locus of $d_\sigma$ is transversal to the projectivized critical set ${\mathbb P}K_{\sigma'}$ in the complement to $\sigma$.

3) Moreover, for every non-critical cone $\sigma'$ and every delimited face $\sigma''\subset\sigma'$, the corner locus of $d_\sigma$ is transversal to the projectivized delimiter ${\mathbb P}D_{\sigma''}\cap\sigma'$ and the critical set ${\mathbb P}D_{\sigma''}\cap K_{\sigma'}$ in the complement to $\sigma$.

Such a distance function can be constructed from a suitably generic piecewise-linear norm in the same way as in Section 4. For this distance, we shall consider $\varepsilon$-neighborhoods $B_\sigma(\varepsilon)=\{u\,|\, d_\sigma(u)<\varepsilon\}$ and their boundaries $S_\sigma(\varepsilon)$, which are all piecewise-linear sets.

Associating positive numbers $\varepsilon_\sigma$ to all $\sigma\in\Sigma_c$, introduce the following sets:

-- the open neighborhood $U=\bigcup_{\sigma\in\Sigma_c} B_\sigma(\varepsilon_\sigma)$ of the critical complex $\Sigma_c$;

-- for every $\sigma\in\Sigma_c$, the set $U_\sigma=B_\sigma(\varepsilon_\sigma)\setminus\bigcup_{\sigma'\subsetneq\sigma} B_{\sigma'}(\varepsilon_{\sigma'})$;

-- for every delimited $\sigma\in\Sigma_c$, the {\it delimiter disk} $DD_\sigma=U_\sigma\cap {\mathbb P}D_\sigma$.

\begin{lemma}\label{ltubular} One can choose the numbers $\varepsilon_\sigma$ so that
\begin{enumerate}
\item For every projectivized non-critical cone $\sigma'\in{\mathbb P}\Sigma_0$, no vertex of the boundary of 
$U_\sigma\cap\sigma'$ and $DD_\sigma\cap\sigma'$ is contained in the critical set ${\mathbb P}K_{\sigma'}\cap\relint\sigma'$;

\item $U_\sigma$ is contained in the star of $\sigma$ and is border-convex (Definition \ref{defbordconv}) if $\sigma$ is dual to a border;
\item $U_\sigma\cap U_\delta=\varnothing$ unless $\sigma=\delta$, and $\bar U_\sigma\cap \bar U_\delta=\varnothing$ unless $\sigma\subset\delta$ or vice versa;
\item $DD_\sigma$ divides $U_\sigma$ into two connected components. 
\end{enumerate}
\end{lemma}
\begin{proof} All of these properties are satisfied if the tuple $(\varepsilon_\sigma)$ is chosen generically (i.e. avoiding finitely many hyperplanes in the space of all such tuples) and rapidly decreasing (i.e. 
$\varepsilon_\sigma\ll\varepsilon_\delta\ll 1$ for all $\delta\subset\sigma$).
\end{proof}

\begin{definition}[c.f. Definition \ref{defvf4}] The {\it vertex function} $$v:U\setminus(\mbox{delimiter disks})\to(\mbox{vertices of }\Gamma_+(f))$$ is defined on every $U_\sigma$ as follows:
\begin{enumerate}
\item if a projectivized cone $\sigma\in\Sigma_c$ is dual to a non-delimited face $F\subset\Gamma_+(f)$, then all facets containing $F$ are $B$-facets, and their apices on the side of $F$ are all equal to the same vertex $P$. Then we define $v(\cdot)$ on $U_{\sigma}$ as $P$.
\item if $\sigma\in\Sigma_c$ is dual to a delimited triangle $F$, separating two $B$-facets $\tau_1$ and $\tau_2$, then we define $v(\cdot)$ on each of the two components of $U_\sigma\setminus DD_\sigma$ as the apex of the corresponding facet $\tau_i$ on the side of this component.
\item if $\sigma\in\Sigma_c$ is dual to a $B_2$-facet $\tau$, then we define $v(\cdot)$ on each of the two components of $U_\sigma\setminus DD_\sigma$ as the apex of $\tau$ on the side of this component.
\end{enumerate}
\end{definition}

Note that the vertex function is locally constant on its domain (thanks to Proposition \ref{NEVC}).
We now use this observation to consistently substitute every piece of the neighborhood $U$ by an appropriate sprout. Recall that we refer to cones and their projectivizations interchangeably, and, in particular,  $S_{W,\tau}$ for the projectivization $W$ of a cone $C$ is another notation for the sprout $S_{C,\tau}$.

-- If $\sigma\in\Sigma_c$ is neither delimited nor dual to a border, 
then the vertex function $v$ equals a constant $P$ on $\sigma$, so we define $V_{\sigma,0}$ as the sprout $S_{U_\sigma,\tau}$ (see Definition \ref{defsprout}) for any $B$-facet $\tau$ containing the dual face of $\sigma$. Note that neither the sprout $V_{\sigma,0}:=S_{U_\sigma,\tau}$ nor its root $R_{\sigma,0}:=R_{U_\sigma,\tau}=P$ depend on the choice of $\tau$.

-- If $\sigma\in\Sigma_c$ is dual to a border, then the vertex function $v$ equals a constant $P$ on $\sigma$, and we define $V_{\sigma,0}$ as the border sprout $S_{U_\sigma,\sigma}$ (see Definition \ref{defbsprout}) and the root $R_{\sigma,0}$ as $R_{U_\sigma,\tau}=P$.

-- If $\sigma\in\Sigma_c$ is dual to a very critical triangle, separating two $B$-facets $\tau_{\pm 1}$ with vertices $P_{\pm 1}$ on the side of $\sigma$, then we define $V_{\sigma,\pm 1}$ as the sprout  $S_{\{v=P_{\pm 1}\}\cap U_\sigma, \tau_{\pm 1}}$ and the root $R_{\sigma,\pm 1}$ as $R_{U_\sigma,\tau_{\pm 1}}=P_{\pm 1}$.

-- If $\sigma\in\Sigma_c$ is dual to a  $B_2$-facet $\tau$ with the apices $P_{\pm 1}$, then we define $V_{\sigma,\pm 1}$ as the sprout $S_{\{v=P_{\pm 1}\}\cap U_\sigma,\tau}$ and the root $R_{\sigma,\pm 1}$ as $R_{\{v=P_{\pm 1}\}\cap U_\sigma,\tau}=P_{\pm 1}$. Also, choosing a ray $r$ in the dual cone $\sigma'$ to the quadrilateral non-$V$-face of $\tau$ outside of $K_{\sigma'}$, we define $V_{\sigma,0}$ as the delimeter sprout $S_{r,\tau}$ (see Lemma \ref{ldsproutn}), leaving $R_{\sigma,0}$ undefined. 

Thanks to Proposition 5.13, we have no other cases to consider. Now define the {\it sprouting} $S_P$ of a vertex $P$ as the union of all the sprouts $V_{\sigma,\ast}\, (\ast=0,\pm 1)$ such that $R_{\sigma,*}=P$.

We have introduced the same system of notation as in Section \ref{ssgeomn}, giving it meaning in a more general setting (most notably, admitting $B^2$-borders). With this wider meaning of notation at hand, the proof of Theorem \ref{mainB} almost literally repeats the one for Theorem \ref{ADJJJ} (we repeat it for the convenience of the reader).

\begin{lemma} \label{lproofmainB} \begin{enumerate}
\item The contribution of $S_P$ equals $$\int_{S_P} \exp\bigl(-\langle u, P\rangle s-\langle u, \mbbo\rangle\bigr)\, du$$ modulo a function that has no pole at $s_0$.
\item No edge of the boundary of $S_P$ is critical for $(s_0,P)$.
\end{enumerate}
\end{lemma}
\begin{proof} The part (1) follows from Lemmas \ref{lsprout} and \ref{lbsprout}. To deduce (2), it is enough (by Lemma \ref{sproutrays}) to show that every vertex of $(\partial U_\sigma)\cap\sigma'$ and $(\partial DD_\sigma)\cap\sigma'$ for every projectivized cone $\sigma'\supset\sigma$ is either not in ${\mathbb P}K_{\sigma'}$, or not a projectivized edge of $\partial S_P$. For non-critical $\sigma'$, this follows from Lemma \ref{ltubularn}(1), and every critical $\sigma'$ is either in the projectivized interior of $S_P$, or disjoint from its closure, or intersects its boundary at an interior point of its facet $DD_{\sigma'}$ (thus no vertex in $\sigma'$ can be a projectivized edge of $\partial S_P$).

\end{proof}

{\it Proof of Theorem \ref{mainB}.} By the preceding lemma and Lemma \ref{regintegr}, the contribution of  every $S_P$ to the topological $\zeta$-function of $f$ has no pole at $s_0$. By Lemma \ref{ldsproutn}, the same is true for $V_{\sigma,0}$ for every $B_2$-facet $\sigma$. Since $$V=\bigcup_{\sigma,\; \ast=0,\pm 1} V_{\sigma,\ast}$$ contains all the critical cones in its interior, for every cone $\sigma'\in\Sigma_0$ the set $\sigma'\setminus V$ has edges of two types: either edges of $(\partial V)\cap\sigma'$ or non-critical 1-dimensional cones of $\Sigma_0$. The edges of the first kind are not in the critical set of any cone by Lemma \ref{lproofmainBn}(2), and for the second kind the same holds by definition. Thus the contribution of $\sigma'\setminus V$ to the topological zeta function has no pole at $s_0$ as well by Lemma \ref{regintegr}. We have subdivided $\RR^n_+$ into pieces 

-- $S_P$ for some vertices $P$, 

-- $V_{\tau^\circ,0}$ for some $B_2$-facets $\tau$ ,

-- $\sigma'\setminus V$ for some non-critical cones $\sigma'$,

\noindent so that none of them contributes the pole $s_0$.
$\hfill \square$

\subsection{Generalizing the notion of $B$-facets} 

We have seen that a $B_1$ or $B_2$-facet alone never contributes its candidate pole. In Section \ref{sec:7} we shall prove a somewhat complementary fact:
$$ \mbox{\it For $n=4$, all other facets do contribute their (nearby) monodromy eigenvalues.} \eqno{(*)}
$$
(See Section \ref{sec:7} for a precise statement.) This dichotomy is central for the proof of the monodromy conjecture for non-degenerate singularities.

However, for $n>4$, in order to keep the fact $(*)$ true, we should exclude from our consideration $B$-facets in a certain more general sense than the one assumed in Definition \ref{B-F}. What is the proper general notion of a $B$-facet in arbitrary dimension? A possible answer given in  Definition \ref{defbfacet0}.2 is based on the following lemma that we need in order to prove the fact $(*)$.

\begin{lemma}\label{ALEX} 
For $n=4$ if a compact facet 
$\tau \prec \Gamma_+(f)$ is not 
a $B$-facet, then 
it splits into lattice simplices (with no new vertices)
so that one of the simplices is not of type $B_1$. 

Equivalently: if every four affinely independent vertices of a compact facet $\tau \prec \Gamma_+(f)$ form a $B_1$-simplex, then $\tau$ is a $B$-facet.
\end{lemma}
{\it Proof.} 
The facet $\tau$ contains a face $F$ 
not contained in a coordinate 
hyperplane. Note the following facts about every such $F$:
\medskip 
\par \noindent 
{\bf (1)}: $F$ has at most $4$ vertices. 
Otherwise it contains a triangle whose sides 
are not in coordinate hyperplanes, and the union 
of this triangle and any vertex 
of $\tau \setminus F$ gives a non-$B_1$-simplex 
in $\tau$.
\medskip 
\par \noindent 
{\bf (2)}: If $F$ is a quadrilateral, 
then some pair of its opposite edges are 
contained in coordinate hyperplanes, say, 
$\{ v_1=0 \}$ and $\{ v_2=0 \}$, otherwise we get 
the same contradiction as in {\bf (1)}.
In this case, if a vertex of $F$ at 
the hyperplane $\{ v_1=0 \}$ has $v_2>1$, then 
this vertex, the two vertices of 
$F \cap \{ v_2=0 \}$ and any other vertex of 
$\tau \setminus F$ form a non-$B_1$-simplex 
in $\tau$. Thus both vertices of 
$F$ in the hyperplane $\{ v_1=0 \}$ have $v_2=1$ 
and vice versa. Thus $\tau$ is a $B_2$-facet. 
\medskip 
\par \noindent 
{\bf (3)}: If $F$ is a triangle, 
then at least one of its edges is contained in a 
coordinate hyperplane, otherwise we get 
the same contradiction as in {\bf (1)}. 
\medskip 
\par \noindent 
{\bf (3.1)}:  If the triangle $F$ has exactly one edge in a coordinate hyperplane, say, 
$\{ v_1=0 \}$, then the coordinate $v_1$ of 
the other vertex of $F$ equals $1$, 
otherwise $F$ together with any vertex 
from $\tau \setminus F$ form a 
non-$B_1$-simplex in $\tau$. 
Also in this case, all other vertices 
of $\tau$ should be in the hyperplane 
$\{ v_1=0 \}$, because otherwise such vertex 
together with $F$ form a non-$B_1$-simplex 
in $\tau$. 
Thus $\tau$ is a $B_1$-pyramid for $v_1$. 
\medskip 
\par \noindent 
{\bf (3.2)}: If the triangle $F$ has all 
three edges in coordinate hyperplanes, then 
denote by $V$ the set of the vertices of $\tau$ 
outside $F$. Note that every point of $V$ is in a coordinate hyperplane, otherwise it would form a non-$B_1$-simplex together with $F$.
\medskip
\par \noindent 
{\bf (3.2.1)}: If $V$ is contained in a coordinate 
hyperplane $L$, containing one of the edges of $F$, 
then $\tau$ has one vertex $v$ outside $L$. So, 
depending on the distance of $v$ to $L$, the 
facet $\tau$ either contains a non-$B_1$-simplex, 
or is itself a $B_1$-pyramid with the base $L$.
\medskip
\par \noindent 
{\bf (3.2.2)}: If $V$ has a point in each of the three coordinate 2-planes containing the vertices of $F$, then w.l.o.g. these three points are outside the common coordinate edge (say, $v_1=v_2=v_3=0$) of the three 2-planes (otherwise we would arrive at {\bf (3.2.1)}). Since two vertices of a Newton diagram in a 2-plane cannot have the same coordinate, either the vertex of $F$ or a point of $V$ in the $(v_1v_2)$-plane has $v_3\ne 1$. This point, two similar points with respective non-unit coordinates in the other 2-planes, and one of the remaining vertices of $\tau$ form a non-$B_1$-simplex.
\medskip
\par \noindent 
{\bf (3.2.3)}: If $V$ has a point in a coordinate 2-plane (say, $v_1=v_2=0$) containing a vertex of $F$, and a point in the coordinate 3-plane (say, $v_3=0$) containing the opposite edge of $F$, then w.l.o.g. these two points do not belong to smaller coordinate planes (otherwise we would arrive at {\bf (3.2.1)}). Since two vertices of a Newton diagram in a 2-plane cannot have the same coordinate, either the vertex of $F$ or a point of $V$ in $v_1=v_2=0$ has $v_3\ne 1$. This point, together with two vertices of $F$ and one point of $V$ in $v_3=0$, form a non-$B_1$-simplex.
\medskip
\par \noindent 
{\bf (3.2.4)}: It remains to consider the case when $V$ has a point in each of at least two coordinate 3-planes (say, $v_1=0$ and $v_2=0$) containing the edges of $F$, and each of these two points is not contained in a smaller coordinate plane (otherwise we would arrive at {\bf (3.2.1-3)}). Then the two mentioned points of $V$ and the two vertices of $F$ outside of $v_1=v_2=0$ form a non-$B_1$-simplex (they cannot form a quadrilateral, otherwise we would arrive at {\bf (2)}).
\medskip 
\par \noindent 
{\bf (3.3)}: The only remaining case is 
that $F$ is a triangle, exactly two of 
whose faces are in coordinate hyperplanes. 
Since its third edge is not in a 
coordinate hyperplane, then it should be an edge 
of another $2$-dimensional face 
$G$ of $\tau$ not contained in a coordinate hyperplane. 
\medskip
\par \noindent 
{\bf (3.3.1)}: If $G$ is also a triangle, exactly two of 
whose faces are in coordinate hyperplanes, then 
the convex hull of $F\cup G$ has a triangular face $F'$, 
whose edges are in three different coordinate hyperplanes. 
So this case can be done in the same way as {\bf (3.2)} 
(although $F'$ is not necessarily a face of $\tau$, we can 
still consider the set $V$ of all the vertices of $\tau$ 
outside $F'$ and proceed as in {\bf (3.2)}). 
\medskip
\par \noindent 
{\bf (3.3.2)}: Otherwise, $G$ is of one of the types {\bf (2)} or
{\bf (3.1)}, and thus $\tau$ is $B_1$ or 
$B_2$ as shown in the corresponding 
paragraphs. 
\hfill{}$\square$

\section{Eigenvalues of monodromy and corners}\label{sec:4}

In the first subsection, we formulate the main result of this section: 
certain configurations of $V$-faces (so called corners) always 
contribute a non-zero multiplicity of the expected sign to the 
corresponding monodromy eigenvalue. As a corollary, we prove the 
monodromy conjecture for a large class of Newton-non-degenerate 
singularities in arbitrary dimension.

The rest of the section is devoted to the proof of the main result. 
In particular, in the second subsection we introduce the notion 
of a hypermodular function, which may be of independent 
interest for convex geometry and analysis.

\subsection{Motivation and results}

Let $f(x_1,\ldots,x_n)$ be a polynomial 
on $\CC^n$ such that $f(0)=0$. For lattice 
simplices $\tau$ contained in compact facets of 
$\Gamma_+(f)$ we define their V-faces and 
polynomials $\zeta_{\tau} (t) = 
(1-t^{N( \tau )})^{\Vol_{\ZZ}( \tau )} 
\in \CC [t]$ in the same way as for faces of $\Gamma_+(f)$. 

Let us first observe the following fact. 

\begin{proposition}\label{Key-1} 
Let $\tau \prec \Gamma_+(f)$ be a compact 
facet such that $\gamma = \tau \cap \{ v_i=0 \}$ 
is one of its facets. Then 
$F_{\tau, \gamma}(t):= 
\zeta_{\tau}(t)/\zeta_{\gamma}(t) 
\in \CC (t)$ is a polynomial of $t$. If we 
assume moreover that $\tau$ is 
not a $B_1$-pyramid for the 
variable $v_i$, then the complex number 
\begin{equation*}
\lambda = \exp 
\left(-2 \pi i \frac{\nu ( \tau )}{N ( \tau )} 
\right) \in \CC
\end{equation*}
is a root of the polynomial. 
\end{proposition} 

\begin{proof} 
By Lemma \ref{LLT} we can easily prove that 
$F_{\tau, \gamma} = \zeta_{\tau} / \zeta_{\gamma} 
\in \CC (t)$ is a polynomial. Let us prove 
the remaining assertion. If $\tau$ is 
not a pyramid over $\gamma = \tau \cap \{ v_i=0 \}$, 
then we have $\Vol_{\ZZ} ( \tau ) > 
\Vol_{\ZZ} ( \gamma )$ and the assertion is 
obvious. So it suffices to consider the case 
where $\tau$ is a pyramid over 
$\gamma = \tau \cap \{ v_i=0 \}$ but its 
unique vertex $P \prec \tau$ such that 
$P \notin \gamma$ has height $h \geq 2$ from the 
hyperplane $\{ v_i=0 \} \subset \RR^n$. 
In this case, we define two hyperplanes 
$H_{\tau}$ and $L_{\tau}$ in $\RR^n$ by 
\begin{equation*}
H_{\tau}= \{ v \in \RR^n \ | \ 
\langle a( \tau ) , v \rangle 
= N ( \tau ) \}, 
\end{equation*}
\begin{equation*} 
L_{\tau}= \{ v \in \RR^n \ | \ 
\langle a( \tau ) , v \rangle 
= \langle a( \tau ) , \mbbo 
\rangle = \nu ( \tau ) \}. 
\end{equation*} 
Note that $P \in \tau \subset H_{\tau}$ 
and $L_{\tau}$ is the hyperplane 
passing through the point 
$(1,1, \ldots, 1) \in \RR^n_+$ and 
parallel to $H_{\tau}$. 
Namely $H_{\tau}$ is the affine span 
${\rm aff} ( \tau ) \simeq \RR^{n-1}$ 
of $\tau$. Moreover the 
affine subspace $L_{\tau} \cap \{ v_i=0 \} 
\subset \RR^n$ is parallel to the 
affine span $H_{\tau} \cap \{ v_i=0 \} 
\subset \RR^n$ of $\gamma = \tau \cap \{ v_i=0 \}$. 
By Lemma \ref{RESON}, this implies that $\lambda = \exp 
(-2 \pi i \nu ( \tau )/N ( \tau ) ) \in \CC$ 
is a root of $\zeta_{\gamma}(t)$ if and only if 
$L_{\tau} \cap \{ v_i=0 \}$ is rational i.e. 
$L_{\tau} \cap \{ v_i=0 \} \cap \ZZ^n 
\not= \emptyset$. On the other hand, it is 
easy to see that the affine subspace 
$H_{\tau} \cap \{ v_i=h-1 \} 
\subset \RR^n$ is a parallel translation of 
$L_{\tau} \cap \{ v_i=0 \}$ by a lattice 
vector. Hence if $L_{\tau} \cap \{ v_i=0 \}$ 
is rational, then 
$H_{\tau} \cap \{ v_i=h-1 \} \cap \ZZ^n 
\not= \emptyset$ and the lattice height of 
the pyramid $\tau$ from its base 
$\gamma = \tau \cap \{ v_i=0 \}$ 
is $h \geq 2$ i.e. $\Vol_{\ZZ} ( \tau ) \geq 2 
\Vol_{\ZZ} ( \gamma )$. It follows that 
the polynomial $F_{\tau, \gamma}$ 
is divisible by the factor $t- \lambda$. 
This completes the proof. 
\end{proof} 

Motivated by this proposition, we introduce the following 
definitions. 

\begin{definition}\label{HCC} 
Let $\tau$ be a $k$-dimensional  
lattice V-simplex contained in a compact facet of 
$\Gamma_+(f)$. 
\begin{enumerate}
\item 
We say that $\tau$ has a (possibly empty) 
\emph{corner $\tau_0 \prec \tau$ of codimension $r$} if 
$\dim \tau - \dim \tau_0 =k- \dim \tau_0 =r$ and 
any face $\sigma$ of $\tau$ 
containing it is a $V$-face. 
\item 
If $\tau$ has a (possibly empty) 
corner $\tau_0 \prec \tau$ of codimension $r$, then we set 
\begin{equation*} 
F_{\tau,\tau_0} (t)= \prod_{\sigma: \tau_0 \prec \sigma \prec \tau, 
\sigma \not= \emptyset} 
\Bigl\{  \zeta_{ \sigma } (t) 
\Bigr\}^{(-1)^{k - \dim \sigma}} \in \CC (t). 
\end{equation*} 
\end{enumerate}
\end{definition} 

\begin{remark} Every $(n-1)$-dimensional  
lattice simplex $\tau$ contained in a compact facet of 
$\Gamma_+(f)$ has a unique corner of maximal 
codimension, which we will denote by $\mathcal{C}_{\tau}$.
We will also write shortly $F_{\tau} (t)$ for 
$F_{\tau,\mathcal{C}_{\tau}} (t)$.
\end{remark} 

In the next subsection we will prove 
the following result. 

\noindent \begin{theorem}\label{Key-5} 
\begin{enumerate} 
\item
Let $\tau$ be a $k$-dimensional  
lattice V-simplex contained in a compact facet of 
$\Gamma_+(f)$. Assume that for some $r \geq 1$ 
it has a non-empty corner $\tau_0$ of codimension $r$, then 
$F_{\tau,\tau_0}(t) \in \CC (t)$ 
is a polynomial of $t$. 
\item
If $k=n-1$ and if moreover $\tau$ is 
not a $B_1$-pyramid, then the complex number 
\begin{equation*}
\lambda = \exp 
\left(-2 \pi i \frac{\nu ( \tau )}{N ( \tau )} 
\right) \in \CC
\end{equation*}
is a root of the polynomial $F_{\tau}(t)$. 
\end{enumerate}
\end{theorem}

We can generalize Theorem \ref{Key-5} slightly 
to allow also simplices with empty corners as follows. 
If an $(n-1)$-dimensional  
lattice simplex $\tau$ contained in a compact facet of 
$\Gamma_+(f)$ has an empty corner $\tau_0 = \emptyset$, then 
we have the expression $\tau =A_1A_2 \cdots A_n$ such that 
for any $1 \leq i \leq n$ the vertex 
$A_i$ is on the positive part of the $i$-th 
coordinate axis of $\RR^n$. 

\begin{proposition}\label{EMPCC} 
Let $\tau$ be an $(n-1)$-dimensional  
lattice simplex contained in a compact facet of 
$\Gamma_+(f)$. Assume that it has 
an empty corner $\tau_0 = \emptyset$, then 
the function of $t$ 
\begin{equation*}
F_{\tau}(t) \cdot  
\Bigl( 1-t \Bigr)^{(-1)^{n}} \in \CC (t) 
\end{equation*}
is a polynomial. If we assume moreover 
that $\tau$ is not a $B_1$-simplex, then 
the complex number 
\begin{equation*}
\lambda = \exp 
\left(-2 \pi i \frac{\nu ( \tau )}{N ( \tau )} 
\right) \in \CC
\end{equation*}
is a root of the polynomial. 
\end{proposition} 

\begin{proof} 
By the embedding $\RR^n \hookrightarrow 
\RR^n \times \RR$, $v \longmapsto (v,0)$ we 
regard $\tau$ as a lattice simplex in 
$\RR^n \times \RR$ and set $Q(0,1) \in 
\RR^n \times \RR$. Let $\tau^{\prime}$ 
be the convex hull of $\{ Q \} \cup \tau$ 
in $\RR^n \times \RR$. 
In this situation, it is easy to see that we have 
an equality 
\begin{equation*}
F_{\tau^{\prime},Q}(t)= F_{\tau}(t) \cdot  
\Bigl( 1-t \Bigr)^{(-1)^{n}} 
\end{equation*}
from which the first assertion immediately follows by Theorem \ref{Key-5}. 
Since we have $N( \tau^{\prime} )=N( \tau )$ and 
$\nu ( \tau^{\prime} )= \nu ( \tau ) + N( \tau )$, 
the second assertion also follows from 
Theorem \ref{Key-5}.
\end{proof}

Together with Theorem \ref{ADJJJ}, 
following the strategy of Lemahieu-Van Proeyen \cite{L-V} 
we can now confirm the monodromy conjecture for 
non-degenerate hypersurfaces in many cases also for 
$n \geq 4$. 
Let $\tau_1, \ldots, \tau_k \prec \Gamma_+(f)$ be 
the compact facets of $\Gamma_+(f)$. Then we say 
that the Newton polytope of $f$ has a good pavement 
by lattice simplices if for any $1 \leq i \leq k$ 
there exists a decomposition $\tau_i= \cup_{j=1}^{n_i} \tau_{ij}$ 
of $\tau_i$ into ($n-1$)-dimensional lattice simplices 
$\tau_{ij}$ for which the following conditions are satisfied. 

\medskip
\par \noindent
{\bf (i)} For any $1 \leq i \leq k$ and $1 \leq j \leq n_i$ 
the lattice simplex $\tau_{ij}$ has no V-face or 
it has a (non-empty) corner 
which is contained in any V-face of $\tau_{ij}$. 

\medskip
\par \noindent
{\bf (ii)} If $\tau_{ij} \not= \tau_{i^{\prime} j^{\prime}}$ 
then they have no common V-face. Moreover any V-face of 
$\Gamma_+(f)$ is decomposed into those of the lattice 
simplices $\tau_{ij}$. 

\medskip
\par \noindent
{\bf (iii)} For any $1 \leq i \leq k$ the facet $\tau_i$ 
is a $B$-facet or there exists a lattice simplex 
$\tau_{ij}$ in it which is not a $B_1$-pyramid. 
\medskip 

\begin{theorem}\label{APPT} 
Assume that $f$ is non-degenerate and 
the Newton polytope of $f$ has a good pavement 
by lattice simplices. 
Let $s_0 \in \CC$ be a pole 
of $Z_{ {\rm top}, f}(s)$ which is 
contributed only by compact facets of $\Gamma_+(f)$, and 
those of them that are $B$-facets are consistent 
in the sense of Definition \ref{defconsist}. 
Then the complex number 
$\exp ( 2 \pi i  s_0 ) \in \CC$ 
is an eigenvalue of the 
monodromy of $f$ at 
the origin $0 \in \CC^n$.
\end{theorem} 

\begin{proof}
If $s_0=-1$ the assertion is trivial. Otherwise, 
by Theorem \ref{ADJJJ} the pole $s_0 \not= -1$ 
is contributed also by a non-$B$-facet $\tau_i$ 
of $\Gamma_+(f)$. Moreover by Theorem \ref{TOV} 
and the conditions {\bf (i)}, {\bf (ii)} we have 
\begin{equation*}
\zeta_{f,0} (t) = 
\Bigl\{ \prod_{i=1}^k \prod_{j=1}^{n_i} 
F_{\tau_{ij}} (t) \Bigr\}^{(-1)^{n-1}}. 
\end{equation*}
Then the complex number 
$\exp ( 2 \pi i  s_0 ) \in \CC$ 
is an eigenvalue of the 
monodromy of $f$ at the origin $0 \in \CC^n$ 
by Theorem \ref{Key-5} applied to 
the condition {\bf (iii)}. 
\end{proof}

\begin{example}
We consider the hypersurface 
$$H: f(x_1,x_2,x_3,x_4,x_5)=x_1^8+x_2^6+x_3^{10}+x_4^{12}+
x_5^9+x_1^2x_3^3+x_1x_2x_4+x_1x_2x_5^3+x_2^2x_3=0$$ which is 
non-degenerate at the origin in $\mathbb{C}^5$.
We denote the vertices of $\Gamma_+(f)$ by 
$$A=(8,0,0,0,0), B=(0,6,0,0,0), C=(0,0,10,0,0), 
D=(0,0,0,12,0), E=(0,0,0,0,9),$$ 
$$F=(2,0,3,0,0), G=(1,1,0,1,0), H=(1,1,0,0,3), 
I=(0,2,1,0,0).$$
The Newton polyhedron $\Gamma_+(f)$ has $10$ 
compact facets, which are $$\tau_1=EFGHI, 
\tau_2=CEFGI, \tau_3=CDEFG, \tau_4=BEGHI, \tau_5=CDEGI, $$ 
$$\tau_6=BDEGI, 
\tau_7=ABGHI, \tau_8=AEFGH, \tau_9=ADEFG, \tau_{10}=AFGHI.$$
Then we find that 
$$\zeta_{f,0}(t)=\prod_{i=1}^{10}F_{\tau_i}(t) 
F_{CDE}(t)F_{AEF}(t)F_{BEI}(t)F_{CEI,CI}(t)F_{CEF,CF}(t).$$
It follows immediately by Theorem \ref{Key-5} 
that the monodromy conjecture holds for $H$ at the origin.
Notice that we did not take $F_{CEI}(t)=
(\zeta_{CEI}(t)\zeta_{C}(t))/(\zeta_{CE}(t)\zeta_{CI}(t))$ 
nor $F_{CEF}(t)=(\zeta_{CEF}(t)
\zeta_{C}(t))/(\zeta_{CE}(t)\zeta_{CF}(t))$
as we have already the contributions of the 
V-faces $CE$ and $C$ in $F_{CDE}(t)=
(\zeta_{CDE}(t)\zeta_{C}(t)\zeta_{D}(t)
\zeta_{E}(t))/(\zeta_{CD}(t)\zeta_{CE}(t)\zeta_{DE}(t))$.  \qed 
\end{example}

\subsection{Hypermodular functions} 

For the proof of Theorem \ref{Key-5} we shall introduce 
some new notions and their basic properties. 
Let $S$ be a finite set 
and denote its power set by $2^S$. Namely 
elements of $2^S$ are subsets $I \subset S$ of $S$. 
Then for a function $\phi : 2^S \longrightarrow \ZZ$ 
we define new ones $\phi^{\downarrow}, 
\phi^{\uparrow} : 2^S \longrightarrow \ZZ$ by 
\begin{equation*} 
\phi^{\downarrow} (I) = 
\sum_{J \subset I} \phi (J), \qquad 
\phi^{\uparrow} (I) = 
\sum_{J \subset I} (-1)^{|I|-|J|} \phi (J). 
\end{equation*}
We call $\phi^{\downarrow}$ (resp. $\phi^{\uparrow}$) 
the antiderivative (resp. derivative) of $\phi$. 
Then we can easily check that 
$\phi^{\uparrow \downarrow} = 
\phi^{\downarrow \uparrow} = \phi$. 

\begin{definition}\label{FSM} 
\begin{enumerate} 
\item We say that the function $\phi$ is 
\emph{hypermodular} if $\phi^{\uparrow} 
(I) \geq 0$ for 
any subset $I \subset S$. 
\medskip 
\item The 
function $\phi$ is called \emph{strictly 
hypermodular} if it is 
hypermodular and $\phi^{\uparrow} 
(S) > 0$. 
\end{enumerate} 
\end{definition} 

\begin{lemma}\label{SSM} 
The product of two 
hypermodular functions 
$\phi, \psi : 2^S \longrightarrow \ZZ$ 
is hypermodular. Moreover it is 
strictly hypermodular if and only 
if there exist subsets $I, J \subset S$ 
of $S$ such that $I \cup J=S$ and 
both $\phi^{\uparrow}(I)$ and $  
\psi^{\uparrow}(J)$ are strictly positive. 
\end{lemma} 

\begin{proof} 
For any subset $R \subset S$ of $S$ we have 
\begin{align*}
( \phi \psi )^{\uparrow}(R)   = & 
( \phi^{\uparrow \downarrow} 
\psi^{\uparrow \downarrow} 
 )^{\uparrow}(R) = 
\sum_{I \cup J \subset U \subset R} 
(-1)^{|R|-|U|} 
\phi^{\uparrow}(I) \psi^{\uparrow}(J) 
\\
& = \sum_{I \cup J =R} \phi^{\uparrow}(I) 
\psi^{\uparrow}(J). 
\end{align*}
Then the assertion immediately follows. 
\end{proof}

\subsection{Reduction to the case $k=n-1$ and $r=n-1$} 

We can obviously suppose that $k=n-1$ as the 
computations are made in the coordinate hyperplane containing $\tau$. 
We now explain how to reduce the proof of Theorem \ref{Key-5} 
to the case $r=n-1$. For simplicity assume 
that the corner $\gamma \prec \tau$ of the simplex 
$\tau \subset \partial \Gamma_+(f)$ is defined 
by $\gamma = \tau \cap 
\{ v_1=v_2= \cdots =v_r=0 \}$. We set 
\begin{equation*}
\lambda = \exp 
\left(-2 \pi i \frac{\nu ( \tau )}{N ( \tau )} 
\right) \in \CC. 
\end{equation*}
As in the proof of Proposition \ref{Key-1} 
we define two parallel affine hyperplanes 
$H_{\tau}$ and $L_{\tau}$ in $\RR^n$ by 
\begin{equation*}
H_{\tau}= \{ v \in \RR^n \ | \ 
\langle a( \tau ) , v \rangle 
= N ( \tau ) \}, 
\end{equation*}
\begin{equation*} 
L_{\tau}= \{ v \in \RR^n \ | \ 
\langle a( \tau ) , v \rangle 
= \langle a( \tau ) , \mbbo 
  \rangle = \nu ( \tau ) \}. 
\end{equation*} 
Let $W= \{ v_1=v_2= \cdots =v_r=0 \} 
\simeq \RR^{n-r} \subset \RR^n$ be 
the linear subspace of $\RR^n$ 
spanned by $\gamma$. Similarly, for a 
face $\sigma$ of $\tau$ containing $\gamma$ 
let $W_{\sigma} \simeq 
\RR^{\dim \sigma +1} \subset \RR^n$ be 
the linear subspace of $\RR^n$ 
spanned by $\sigma$. Then by Lemma \ref{RESON} 
$\zeta_{\sigma} ( \lambda )=0$ if and only if 
the affine hyperplane 
$L_{\tau} \cap W_{\sigma} \subset W_{\sigma}$ 
of $W_{\sigma}$ is rational i.e. 
$L_{\tau} \cap W_{\sigma} \cap \ZZ^n 
\not= \emptyset$. Let $\Phi_0: W \simto W$ be 
a unimodular transformation of $W$ such 
that $\Phi_0( \gamma ) \subset W \cap 
\{ v_n=c \}$ for some $c \in \ZZ_{>0}$. Then we can 
easily extend it to a unimodular transformation 
$\Phi : \RR^n \simto \RR^n$ of $\RR^n$ 
which preserves $W_{\sigma}$ for any 
$\sigma \prec \tau$ containing $\gamma$ 
and the point $\mbbo = (1,1, \ldots, 1) \in \RR^n$ 
We can choose such $\Phi$ so that 
the heights of $\tau$ and 
$\Phi ( \tau )$ from each coordinate hyperplane 
in $\RR^n$ containing $W$ are the same. 
Indeed, for the invertible matrix 
$A_0 \in {\rm GL}_{n-r}( \ZZ )$ representing 
$\Phi_0: W \simto W$ it suffices to define 
$\Phi : \RR^n \simto \RR^n$ by taking an 
invertible matrix 
$A \in {\rm GL}_{n}( \ZZ )$ of the form 
\begin{equation*} 
A =  \left( \begin{array}{c|c}
I_r & 0 \\ 
\hline  
* & A_0 \\ 
\end{array} \right) 
\in {\rm GL}_{n}( \ZZ )
\end{equation*} 
such that $A \mbbo = \mbbo$, 
where $I_r \in {\rm GL}_{r}( \ZZ )$ 
stands for the identity matrix of size $r$. 
By this construction of $\Phi$, 
$\tau$ is a $B_1$-pyramid if and only if 
$\Phi ( \tau )$ is so. 
Set $\tau^{\prime} = \Phi ( \tau )$ and 
define two parallel affine hyperplanes 
$H_{\tau^{\prime}}$ and $L_{\tau^{\prime}}$ 
in $\RR^n$ similarly to the case of $\tau$ 
so that we have 
$\Phi ( H_{\tau} )= H_{\tau^{\prime}}$. 
Since $\Phi ( L_{\tau} )$ is parallel to 
$\Phi ( H_{\tau} )= H_{\tau^{\prime}}$ 
and passes through the point 
$\Phi ( \mbbo ) = \mbbo \in \RR^n$, 
we have also 
$\Phi ( L_{\tau} )= L_{\tau^{\prime}}$. 
Since the unimodular transformation 
$\Phi$ preserves lattice distances, 
we thus obtain $N( \tau )=N ( \tau^{\prime} )$,  
$\nu ( \tau )= \nu ( \tau^{\prime} )$ and 
\begin{equation*}
\lambda = \exp 
\left(-2 \pi i 
\frac{\nu ( \tau^{\prime} )}{N ( \tau^{\prime} )} 
\right). 
\end{equation*}
Moreover for any $\sigma \prec \tau$ containing $\gamma$ 
we have $N( \sigma ) = N( \Phi( \sigma ))$ 
and hence 
$\zeta_{ \sigma } (t) \equiv 
\zeta_{\Phi ( \sigma )} (t)$. 
Then we obtain an equality 
$F_{\tau}(t)=F_{\tau^{\prime}}(t)$, 
where we slightly generalized 
Definition \ref{HCC} in an obvious way 
to define $F_{\tau^{\prime}}(t)$. 
Hence, to prove 
Theorem \ref{Key-5} we may assume that 
the corner $\gamma$ of $\tau$ is contained in 
$W \cap \{ v_n=c \}$ for some $c \in \ZZ_{>0}$. 
Let $\pi : \RR^n \longrightarrow \RR^{r+1}$, 
$v \longmapsto (v_1, \ldots, v_r, v_n)$ be the 
projection. Then by the definition of normalized 
volumes, for any face $\sigma$ of $\tau$ 
containing the corner $\gamma \subset 
W \cap \{ v_n=c \}$ we have 
$\Vol_{\ZZ}( \sigma )= \Vol_{\ZZ}( \pi ( \sigma ))
\cdot \Vol_{\ZZ}( \gamma )$ and hence 
$\zeta_{\sigma}(t)= \{ \zeta_{\pi ( \sigma )}
(t) \}^{\Vol_{\ZZ}( \gamma )}$. We thus 
obtain an equality 
$F_{\tau}(t)= \{ F_{\pi ( \tau )}
(t) \}^{\Vol_{\ZZ}( \gamma )}$. Moreover we have 
$N( \tau )=N( \pi ( \tau ))$ and 
$\nu ( \tau )= \nu ( \pi ( \tau ))$.  
This implies that we have only to consider 
the case $r=n-1$.

\subsection{The proof of the case $r=n-1$} 

We have reduced our proof to the case where 
$r=n-1$, a vertex $Q$ of our simplex 
$\tau =QA_1A_2 \cdots A_{n-1}$ has the form 
$Q=(0,0, \ldots, 0, c)$ for some $c \in \ZZ_{>0}$ 
and its edges are given by 
\begin{equation*} 
\overrightarrow{QA_1} = 
\left(  \begin{array}{c}
      a_1 \\
      0 \\
      \vdots \\
      0 \\
      b_1 
    \end{array}  \right),  \qquad 
\overrightarrow{QA_2} =  
\left(  \begin{array}{c}
      0 \\
      a_2 \\
      \vdots \\ 
      0 \\
      b_2 
    \end{array}  \right), 
\  \ldots \ldots  ,  \qquad  
\overrightarrow{QA_{n-1}} =  
\left(  \begin{array}{c}
      0 \\
      0 \\
      \vdots \\ 
      a_{n-1} \\
      b_{n-1} 
    \end{array}  \right), 
\end{equation*} 
where $a_1, a_2, \ldots, a_{n-1} \in \ZZ_{>0}$ 
and $b_1, b_2, \ldots, b_{n-1}  \in \ZZ$. 
We set 
\begin{equation*} 
D= \prod_{i=1}^{n-1} a_i, \qquad 
K_i= \frac{b_i}{a_i} \cdot D \ (1 \leq i \leq n-1) 
\end{equation*} 
and $K= \sum_{i=1}^{n-1} K_i$. 
Note that $b_i$ and $K_i$ are negative. 
Moreover for 
a subset $I \subset S= \{ 1,2, \ldots, n-1 \}$ 
we denote by $\tau_I \prec \tau$ the face of 
$\tau$ whose vertices are $Q$ and $A_i \ 
(i \in I)$ and set 
\begin{equation*} 
D_I= \prod_{i \in I} a_i, \qquad 
{\rm gcd}_I = {\rm GCD} \Bigl( D, K_i \ (i \in I) 
\Bigr) >0. 
\end{equation*} 

\begin{lemma}\label{LODDP} 
The $|I|$-dimensional normalized volume 
$\Vol_{\ZZ}( \tau_I )$ of $\tau_I$ is given by 
the formula 
\begin{equation*} 
\Vol_{\ZZ}( \tau_I ) =
{\rm gcd}_I \cdot \frac{D_I}{D} 
\end{equation*} 
and we have 
\begin{equation*} 
N( \tau_I) =\frac{D}{{\rm gcd}_I} \cdot c 
= \frac{D_I}{\Vol_{\ZZ}( \tau_I )} \cdot c. 
\end{equation*} 
In particular, we have
\begin{equation*} 
N( \tau )= 
\frac{D}{{\rm gcd}_S} \cdot c. 
\end{equation*} 
\end{lemma} 

\begin{proof}
We only treat the case $I=S$ and $\tau_I= \tau$. The 
general case can be treated similarly. First, 
note that the primitive conormal vector $a( \tau )
\in \ZZ^n$ of the 
$(n-1)$-dimensional simplex $\tau$ is equal to 
\begin{equation*}
\frac{1}{{\rm gcd}_S} 
\left(  \begin{array}{c}
      -K_1 \\
      -K_2 \\
      \vdots \\
      -K_{n-1} \\
      D 
    \end{array}  \right). 
\end{equation*}
From this, the assertion for $N( \tau )$ immediately follows. 
Let $\widetilde{\tau} \subset \RR^n$ be the $n$-dimensional 
simplex obtained by taking the convex hull of $\tau$ 
and the point $R=(0,0, \ldots, 0, c+1)$. 
Then by the above formula for $a( \tau )$, 
the lattice height of $\widetilde{\tau}$ 
from its base $\tau$ (i.e. the lattice 
distance of the point $R$ from the affine 
span ${\rm aff}(\tau)$) 
is equal to $\frac{D}{{\rm gcd}_S}$. 
Since the $n$-dimensional normalized 
volume of $\widetilde{\tau}$ is equal to $D$, 
we get also the remaining assertion 
$\Vol_{\ZZ}( \tau ) ={\rm gcd}_S$. 
\end{proof} 

For a subset $I \subset S= \{ 1,2, \ldots, n-1 \}$ 
we set 
\begin{equation*} 
\zeta_{I} (t) = \Bigl\{ 1-t^{N( \tau_I )} 
\Bigr\}^{\Vol_{\ZZ}( \tau_I )} 
\in \CC [t]
\end{equation*} 
so that we have the equality 
\begin{equation*} 
F_{\tau}(t) = \prod_{I \subset S} 
\Bigl\{ \zeta_{I} (t) 
\Bigr\}^{(-1)^{n-1-|I|}}. 
\end{equation*}

\begin{lemma}\label{LOD} 
The complex number 
\begin{equation*}
\lambda = \exp 
\left(-2 \pi i \frac{\nu ( \tau )}{N ( \tau )} 
\right) \in \CC
\end{equation*}
is a root of the polynomial $\zeta_I(t)$ 
if and only if ${\rm gcd}_I | K$. 
\end{lemma} 

\begin{proof} 
By the formula for $a( \tau )$ in the proof of 
Lemma \ref{LODDP}, we obtain 
\begin{equation*}
\nu ( \tau )= 
\frac{D-K}{{\rm gcd}_S} 
\end{equation*}
and 
\begin{equation}\label{QQE} 
\frac{\nu ( \tau ) N ( \tau_I
 )}{N( \tau )} = 
\frac{D-K}{{\rm gcd}_I}. 
\end{equation}
Note that $\lambda$ is a root of 
$\zeta_I(t)$ if and only if 
$\nu ( \tau ) N ( \tau_I
 )/N( \tau )$ is an integer. Then the 
assertion follows immediately from 
(\ref{QQE}) and the fact ${\rm gcd}_I | D$. 
This completes the proof. 
\end{proof}

By this lemma the multiplicity of $t- \lambda$ 
in the rational function $F_{\tau}(t)$ is 
equal to 
\begin{equation*}
\sum_{I:  \ {\rm gcd}_I | K} 
(-1)^{n-1-|I|}  \ {\rm gcd}_I \cdot \frac{D_I}{D}. 
\end{equation*}
Similarly we obtain the following result. 

\begin{lemma}\label{LODD} 
For any $m \in \ZZ$ the complex number 
$\exp \left(2 \pi i m/N ( \tau ) \right) \in \CC$ 
is a root of the polynomial $\zeta_I(t)$ 
if and only if ${\rm gcd}_I |  (m \cdot {\rm gcd}_S)$. 
\end{lemma} 

\begin{proposition}\label{PGC} 
The function $F_{\tau}(t)$ is a polynomial in $t$. 
\end{proposition} 

\begin{proof} 
By Lemma \ref{LODD} it suffices to show that 
for any $m \in \ZZ$ the alternating sum 
\begin{equation*}
G_m= 
\sum_{I:  \ {\rm gcd}_I | (m \cdot {\rm gcd}_S)} 
(-1)^{n-1-|I|}  \ {\rm gcd}_I \cdot \frac{D_I}{D} 
\end{equation*}
is non-negative. Fix $m \in \ZZ$ and for 
a prime number $p$ denote its multiplicities in 
the prime decompositions of $a_i, b_i$ and 
$m$ by $\alpha (p)_i, \beta (p)_i$ and 
$\delta (p)$ respectively. We set 
\begin{equation*}
\gamma (p)= \delta (p) + 
{\rm min}_{1 \leq i \leq n-1} 
\{ \beta (p)_i- \alpha (p)_i, 0 \} 
\end{equation*}
and define a function $\phi_p : 2^S 
\longrightarrow \ZZ$ by 
\begin{equation*}
\phi_p (I) 
= \begin{cases}
p^{ {\rm min}_{i \in I} 
\{ \beta (p)_i- \alpha (p)_i, 0 \} 
+ \sum_{i \in I} \alpha (p)_i }
 & \Bigl( {\rm min}_{i \in I} 
\{ \beta (p)_i- \alpha (p)_i, 0 \} 
\leq \gamma (p) \Bigr), \\
 & \\ 
\ 0 & ( \text{otherwise} ). 
\end{cases}
\end{equation*}
Then it is easy to see that for the 
function $\phi = \prod_{p: \ \text{prime}} 
\phi_p : 2^S \longrightarrow \ZZ$ we have 
\begin{equation*}
\phi^{\uparrow} (S)=G_m. 
\end{equation*}
Indeed, this follows immediately from the fact that 
for $I \subset S$ the multiplicity of $p$ in 
${\rm gcd}_I$ is equal to 
\begin{equation*}
{\rm min}_{i \in I} 
\{ \beta (p)_i- \alpha (p)_i, 0 \} 
+ \sum_{1 \leq i \leq n-1} \alpha (p)_i. 
\end{equation*}
By Lemma \ref{SSM} we have only to prove 
that for any prime number $p$ the function 
$\phi_p : 2^S \longrightarrow \ZZ$ is 
hypermodular. For this purpose, 
we reorder the pairs $(a_i, b_i)$ 
$(1 \leq i \leq n-1)$ so that we have 
\begin{equation*}
\beta (p)_1- \alpha (p)_1 \leq 
\beta (p)_2- \alpha (p)_2 \leq \cdots \cdots 
\leq \beta (p)_{n-1}- \alpha (p)_{n-1}. 
\end{equation*}
Fix a subset $I= \{ i_1, i_2, i_3, \ldots \} 
\subset S= \{ 1,2, \ldots, n-1 \}$ 
$(i_1 < i_2 < i_3 < \cdots )$ of $S$. 
We will show the non-negativity of 
the alternating sum 
\begin{equation}\label{EQRC} 
\phi_p^{\uparrow} (I) = 
\sum_{J \subset I} (-1)^{|I|-|J|} 
\phi_p(J). 
\end{equation}
We define 
$q \geq 0$ to be the maximal number 
such that 
$\beta (p)_{i_q}- \alpha (p)_{i_q} <0$ 
(resp. $\beta (p)_{i_q}- \alpha (p)_{i_q} 
\leq \gamma (p)$) in the case 
$\gamma (p) \geq 0$ (resp. $\gamma (p) < 0$). 
First let us consider the case 
$\gamma (p) \geq 0$. Then 
for $1 \leq l \leq q$ the part of the 
alternating sum 
\eqref{EQRC} over the 
subsets $J \subset I$ such that 
${\rm min} J = i_l$ is equal to 
\begin{equation*} 
(-1)^{l-1} p^{\beta (p)_{i_l}} \prod_{j>l} 
(p^{\alpha (p)_{i_j}} -1). 
\end{equation*}
Indeed, for instance the term in this 
alternating sum which corresponds to  
$J= \{ i_l, i_{l+1}, \ldots \ldots \} \subset 
I= \{ i_1, i_2, \ldots, i_l, \ldots \ldots \}$ 
is equal to 
\begin{equation*}
(-1)^{l-1} p^{\beta (p)_{i_l}- \alpha (p)_{i_l} 
+ \sum_{j \geq l} \alpha (p)_{i_j} }
=(-1)^{l-1} p^{\beta (p)_{i_l}} 
\prod_{j>l} p^{\alpha (p)_{i_j}}. 
\end{equation*}
Moreover the remaining part of \eqref{EQRC} 
is equal to 
\begin{equation*} 
(-1)^{q} \prod_{j>q} 
(p^{\alpha (p)_{i_j}} -1). 
\end{equation*}
We thus obtain the equality 
\begin{align*}
\phi_p^{\uparrow} (I)  & 
= p^{\beta (p)_{i_1}} \prod_{j>1} 
(p^{\alpha (p)_{i_j}} -1) - 
p^{\beta (p)_{i_2}} \prod_{j>2} 
(p^{\alpha (p)_{i_j}} -1) + \cdots \cdots 
\\
& \cdots + 
(-1)^{q-1} p^{\beta (p)_{i_q}} \prod_{j>q} 
(p^{\alpha (p)_{i_j}} -1) 
+ (-1)^{q} \prod_{j>q} 
(p^{\alpha (p)_{i_j}} -1). 
\end{align*}
Note that for any $1 \leq j \leq q$ we 
have $\beta (p)_{i_j}- \alpha (p)_{i_j}
<0$ and obtain an inequality 
\begin{equation}\label{EQQU} 
p^{\beta (p)_{i_{j-1}}} (p^{\alpha (p)_{i_j}} -1) 
\geq p^{\alpha (p)_{i_j}} -1 \geq 
p^{\beta (p)_{i_{j}}}. 
\end{equation}
Thus, subdividing the terms in 
the above expression of 
$\phi_p^{\uparrow} (I)$ into pairs, 
we get the desired non-negativity 
$\phi_p^{\uparrow} (I) \geq 0$. 
Finally let us consider the case 
$\gamma (p) < 0$. In this case, we have 
the following expression of 
$\phi_p^{\uparrow} (I)$: 
\begin{equation*}
\phi_p^{\uparrow} (I)  
= p^{\beta (p)_{i_1}} \prod_{j>1} 
(p^{\alpha (p)_{i_j}} -1) - 
p^{\beta (p)_{i_2}} \prod_{j>2} 
(p^{\alpha (p)_{i_j}} -1) + \cdots + 
(-1)^{q-1} p^{\beta (p)_{i_q}} \prod_{j>q} 
(p^{\alpha (p)_{i_j}} -1). 
\end{equation*}
Then by using the inequality \eqref{EQQU} 
we can prove the non-negativity 
$\phi_p^{\uparrow} (I) \geq 0$ 
as in the previous 
case $\gamma (p) \geq 0$. 
This completes the proof. 
\end{proof}

\begin{proposition}\label{PPGC} 
Assume that $\tau$ is not a $B_1$-simplex. 
Then the complex number 
\begin{equation*}
\lambda = \exp 
\left(-2 \pi i \frac{\nu ( \tau )}{N ( \tau )} 
\right) \in \CC
\end{equation*}
is a root of 
the polynomial $F_{\tau}(t)$. 
\end{proposition} 

\begin{proof} 
By Lemma \ref{LOD} it suffices to show that 
the alternating sum 
\begin{equation*}
G= 
\sum_{I:  \ {\rm gcd}_I | K} 
(-1)^{n-1-|I|}  \ {\rm gcd}_I \cdot \frac{D_I}{D} 
\end{equation*}
is positive. For 
a prime number $p$ denote its multiplicities in 
the prime decompositions of $a_i, b_i$ and 
$K$ by $\alpha (p)_i, \beta (p)_i$ and 
$\kappa (p)$ respectively. We set 
\begin{equation*}
\mu (p)= \kappa (p) - 
\sum_{i=1}^{n-1}   \alpha (p)_i 
\end{equation*}
and define a function $\psi_p : 2^S 
\longrightarrow \ZZ$ by 
\begin{equation*}
\psi_p (I) 
= \begin{cases}
p^{ {\rm min}_{i \in I} 
\{ \beta (p)_i- \alpha (p)_i, 0 \} 
+ \sum_{i \in I} \alpha (p)_i }
 & \Bigl( {\rm min}_{i \in I} 
\{ \beta (p)_i- \alpha (p)_i, 0 \} 
\leq \mu (p) \Bigr), \\
 & \\ 
\ 0 & ( \text{otherwise} ).  
\end{cases}
\end{equation*}
Then it is easy to see that for the 
function $\psi = \prod_{p: \ \text{prime}} 
\psi_p : 2^S \longrightarrow \ZZ$ we have 
\begin{equation*}
\psi^{\uparrow} (S)=G. 
\end{equation*}
Now let us set $S_p= \{ 1 \leq i \leq n-1 \ | \ 
\alpha (p)_i =0 \}$ and $I_p =S \setminus S_p 
= \{ 1 \leq i \leq n-1 \ | \ 
\alpha (p)_i >0 \}$. By our assumption 
we have $a_i >1$ for any $1 \leq i \leq n-1$ 
and hence 
$\cup_{p: \ \text{prime}} I_p 
=S$.  By Lemma \ref{SSM}, in order to 
show the positivity $\psi^{\uparrow} (S) 
>0$ it suffices to prove that for any 
prime number $p$ we have 
$\psi_p^{\uparrow} (I_p) >0$. 
As in the proof of Proposition \ref{PGC} 
we reorder the pairs $(a_i, b_i)$ 
$(1 \leq i \leq n-1)$ so that we have 
\begin{equation*}
\beta (p)_1- \alpha (p)_1 \leq 
\beta (p)_2- \alpha (p)_2 \leq \cdots \cdots 
\leq \beta (p)_{n-1}- \alpha (p)_{n-1} 
\end{equation*}
and $\alpha (p)_i \geq \alpha (p)_{i+1}$ 
whenever $\beta (p)_i- \alpha (p)_i 
= \beta (p)_{i+1}- \alpha (p)_{i+1}$.  
Moreover we set $I_p = \{ i_1, i_2, i_3, \ldots \}$ 
$(i_1 < i_2 < i_3 < \cdots )$. 
We define 
$q \geq 0$ to be the maximal number 
such that 
$\beta (p)_{i_q}- \alpha (p)_{i_q} <0$ 
(resp. $\beta (p)_{i_q}- \alpha (p)_{i_q} 
\leq \mu (p)$) in the case 
$\mu (p) \geq 0$ (resp. $\mu (p) < 0$). 
Then we have the same expressions of 
$\psi_p^{\uparrow} (I_p) >0$ as 
in the proof of Proposition \ref{PGC}. 
In the case $\mu (p) \geq 0$ we have 
\begin{align}\label{SSS1} 
\psi_p^{\uparrow} (I_p)  & 
= p^{\beta (p)_{i_1}} \prod_{j>1} 
(p^{\alpha (p)_{i_j}} -1) - 
p^{\beta (p)_{i_2}} \prod_{j>2} 
(p^{\alpha (p)_{i_j}} -1) + \cdots \cdots 
\\
& \cdots + 
(-1)^{q-1} p^{\beta (p)_{i_q}} \prod_{j>q} 
(p^{\alpha (p)_{i_j}} -1) 
+ (-1)^{q} \prod_{j>q} 
(p^{\alpha (p)_{i_j}} -1). \nonumber
\end{align}
In the case $\mu (p) < 0$ we have 
\begin{equation}\label{SSS2} 
\psi_p^{\uparrow} (I_p)  
= p^{\beta (p)_{i_1}} \prod_{j>1} 
(p^{\alpha (p)_{i_j}} -1) - 
p^{\beta (p)_{i_2}} \prod_{j>2} 
(p^{\alpha (p)_{i_j}} -1) + \cdots + 
(-1)^{q-1} p^{\beta (p)_{i_q}} \prod_{j>q} 
(p^{\alpha (p)_{i_j}} -1). 
\end{equation}
By the definitions of 
$I_p = \{ i_1, i_2, i_3, \ldots \} \subset S$ 
and $q \geq 0$ we have $i \in I_p$ for any 
$i \leq i_q$. Eventually we find that 
$i_j=j$ for any $j \leq q$. 
First let us consider the case $I_p = \emptyset$. 
Then we have 
\begin{equation*}
\psi_p (I_p) 
= \begin{cases}
p^0=1 & ( \mu (p) \geq 0 ), \\
 & \\ 
\ 0 & ( \mu (p) < 0 ). 
\end{cases}
\end{equation*}
In the case $\mu (p) \geq 0$ we thus obtain 
the positivity $\psi_p^{\uparrow} (I_p) 
>0$. But in the case $\mu (p) < 0$ the 
condition $I_p = \emptyset$ implies $q=0$ and 
such a case cannot occur by the following lemma. 

\begin{lemma}\label{NOCC} 
The case $I_p = \emptyset$ and $\mu (p) < 0$ cannot occur. 
\end{lemma} 
\begin{proof} 
Assume that $I_p = \emptyset$ and $\mu (p) < 0$. By the 
definition of $\mu (p)$ we have 
\begin{equation}\label{ETT} 
{\rm mult}_p (K) = 
{\rm mult}_p (D \cdot p^{\mu (p)} ). 
\end{equation}
Moreover for any $i \in S = S \setminus I_p$ 
we have 
\begin{equation*}
{\rm mult}_p (K_i) \geq {\rm mult}_p (D) >  
{\rm mult}_p (D \cdot p^{\mu (p)} ), 
\end{equation*}
where we used the condition $\mu (p) <0$ in 
the second inequality. We thus obtain 
the inequality 
\begin{equation*}
{\rm mult}_p (K) = {\rm mult}_p ( \sum_{i \in S} K_i)
> {\rm mult}_p (D \cdot p^{\mu (p)} )
\end{equation*}
which contradicts \eqref{ETT}. 
\end{proof}

By this lemma, it remains for us to treat the 
case $I_p \not= \emptyset$. From now on, 
we assume that $I_p \not= \emptyset$. Note 
that the inequality \eqref{EQQU} becomes 
an equality only in the case 
$p=2$, $\beta (p)_{i_{j-1}} = \beta (p)_{i_{j}}
=0$ and $\alpha (p)_{i_{j}}=1$. By Lemma 
\ref{NOCC} this means that the sums 
\eqref{SSS1} and \eqref{SSS2} may 
be zero only in the following two cases: 

\medskip 
\par \noindent 
{\bf Case 1}: $p=2$, $\mu (p) \geq 0$, 
$q=2m+1$ for $m \geq 0$ and 
$( \alpha (p)_1, \beta (p)_1) =(a,0)$ 
for $a>0$, $( \alpha (p)_2, \beta (p)_2) 
= \cdots \cdots 
= ( \alpha (p)_q, \beta (p)_q)=(1,0)$. 

\medskip 
\par \noindent 
{\bf Case 2}: $p=2$, $\mu (p) < 0$, 
$q=2m$ for $m \geq 1$ and 
$( \alpha (p)_1, \beta (p)_1) =(a,0)$ 
for $a>0$, $( \alpha (p)_2, \beta (p)_2) 
= \cdots \cdots 
= ( \alpha (p)_q, \beta (p)_q)=(1,0)$. 

\medskip 
\par \noindent
Indeed, in the case $\mu (p) \geq 0$ and 
$q=2m$ for $m \geq 0$, if $q < |I_p|$ the 
last term $(-1)^{q} \prod_{j>q} 
(p^{\alpha (p)_{i_j}} -1)$ of the 
alternating sum \eqref{SSS1} is positive. 
Even if $q = |I_p|$ we still have the positivity 
\begin{equation*}
(-1)^{q} \prod_{j>q} 
(p^{\alpha (p)_{i_j}} -1) = 
\psi_p( \emptyset )=1>0. 
\end{equation*}
Let us show that none of the above two 
cases can occur. 

\medskip 
\par \noindent 
{\bf Case 1}: Set $\alpha (p)= \sum_{i \in S} 
\alpha (p)_i$. Then $2^{\alpha (p)} |D$ and 
for any $i \in S_2$ we have 
$2^{\alpha (p)} |K_i$. We thus obtain 
the equality 
\begin{align*} 
K  & 
\equiv 2^{\alpha (p)-a} \cdot {\rm odd} + 
(q-1) \cdot 2^{\alpha (p)-1}  \cdot {\rm odd} + 
\sum_{j>q}  2^{\alpha (p)+ 
\beta (p)_{i_j}- \alpha (p)_{i_j} } 
\\
& \equiv 2^{\alpha (p)-a} \cdot {\rm odd} + 
(q-1) \cdot 2^{\alpha (p)-1}  \cdot {\rm odd} 
\equiv 2^{\alpha (p)-a} \cdot {\rm odd} 
\end{align*}
mod $2^{\alpha (p)}$, where we used also the 
fact that $\beta (p)_{i_j}- \alpha (p)_{i_j} 
\geq 0$ for any $j>q$. We conclude that 
$2^{\alpha (p)}$ does not divide $K$, 
which contradicts our assumption 
$\mu (p) \geq 0$. 

\medskip 
\par \noindent 
{\bf Case 2}: By the condition $q \geq 2$ 
we have $-1= \beta (p)_{i_2}- \alpha (p)_{i_2} 
\leq \mu (p)$. Then by $\mu (p) <0$ we obtain 
$\mu (p)=-1$. As in {\bf Case 1}, by using 
the fact that $q-1$ is odd and $\mu (p)=-1$, 
if $a=1$ we obtain the equality 
\begin{equation*}
K \equiv 
\sum_{j>q}  2^{\alpha (p)+ 
\beta (p)_{i_j}- \alpha (p)_{i_j} } 
\equiv 0
\end{equation*}
mod $2^{\alpha (p)}$. But this result 
$2^{\alpha (p)} | K$ contradicts our 
assumption $\mu (p) <0$. If 
$a>1$ we obtain the equality 
\begin{equation*}
K \equiv 2^{\alpha (p)-a} \cdot {\rm odd}
\end{equation*}
mod $2^{\alpha (p)-1}$. But 
it also contradicts $\mu (p) =-1$. 

\medskip 
\par \noindent 
This completes the proof. 
\end{proof}

\section{On non-convenient Newton polyhedra}\label{sec:6}

When dealing with a singularity $(f,0)$ with non-convenient 
Newton polyhedron $\Gamma_+(f)$, it happens already in dimension 
2 and 3 that one has to search for the monodromy eigenvalue at 
some point of the hypersurface $f^{-1}(0)$ close to the origin.

\begin{definition}[cf. \cite{E}] Let $f:(\CC^n,0)\to(\CC,0)$ 
be a germ of a holomorphic function. For all sufficiently 
small $x_0\in\CC^n$, the {\it nearby singularity} germ 
$$f_{x_0}:(\CC^n,x_0)\to(\CC,0),\quad f_{x_0}(x)=f(x_0+x),
$$ is well defined. We shall refer to the roots and poles 
of the monodromy $\zeta$-function of the latter germ as 
{\it nearby monodromy eigenvalues} of $f$.
\end{definition}

\subsection{Nearby singularities at coordinate lines}

Notice that the Newton polyhedron at a generic point of a $k$-dimensional 
coordinate plane is the product of the projection of the Newton 
polyhedron along that coordinate plane by $\mathbb{R}^k_+$. 
In this subsection, we prove the 
following generalization of 
\cite[Lemma 9]{L-V}. 

\begin{proposition}\label{NDPP} 
Assume that $f$ is non-degenerate
at the origin 
$0 \in \CC^n$, then except for finitely many $c \in \CC$ 
the polynomial $f_c(x)=f(x_1,x_2, \ldots, x_{n-1}, x_n+c)$ 
is non-degenerate at the origin 
$0 \in \CC^n$. 
\end{proposition} 

\begin{proof} 
Let $\pi : \RR^n \longrightarrow \RR^{n-1}$ be the 
projection along the last variable. Then 
except for finitely many $c \in \CC$ the Newton polyhedron 
$\Gamma_+(f_c)$ of $f_c$ is equal to 
the product $\pi (\Gamma_+(f)) \times \RR_+$. 
Let $\tau^{\prime} \prec \Gamma_+(f)$ be a 
face of $\Gamma_+(f)$ which is non-compact 
for the variable $v_n$ and denote 
its image by the projection 
$\pi : \RR^n \longrightarrow \RR^{n-1}$ by 
$\sigma \subset \RR^{n-1}$. Assume that 
$\sigma$ is compact. 
Here we shall treat only the case where 
$\tau^{\prime}$ is a facet and hence 
$\dim \sigma =n-2$. The other cases 
can be treated similarly. 
By a unimodular transformation of $\RR^n
=\RR^{n-1} \times \RR$ induced by that of 
its first factor $\RR^{n-1}$. 
we regard $\tau^{\prime}$ as a lattice polytope in 
its affine span ${\rm aff}( \tau^{\prime}) 
\simeq \RR^{n-1}$ 
and the $\tau^{\prime}$-part 
$f_{\tau^{\prime}}$ of $f$ as a Laurent polynomial 
on $T^{\prime} = ( \CC^*)^{n-1}_{x_1, \ldots, x_{n-1}}$, 
where the last variable $x_n$ of $f_{\tau^{\prime}}$ 
corresponds to the last one $x_{n-1}$ of the Laurent 
polynomial. We denote the latter 
also by $f_{\tau^{\prime}}$. Then by 
our assumption for any compact face $\tau$ of 
$\tau^{\prime}$ the hypersurface 
$\{ f_{\tau}=0 \} \subset T^{\prime}$ is smooth. Moreover the 
$\sigma$-part of the polynomial 
$f(x_1,x_2, \ldots, x_{n-1}, x_n+c)$ is 
naturally identified with the Laurent polynomial 
$f_{\tau^{\prime}}(x_1, \ldots, x_{n-2}, c)$. 
Therefore, in order to prove the assertion, 
by our previous description of $\Gamma_+(f_c)$ 
it suffices to show that 
except for finitely many 
$c \in \CC$ the hypersurface 
\begin{equation*}
W_c= \{ (x_1, \ldots, x_{n-2}) \ | \ 
f_{\tau^{\prime}}(x_1, \ldots, x_{n-2}, c) 
=0 \} \subset ( \CC^*)^{n-2}
\end{equation*}
in $T^{\prime} \cap \{ x_{n-1}=c \} \simeq 
( \CC^*)^{n-2}$ is smooth. 
Let $h: T^{\prime}= ( \CC^*)^{n-1} \longrightarrow 
\CC$ be the function defined by 
$h(x_1, \ldots, x_{n-1})=x_{n-1}$. 
Then the set of $c \in \CC$ for which 
$W_c \subset ( \CC^*)^{n-2}$ is not 
smooth is contained in the 
discriminant variety of the map 
$h|_{\{ f_{\tau^{\prime}}=0 \}} : 
\{ f_{\tau^{\prime}}=0 \} \longrightarrow \CC$. 
For $\e >0$ let $B(0; \e )^*= \{ c \in \CC \ | \ 
0 < |c| < \e \}$ be the punctured disk 
centered at the origin $0 \in \CC$. 
Then there exists a sufficiently small 
$0 < \e \ll 1$ such that the hypersurface 
$W_c \subset ( \CC^*)^{n-2}$ is smooth 
for any $c \in B(0; \e )^*$. 
Indeed, let $\Delta =  \tau^{\prime} \cap 
\{ v_n \leq l \} \subset \tau^{\prime}$ 
($l \gg 0$) be the truncation of $\tau^{\prime}$. 
Let $\Sigma_0$ be 
the dual fan of the $(n-1)$-dimensional 
polytope $\Delta$ in $\RR^{n-1}$ and 
$\Sigma$ its smooth subdivision. We denote by 
$X_{\Sigma}$ the toric variety associated to 
$\Sigma$ (see \cite{Fulton} and \cite{Oda} etc.). 
Then $X_{\Sigma}$ is a smooth 
compactification of $T^{\prime} = ( \CC^*)^{n-1}$.
Recall that $T^{\prime} = ( \CC^*)^{n-1}$ 
acts naturally on $X_{\Sigma}$ and the 
$T^{\prime}$-orbits in it are parametrized by 
the cones in the smooth fan $\Sigma$. 
For a cone $C \in \Sigma$ denote by 
$T_C \simeq ( \CC^*)^{n-1- \dim C} \subset 
X_{\Sigma}$ the $T^{\prime}$-orbit associated 
to $C$. By our assumption above, if $C \in \Sigma$ 
corresponds to a compact face $\tau$ of 
$\tau^{\prime}$ then the hypersurface 
$W= \overline{  \{ f_{\tau^{\prime}}=0 \} } 
\subset X_{\Sigma}$ intersects 
$T_C \subset X_{\Sigma}$ transversally. 
We denote the meromorphic 
extension of 
$h: T^{\prime}= ( \CC^*)^{n-1} \longrightarrow 
\CC$ to $X_{\Sigma}$ by the same letter $h$. 
Note that $h$ has no point of indeterminacy 
on the whole $X_{\Sigma}$ (because 
it is a monomial). 
Then as $|c| \longrightarrow 
0$ the level set $h^{-1}(c) \subset 
X_{\Sigma}$ of $h$ tends to the union of 
the $T^{\prime}$-orbits which correspond to 
the compact faces of $\tau^{\prime}$. 
More precisely, if a cone $C \in \Sigma$ 
corresponds to a compact face of 
$\tau^{\prime}$ then there exists an affine 
chart $\CC^{n-1}_y$ of $X_{\Sigma}$ on which 
\begin{equation*}
T_C=  \Bigl\{ y=(y_1, \ldots, y_{n-1} ) \ | \ 
y_i=0 \ (1 \leq i \leq \dim C), \ 
y_i \not=0 \ ( \dim C +1 \leq i \leq n-1) \Bigr\} 
\end{equation*}
and $h(y)=y_1^{m_1} y_2^{m_2} \cdots y_k^{m_k}$ 
($m_i \in \ZZ_{>0}$) for some 
$k \geq 1$. 
By this explicit description of $h$ we see 
that for $0< |c| \ll 1$ the hypersurface 
$h^{-1}(c)$ intersects $W$ 
transversally. It follows that 
\begin{equation*}
W_c= W \cap h^{-1}(c) \cap T^{\prime} 
 \subset h^{-1}(c) \cap T^{\prime} 
\simeq ( \CC^*)^{n-2}
\end{equation*}
is smooth for $0< |c| \ll 1$. 
This completes the proof. 
\end{proof}

Note also that at almost all points 
on a coordinate axis contained 
in the hypersurface, the compact part of the Newton 
polyhedron there coincides with the compact part of 
the projection of $\Gamma_+(f)$ along that coordinate 
axis. Then the monodromy zeta function can be computed 
by the same Varchenko formula in one dimension less, since 
by Proposition \ref{NDPP} 
generic nearby singularity germs are 
still non-degenerate.

\begin{example}\label{exa2dim1}
1) If $f(x_1,x_2)=x_1^{a_1}x_2^{a_2}g(x_1,x_2)$ with $g$ not divisible by $x_i$, then we have the nearby eigenvalue $\sqrt[a_i]{1}$ on the $i$-th axis.

2) If $f(x_1,x_2,x_3)=x_1^{a_1}x_2^{a_2}x_3^{a_3}g(x_1,x_2,x_3)$ with $g$ not divisible by $x_i$, then we have the nearby eigenvalue $\sqrt[a_i]{1}$ at every point of the $i$-th coordinate plane except for the points of the coordinate lines and, most notably, the surface $g=0$.
\end{example}

\subsection{Nearby singularities outside coordinate lines}

The following example shows that from dimension four on, one 
might not always find the eigenvalue of monodromy corresponding to a pole of the topological zeta function at a point on a coordinate axis or even a generic point on a coordinate plane (a subtle shadow of this difference between generic and not so generic points of a coordinate plane can be seen already in dimension 3, see Example \ref{exa2dim1}).

\begin{example} \label{exnon0convenient}
Consider the polynomial $f(x_1,x_2,x_3,x_4)=x_3^6+x_2^4 x_3^5 +x_1^2 
x_2^{13} x_3^2 + x_2^{13} x_3^2 x_4^2$, which is 
non-degenerate at the origin. One finds that $-1/3$ 
is a pole of
$Z_{ {\rm top}, f}(s)$, contributed by the only compact facet of 
$\Gamma_+(f)$. For the zeta function of monodromy at the origin one finds
$$\zeta_{f,0} (t) = \frac{1-t^6}{1-t^{24}},$$
and so $e^{-2i \pi/3 }$ is not a zero or pole of this function. 
One can check that there does not exist a point $b=(c_1,c_2,c_3,c_4)$ 
in $\{f=0\}$ such that the compact part of the Newton polyhedron 
$\Gamma_+(g)$ of $g(x_1,x_2,x_3,x_4)=f(x_1-c_1,x_2-c_2,x_3-c_3,x_4-c_4)$ is a 
projection of $\Gamma_+(f)$ along the minimal coordinate plane containing $b$,
and $e^{-2i \pi/3 }$ is a zero or pole of $\zeta_{f,b}(t)$. 
However, $e^{-2i \pi/3 }$ is an eigenvalue of monodromy at 
the points of the curve $\mathcal{C}=\{(c,0,0,-ic),  c \in 
\mathbb{C}^* \} \subset \{f=0\}$. Note that these are exactly the points where $\Gamma_+(g)$ is strictly smaller that the projection of $\Gamma_+(f)$, due to a cancellation of two monomials in $f$. \qed
\end{example}

In this particular example, one can check that the singularity is still non-degenerate at the points of the curve $\mathcal{C}$ where we found the eigenvalue of monodromy.
However, this will not always be the case, as shown in the next example.

\begin{example}\label{exquitnondegenerate}
We consider $g(x_1,x_2,x_3,x_4)= x_3^6 + x_2^4 x_3^5 + x_1^2 x_2^{13} x_3^2 + x_2^{13} x_3^2 x_4^2 + 2 x_1^{100} x_2^7 x_3^4 x_4^{100} + x_1^{200} x_2^{14} x_3^2 x_4^{200}$, which up to terms of higher order equals the polynomial $f$ 
considered in Example \ref{exnon0convenient}. 
The polynomial $g$ is non-degenerate at the origin, and its Newton polyhedron at the origin has one compact facet, spanned by the vertices $A(0,13,2,2),B(2,13,2,0),C(0,4,5,0)$ and $D(0,0,6,0)$. It contributes the candidate pole $-1/3$ which is a pole of 
$Z_{top,g}(s)$. For the zeta function of monodromy at the origin one finds 
$$\zeta_{g,0}(t)=\frac{1-t^6}{1-t^{24}}.$$ 
In Example \ref{exnon0convenient} we found the eigenvalue of monodromy $e^{-2\pi i/3}$ at the points $(c, 0, 0, \pm i c), c \in \CC$.
In the translated local coordinates at the point $(c, 0, 0, \pm i c)$, the principal part of $g$ will be 
$$x_3^6 + x_2^4 x_3^5 + 2 c^{200} x_2^7 x_3^4 + c^{400} x_2^{14} x_3^2 + 2 c x_1 x_2^{13} x_3^2 - 2 i c x_2^{13} x_3^2 x_4,$$
having a degenerate edge $x_3^6 + 2 c^{200} x_2^7 x_3^4 + c^{400} x_2^{14} x_3^2$.
 \qed
\end{example}

This example shows we have to quit the non-degenerate setting to prove the monodromy conjecture for non-degenerate singularities in dimension 4 and higher. 
It also shows that the Newton polyhedron of a singularity at an adjacent singular point may depend not only on the Newton polyhedron of the initial singularity, but also on higher order terms.

This obstacle motivated the first author to introduce the notion of tropical monodromy eigenvalues (see \cite{E}). The main result in \cite{E} makes it possible to find some of the nearby monodromy eigenvalues outside the coordinate axes, given only the Newton polyhedron of a non-degenerate singularity at the origin.
We recall this result and restrict to dimension four from now on.

Assume that $f(x) \in \CC [x_1, \ldots, x_4]$ 
is non-degenerate at the origin $0 \in \CC^4$. 
Pick some pole of the topological zeta function $s_0$ and denote the corresponding candidate eigenvalue of the Milnor monodromy by $t_0=\exp(2\pi i s_0)$.

We suppose that there is a V-vertex $A$ contained in an unbounded face, contributing to the eigenvalue $t_0$. 
With no loss in generality, assume that $A$ is on the coordinate axis $O_1$. 

Let $I \subset \{1,\ldots,4\}$, $I \neq \emptyset$ and $I \neq \{1,\ldots,4\}$. We will denote $\pi_I=\pi_{\{x_i\}_{i \in I}}$ for the projection map 
\begin{eqnarray*}
\RR^4 & \longrightarrow  & \RR^{4-\vert I \vert} \\
(x_1,x_2,x_3,x_4) & \longmapsto  & (\hat{x}_1,\hat{x}_2,\hat{x}_3,\hat{x}_4),
\end{eqnarray*}
where $\hat{x_i}$ means that $x_i$ is removed if and only if $i \in I$. When $I=\{x_i\}$ is a singleton, we will also write $\pi_{x_i}$.

\begin{theorem}\label{nearby}\cite[Cor. 6.15]{E}
Let $\pi_I : \ZZ^4 \to \ZZ^2$ be the projection map. Let $N$ be the projection of the Newton polyhedron $\Gamma_+(f)$ under $\pi_I$.
Let a compact edge $E$ of the polygon $N$ be the projection of a two-dimensional compact face $F$ of $\Gamma_+(f)$. Denote the lattice distance from $E$ to the origin by $d$ and some root of the polynomial $t^d-1$ by $t'_0$. 
If $E$ is not a segment of lattice length $1$ such that exactly one of its end points is a $V$-vertex of $N$ contributing to $t'_0$ and such that another end point has a unique preimage in $F$, then 
$t'_0$ is a monodromy eigenvalue of the germ of $f$ at a non-zero point of $\{f=0\}$. 

More specifically, when $I = \{x_1,x_2\}$, then there exists a curve $C_{E,N,t'_0}$ through the origin in the coordinate plane $x_3=x_4=0$ (and outside the axes $O_1$ and $O_2$) such that $t'_0$ is a monodromy eigenvalue of the germ of $f$ at a generic point of $C_{E,N,t'_0}$.
\end{theorem}
\begin{remark}\label{remnearby}
\begin{enumerate}
\item If the face $F$ contains a V-vertex on a coordinate axis out of $I$ contributing to $t'_0$, then the condition in Theorem \ref{nearby} is fulfilled. 
\item The nearby monodromy eigenvalues provided by this theorem are called {\it tropical}, because the proof of their existence in \cite{E} is based on the calculus of so called tropical characteristic classes.
\end{enumerate}
\end{remark}
For instance, this theorem allows to find the ``complicated'' nearby monodromy eigenvalue in Example \ref{exquitnondegenerate}.

\section{The monodromy conjecture for $n=4$}\label{sec:7}

Assume that $f(x) \in \CC [x_1, \ldots, x_4]$ 
is non-degenerate at the origin $0 \in \CC^4$. 
Pick some pole of the topological zeta function $s_0$ and denote the corresponding candidate eigenvalue of the Milnor monodromy by $t_0=\exp(2\pi i s_0)$.
Our aim is to prove that, once Theorem \ref{mainB} does not guarantee that $s_0$ is fake, $t_0$ is a root or pole of the monodromy $\zeta$-function of a singularity of $f$ at some point near the origin.

For the rest of the paper, we may assume w.l.o.g. the compactness of every $B$-border, adjacent to two $B$-facets contributing the pole $s_0$. Indeed, towards the contradiction, if at least one such border $\tau$ were non-compact, then we would have one of the following (up to reordering the coordinates):

-- if $\tau$ is the Minkowski sum of the point $(1,1,*,*)$ and the 3rd and 4th coordinate rays, then $s_0=-1$;

-- if $\tau$ is the Minkowski sum of the segment  from $(1,1,*,*)$ to $(0,0,*,*)$ and the 4th coordinate ray, then its projection along the 4th coordinate is a segment from $(1,1,*)$ to $(0,0,*)$, adjacent to two $B$-faces contributing $s_0$ in the Newton polyhedron of $f(x_1,x_2,x_3,x_4+\varepsilon)$ for a small constant $\varepsilon\ne 0$. Then \cite{L-V} ensures that in this case $\exp(2\pi i s_0)$ is a monodromy eigenvalue of the singularity of the function $f$ at $(0,0,0,\varepsilon)$.

In the second subsection we will see how \cite{E} allows us to 
isolate many cases of the combinatorial structure of the Newton polyhedron $\Gamma_+(f)$, which ensure that $t_0$ is a nearby monodromy eigenvalue outside the origin.

In the third subsection, we continue this work in the presence of a triangulation of $\Gamma_+(f)$ (which can be constructed only if $\Gamma_+(f)$ does not fall within the scope of the second subsection). In the fourth subsection,  we subdivide the $V$-pieces of this triangulation into groups such that each of them `contributes' a nonnegative multiplicity to $t_0$ as an eigenvalue of the monodromy $\zeta$-function at the origin, in the sense of the following definition.

\begin{definition} 
Recall that a $V$-face $F$ is said to contribute to the eigenvalue $t_0$ if $t_0^{N(F)}=1$. The number $(-1)^{\codim F-1}\Vol_\ZZ(F)$ for a contributing $F$ and 0 for a non-contributing $F$ is called \emph{the contribution of $F$}. The sum of contributions of all faces from some set of faces $S$ is called \emph{the contribution of $S$}.
\end{definition}

Finally, in the last subsection, we show that once one of the aforementioned groups violates the assumptions of Theorem \ref{mainB}, its contribution is strictly positive. This proves the monodromy conjecture for non-degenerate singularities of 4 variables. 

The aforementioned subdivision of facets into groups is based on the fact that every $V$-simplex $F$ has a unique minimal corner (possibly equal to $F$), and the following notion.

\begin{definition} \label{deffamily} For any triangulation of the union of compact faces of $\Gamma_+(f)$, the \emph{family of a $V$-simplex $F$} of this triangulation is the set of all faces of $F$, containing its minimal corner (note that all of them are $V$-faces by definition of the corner). A family is said to be trivial if it consists of one element. The dimension of the maximal simplex in a family is also referred to as the \emph{dimension of the family}.
\end{definition}

\begin{example}
We illustrate the notion of family of a $V$-simplex by an example in dimension $3$.
Let $P=(0,0,a), Q=(0,b,c), R=(d,0,e)$ with $a, b, c, d$ and $e$ non-zero and let $S$ be a vertex outside the coordinate planes such that $PQS$ and $PRS$ are different $V$-triangles containing $P$. Then the minimal corner of $PQS$ is $PQ$ and the family of $PQS$ is $\{PQS,PQ\}$.
The minimal corner of $PRS$ is $PR$. Hence the family of $PRS$ is $\{PRS,PR\}$.
\end{example}

%\begin{remark}
Recall that, by Theorem \ref{Key-5}, the contribution of every family to every eigenvalue is non-negative or non-positive, depending on the dimension of the family.
%\end{remark}

Before implementing our general plan, we devote the first subsection to a sandbox three-dimensional version of this story (first, in order to illustrate a significantly more complicated four-dimensional case beforehand, and, secondly, because we shall need this three-dimensional statement anyway). Although the three-dimensional result is essentially covered by \cite{L-V}, the logic of our reasoning is different, and this difference becomes important in higher dimensions.
Indeed, the result in \cite{E} makes us to approach the monodromy conjecture for non-isolated singularities first by searching for the monodromy eigenvalue outside the origin and then, if necessary, at the origin (having already excluded many combinatorial possibilities for the structure of the Newton polytope).

\subsection{Two- and three-dimensional case}

We intentionally formulate things in a more complicated way than we could for two or three variables, in order to keep all the wording consistent with the four-dimensional case.

\begin{theorem} \label{sandbox2} For every non-degenerate $f\in{\mathbb C}[x,y]$, if a family $F$ has positive contribution to the candidate eigenvalue $t_0$, then $t_0$ is a nearby monodromy eigenvalue for $f$. 
\end{theorem}
\begin{proof} Note that positive contribution implies that the family $F$ is 1-dimensional (not 0-dimensional).

Step 1: looking for the monodromy eigenvalue outside the origin (c.f. Example \ref{exa2dim1}). 

If $\Gamma_+(f)$ has a $V$-vertex that contributes to the eigenvalue $t_0$ and is not contained in a compact $V$-edge, then $t_0$ is obviously a nearby monodromy eigenvalue of $f$ at a point of a coordinate axis.

Step 2: splitting $\Gamma_+(f)$ into families otherwise. 

Define the register of $\Gamma_+(f)$ as the set of families of all $V$-edges. Neglecting the cases covered by Step 1, we notice that every $V$-simplex enters exactly one family in the register. 

All families in the register have non-negative contribution, and the family $F$ has positive contribution, thus the total contribution to the eigenvalue $t_0$ is positive.
\end{proof}

\begin{theorem} \label{sandbox3} For every non-degenerate $f\in{\mathbb C}[x,y,z]$ and a triangulation of the compact faces of its Newton polyhedron, if a family $F$ of this triangulation has positive contribution to the candidate eigenvalue $t_0$, then $t_0$ is a nearby monodromy eigenvalue for $f$, unless all compact $V$-simplices containing $F$ 
are contained in 
the same (larger) family whose contribution is 0.
\end{theorem}
\begin{proof} Step 1: looking for the monodromy eigenvalue outside the origin.

If $\Gamma_+(f)$ has a $V$-edge $G$, whose family has non-zero contribution to the eigenvalue $t_0$, and which is not contained in a compact $V$-triangle, then such face is contained in a non-compact face, parallel to a coordinate axis (say, $O_1$) and not contained in a coordinate plane. Then $G$ is also a $V$-edge of the projection of $\Gamma_+(f)$ along $O_1$, satisfying the assumption of the preceding theorem (recall Proposition \ref{NDPP}), so $t_0$ is a nearby monodromy eigenvalue of $f$ at a point of axis $O_1$.

Step 2: splitting $\Gamma_+(f)$ into families otherwise.

Define the register of $\Gamma_+(f)$ as the set of

-- families of all $V$-triangles,

-- for all $V$-edges outside the aforementioned families, the families of these $V$-edges;

-- (families of) all $V$-vertices outside the aforementioned families. 

Every $V$-simplex enters exactly one family on the register. In particular, assume towards contradiction that a $V$-vertex $A$ is contained in several families. If one of them contains the others, then the others are not on the register; otherwise $A$ is contained in two 1-dimensional families, but then none of them is on the register. 

Neglecting the cases covered by Step 1, there are no families of $V$-edges that are not contained in $V$-triangles and have non-zero contribution. Thus the register contains only even-dimensional families and families having zero contribution.

By Theorem \ref{Key-5} the contribution of every even-dimensional family to the multiplicity of $t_0$ is non-negative. Thus all families on the register have non-negative contribution.

Moreover, at least one of them has positive contribution: if $F$ is a 2-dimensional family, then it is on the register with positive contribution.
If $F$ is 0-dimensional, and all compact faces containing $F$ are contained in a larger family, then the larger family is on the register with non-zero (i.e. positive) contribution. If $F$ is 0-dimensional otherwise, then it is on the register with positive contribution.
\end{proof}

\subsection{Dimension 4 : Looking for the eigenvalue $t_0$ outside the origin.}
We suppose that there is a V-vertex $A$ contributing to the eigenvalue $t_0$, and we assume that $A$ is on the coordinate axis $O_1$. 

In this subsection we start to exploit as much as possible Proposition 
\ref{NDPP} and Theorem \ref{nearby} to derive properties on the combinatorial structure of the Newton polyhedron locally around $A$.

\begin{definition} The \emph{link} $L_A$ is the subdivision of the triangle $T_A$ defined by
$$x_1=0,\, x_2+x_3+x_4=\varepsilon, x_2\geqslant 0,\, x_3\geqslant 0,\, x_4\geqslant 0 $$
into the isomorphic images of faces of $\Gamma_+(f)$ intersected with the hyperplane $x_2+x_3+x_4=\varepsilon$ under the projection $\pi_{x_1}$, for $\varepsilon>0$ small enough. (The link does not depend on the choice of $\varepsilon>0$ in the sense that the links for all $\varepsilon>0$ small enough are affine isomorphic to each other.)

The image of a face $F\subset\Gamma_+(f)$ containing $A$ in the link $L_A$ is referred to as the \emph{link of $F$} in $L_A$.
\end{definition}
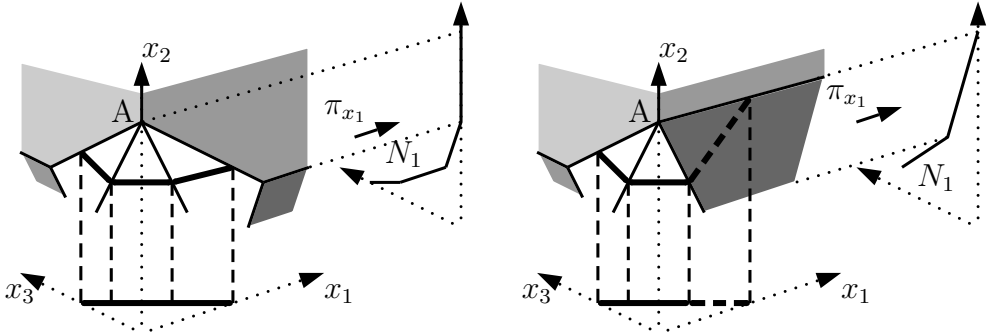
\begin{figure}\begin{center}
\skipfig{\psscalebox{1.0 1.0} 
{
\begin{pspicture}(0,-2.1114144)(13.06,2.1114144)
\definecolor{colour0}{rgb}{0.6,0.6,0.6}
\definecolor{colour1}{rgb}{0.4,0.4,0.4}
\definecolor{colour2}{rgb}{0.8,0.8,0.8}
\psline[linecolor=white, linewidth=0.04, fillstyle=solid,fillcolor=colour0](0.6,0.11141449)(0.2,0.31141448)(0.4,-0.08858551)(0.8,-0.2885855)(0.6,0.11141449)
\psline[linecolor=white, linewidth=0.04, fillstyle=solid,fillcolor=colour1](3.4,-0.08858551)(3.2,-0.6885855)(3.8,-0.4885855)(4.0,0.11141449)
\psline[linecolor=white, linewidth=0.04, fillstyle=solid,fillcolor=colour2](1.8,1.1114144)(1.8,0.7114145)(0.6,0.11141449)(0.2,0.31141448)(0.2,1.5114145)(1.8,1.1114144)
\psline[linecolor=white, linewidth=0.04, fillstyle=solid,fillcolor=colour0](1.8,1.1114144)(1.8,0.7114145)(3.4,-0.08858551)(4.0,0.11141449)(4.0,1.7114145)(1.8,1.1114144)
\psline[linecolor=black, linewidth=0.04, linestyle=dotted, dotsep=0.10583334cm, arrowsize=0.05291667cm 4.0,arrowlength=1.5,arrowinset=0.0]{->}(1.8,-2.0885856)(1.8,1.5114145)
\psline[linecolor=black, linewidth=0.04, linestyle=dotted, dotsep=0.10583334cm, arrowsize=0.05291667cm 4.0,arrowlength=1.5,arrowinset=0.0]{->}(1.8,-2.0885856)(4.2,-1.2885855)
\psline[linecolor=black, linewidth=0.04, linestyle=dotted, dotsep=0.10583334cm, arrowsize=0.05291667cm 4.0,arrowlength=1.5,arrowinset=0.0]{->}(1.8,-2.0885856)(0.2,-1.2885855)
\psline[linecolor=black, linewidth=0.08, dotsize=0.07055555cm 2.0,dotsize=0.07055555cm 2.0]{cc-cc}(1.0,-1.6885855)(3.0,-1.6885855)
\psline[linecolor=black, linewidth=0.04](1.8,0.7114145)(0.6,0.11141449)
\psline[linecolor=black, linewidth=0.04](1.8,0.7114145)(1.2,-0.4885855)
\psline[linecolor=black, linewidth=0.04](1.8,0.7114145)(2.4,-0.4885855)
\psline[linecolor=black, linewidth=0.04](1.8,0.7114145)(3.4,-0.08858551)
\psline[linecolor=black, linewidth=0.04](3.4,-0.08858551)(4.0,0.11141449)
\psline[linecolor=black, linewidth=0.04](3.4,-0.08858551)(3.2,-0.6885855)
\psline[linecolor=black, linewidth=0.08, dotsize=0.07055555cm 2.0,dotsize=0.07055555cm 2.0]{cc-cc}(1.0,0.31141448)(1.4,-0.08858551)
\psline[linecolor=black, linewidth=0.08, dotsize=0.07055555cm 2.0,dotsize=0.07055555cm 2.0]{cc-cc}(1.4,-0.08858551)(2.2,-0.08858551)
\psline[linecolor=black, linewidth=0.08, dotsize=0.07055555cm 2.0,dotsize=0.07055555cm 2.0]{cc-cc}(2.2,-0.08858551)(3.0,0.11141449)
\psline[linecolor=black, linewidth=0.04, linestyle=dashed, dash=0.17638889cm 0.10583334cm](1.0,0.31141448)(1.0,-1.6885855)
\psline[linecolor=black, linewidth=0.04, linestyle=dashed, dash=0.17638889cm 0.10583334cm](1.4,-0.08858551)(1.4,-1.6885855)
\psline[linecolor=black, linewidth=0.04, linestyle=dashed, dash=0.17638889cm 0.10583334cm](2.2,-0.08858551)(2.2,-1.6885855)
\psline[linecolor=black, linewidth=0.04, linestyle=dashed, dash=0.17638889cm 0.10583334cm](3.0,0.11141449)(3.0,-1.6885855)
\psline[linecolor=black, linewidth=0.04](1.8,0.7114145)(1.8,1.3114145)
\psline[linecolor=black, linewidth=0.04](0.6,0.11141449)(0.2,0.31141448)
\psline[linecolor=black, linewidth=0.04](0.6,0.11141449)(0.8,-0.2885855)
\psline[linecolor=black, linewidth=0.04, arrowsize=0.05291667cm 4.0,arrowlength=1.5,arrowinset=0.0]{->}(4.6,0.51141447)(5.2,0.7114145)
\psline[linecolor=black, linewidth=0.04, linestyle=dotted, dotsep=0.10583334cm, arrowsize=0.05291667cm 4.0,arrowlength=1.5,arrowinset=0.0]{->}(6.0,-0.6885855)(4.4,0.11141449)
\psline[linecolor=black, linewidth=0.04, linestyle=dotted, dotsep=0.10583334cm, arrowsize=0.05291667cm 4.0,arrowlength=1.5,arrowinset=0.0]{->}(6.0,-0.6885855)(6.0,2.3114145)
\psline[linecolor=black, linewidth=0.04, linestyle=dotted, dotsep=0.10583334cm](1.8,0.7114145)(6.0,1.9114145)
\psline[linecolor=black, linewidth=0.04, linestyle=dotted, dotsep=0.10583334cm](4.0,0.11141449)(6.0,0.7114145)
\psline[linecolor=black, linewidth=0.04](6.0,2.1114144)(6.0,0.7114145)(5.8,0.11141449)
\psline[linecolor=black, linewidth=0.04](5.8,0.11141449)(5.2,-0.08858551)
\psline[linecolor=black, linewidth=0.04](5.2,-0.08858551)(4.8,-0.08858551)
\rput[bl](4.2,-1.6885855){$x_2$}
\rput[bl](1.8,1.5114145){$x_1$}
\rput[bl](0.0,-1.6885855){$x_3$}
\rput[bl](4.2,0.7114145){$\pi_{x_2}$}
\rput[bl](5.0,0.11141449){$N_2$}
\rput[bl](1.4,0.7114145){A}
\psline[linecolor=white, linewidth=0.04, fillstyle=solid,fillcolor=colour0](7.4,0.11141449)(7.0,0.31141448)(7.2,-0.08858551)(7.6,-0.2885855)(7.4,0.11141449)
\psline[linecolor=white, linewidth=0.04, fillstyle=solid,fillcolor=colour1](8.6,0.7114145)(9.2,-0.4885855)(10.4,-0.08858551)(10.78,1.3114145)
\psline[linecolor=white, linewidth=0.04, fillstyle=solid,fillcolor=colour2](8.6,1.1114144)(8.6,0.7114145)(7.4,0.11141449)(7.0,0.31141448)(7.0,1.5114145)(8.6,1.1114144)
\psline[linecolor=white, linewidth=0.04, fillstyle=solid,fillcolor=colour0](8.6,1.1114144)(8.6,0.7114145)(10.8,1.3114145)(10.8,1.3114145)(10.8,1.7114145)(8.6,1.1114144)
\psline[linecolor=black, linewidth=0.04, linestyle=dotted, dotsep=0.10583334cm, arrowsize=0.05291667cm 4.0,arrowlength=1.5,arrowinset=0.0]{->}(8.6,-2.0885856)(8.6,1.5114145)
\psline[linecolor=black, linewidth=0.04, linestyle=dotted, dotsep=0.10583334cm, arrowsize=0.05291667cm 4.0,arrowlength=1.5,arrowinset=0.0]{->}(8.6,-2.0885856)(11.0,-1.2885855)
\psline[linecolor=black, linewidth=0.04, linestyle=dotted, dotsep=0.10583334cm, arrowsize=0.05291667cm 4.0,arrowlength=1.5,arrowinset=0.0]{->}(8.6,-2.0885856)(7.0,-1.2885855)
\psline[linecolor=black, linewidth=0.08, dotsize=0.07055555cm 2.0,dotsize=0.07055555cm 2.0]{cc-cc}(7.8,-1.6885855)(9.0,-1.6885855)
\psline[linecolor=black, linewidth=0.04](8.6,0.7114145)(7.4,0.11141449)
\psline[linecolor=black, linewidth=0.04](8.6,0.7114145)(8.0,-0.4885855)
\psline[linecolor=black, linewidth=0.04](8.6,0.7114145)(9.2,-0.4885855)
\psline[linecolor=black, linewidth=0.04](8.6,0.7114145)(10.78,1.3114145)
\psline[linecolor=black, linewidth=0.08, dotsize=0.07055555cm 2.0,dotsize=0.07055555cm 2.0]{cc-cc}(7.8,0.31141448)(8.2,-0.08858551)
\psline[linecolor=black, linewidth=0.08, dotsize=0.07055555cm 2.0,dotsize=0.07055555cm 2.0]{cc-cc}(8.2,-0.08858551)(9.0,-0.08858551)
\psline[linecolor=black, linewidth=0.08, linestyle=dashed, dash=0.17638889cm 0.10583334cm, dotsize=0.07055555cm 2.0,dotsize=0.07055555cm 2.0]{cc-cc}(9.0,-0.08858551)(9.8,1.0314145)
\psline[linecolor=black, linewidth=0.04, linestyle=dashed, dash=0.17638889cm 0.10583334cm](7.8,0.31141448)(7.8,-1.6885855)
\psline[linecolor=black, linewidth=0.04, linestyle=dashed, dash=0.17638889cm 0.10583334cm](8.2,-0.08858551)(8.2,-1.6885855)
\psline[linecolor=black, linewidth=0.04, linestyle=dashed, dash=0.17638889cm 0.10583334cm](9.0,-0.08858551)(9.0,-1.6885855)
\psline[linecolor=black, linewidth=0.04, linestyle=dashed, dash=0.17638889cm 0.10583334cm](9.8,0.9114145)(9.8,-1.6885855)
\psline[linecolor=black, linewidth=0.04](8.6,0.7114145)(8.6,1.3114145)
\psline[linecolor=black, linewidth=0.04](7.4,0.11141449)(7.0,0.31141448)
\psline[linecolor=black, linewidth=0.04](7.4,0.11141449)(7.6,-0.2885855)
\psline[linecolor=black, linewidth=0.04, arrowsize=0.05291667cm 4.0,arrowlength=1.5,arrowinset=0.0]{->}(11.2,0.7114145)(11.8,0.9114145)
\psline[linecolor=black, linewidth=0.04, linestyle=dotted, dotsep=0.10583334cm, arrowsize=0.05291667cm 4.0,arrowlength=1.5,arrowinset=0.0]{->}(12.8,-0.6885855)(11.2,0.11141449)
\psline[linecolor=black, linewidth=0.04, linestyle=dotted, dotsep=0.10583334cm, arrowsize=0.05291667cm 4.0,arrowlength=1.5,arrowinset=0.0]{->}(12.8,-0.6885855)(12.8,2.3114145)
\psline[linecolor=black, linewidth=0.04, linestyle=dotted, dotsep=0.10583334cm](10.68,1.2914145)(12.8,1.9114145)
\psline[linecolor=black, linewidth=0.04, linestyle=dotted, dotsep=0.10583334cm](10.4,-0.08858551)(12.4,0.51141447)
\psline[linecolor=black, linewidth=0.04](12.8,1.9114145)(12.4,0.51141447)(11.8,0.11141449)
\rput[bl](11.0,-1.6885855){$x_2$}
\rput[bl](8.6,1.5114145){$x_1$}
\rput[bl](6.8,-1.6885855){$x_3$}
\rput[bl](10.8,0.9114145){$\pi_{x_2}$}
\rput[bl](12.0,-0.2){$N_2$}
\rput[bl](8.2,0.7114145){A}
\psline[linecolor=black, linewidth=0.08, linestyle=dashed, dash=0.17638889cm 0.10583334cm, dotsize=0.07055555cm 2.0,dotsize=0.07055555cm 2.0]{cc-cc}(9.0,-1.6885855)(9.8,-1.6885855)
\end{pspicture}
}}\end{center}
\caption{links of Newton polyhedra}\label{piclink}
\end{figure}
For example, the link of a vertex of a Newton polyhedron is shown in bold on Figure \ref{piclink}. 

The following fact seems to be common knowledge, but we give a proof as we have not found an exact reference, and the fact is not entirely tautological. 
\begin{proposition}
The union of the relative interiors of the links of bounded faces in $L_A$ is closed and contractible.
\end{proposition}
\begin{proof}
We denote $A(a,0,0,0)$. By taking a slice of the Newton polyhedron $\Gamma_+(f)$ with the hyperplane $x_1=a-\epsilon$, it is sufficient to prove that the union of the compact faces $\mathcal{C}$ of every Newton polyhedron $N$ is contractible.

Let $N_\epsilon$ be the Minkowski sum of $N$ and a ball of radius $\epsilon$. The union of its compact faces $\mathcal{C}_\epsilon$ (i.e. the set of those boundary points that do not belong to a ray of boundary points) is a topological disc, because the homeomorphism with the standard simplex is provided by the Gauss-Bonnet map (sending every boundary point to its unit exterior normal vector). Now the family of sets $\mathcal{C}_\epsilon$ is a family of vanishing neighborhoods of $\mathcal{C}$, so the contractibility of all $\mathcal{C}_\epsilon$ implies the contractibility of $\mathcal{C}$.
\end{proof}
Recall that a ray $r\in\RR^n$ is said to belong to the recession cone of a set $S\in\RR^n$, if the Minkowsky sum of $r$ and $S$ equals $S$.
\begin{definition} \label{defiface} A face of $\Gamma_+(f)$ is called an \emph{at least $I$-face (resp. at most $I$-face)}, $I\subset\{1,2,3,4\}$, if its recession cone contains the positive coordinate axes $O_i,\, i\in I$ (resp. its recession cone does not contain any coordinate axe $O_i$ if $i \notin I$), and is called an (exactly) $I$-face, if it is at least and at most $I$-face. This terminology transfers to the corresponding pieces of the link $L_A$.
\end{definition}
For example, on Figure \ref{piclink}, unbounded facets of the Newton polyhedron are grey, and a $2$-piece of the link, corresponding to an unbounded $1$-facet, is shown in bold dashes.
\begin{definition}
Let $v_i$ be the vertex of the triangle $T_A$ opposite to its edge $x_i=0$.
Let $l$ be a line separating $v_i$ from the other vertices of the pieces of $L_A$ (i.e. passing close enough to $v_i$). Then the subdivision of the segment $l\cap T_A$ into its intersections with the pieces of $L_A$ is independent of the choice of $l$ (up to a projective transformation) and is called the \emph{link of the vertex $v_i$ in the link $L_A$}.
\end{definition}

\begin{remark}\label{remlink} Let $v_i$ be the vertex of the triangle $T_A$ opposite to its edge $x_i=0$. Denote by $N_i$ the projection $\pi_{x_i}\Gamma_+(f)$ along $O_i$. Notice that for almost every $c \in \mathbb{C}$, the polynomial $f(x_1,\ldots,x_{i-1},x_i-c,x_{i+1},\ldots,x_n)$ has $N_i$ as Newton polyhedron and is non-degenerate by Proposition \ref{NDPP}.
\begin{enumerate}
\item If the vertex $v_i$ as a piece of the subdivision $L_A$ corresponds to a bounded edge of $\Gamma_+(f)$, then no pieces of the link $L_A$ correspond to at least $i$-faces. In this case the point $\pi_{x_i}(A)$ is not a vertex of $N_i$.
\item Otherwise, the pieces of the link $L_A$, containing $v_i$, are exactly the pieces corresponding to at least $i$-faces. 
Then the point $\pi_{x_i}(A)$ is a vertex $A_i$, whose link is (projectively) isomorphic to a subdivision of the link of $v_i$ in $L_A$.
\end{enumerate}
\end{remark}
See Figure \ref{piclink} for three-dimensional examples of both cases.
The preceding proposition extends to $I$-faces as follows. We shall refer to the links of $I$-faces in the link $L_A$ as $I$-pieces of the link $L_A$.
\begin{corollary} The union of the relative interiors for all exactly $I$-pieces of the link $L_A$ is contractible, and the union of the relative interiors for all at most $I$-pieces of the link $L_A$ is closed.
\end{corollary}

We get the first result about the combinatorial configuration locally at the V-vertex $A$. 

\begin{lemma}\label{lnoij} If $A$ contributes to the monodromy eigenvalue $t_0$, and $t_0$ is not a nearby eigenvalue outside the origin, then there are two possibilities:

-- either $A$ is contained in a unique facet outside the coordinate planes, and this facet is an $\{i,j\}$-facet, 

-- or $A$ is contained in no $I$-faces for $|I|>1$, and in at most one $i$-facet for every $i \in \{2,3,4\}$.

\end{lemma}
\begin{proof} We discuss only faces containing $A$ and only pieces of the link $L_A$.

(0) If there is an $\{i,j,k\}$-facet, then $t_0$ is a nearby monodromy eigenvalue at every point of the $\{i,j,k\}$-coordinate hyperplane.

(1) If there is more than one two-dimensional at least $i$-piece in the link, then, by Remark \ref{remlink}, the (family of the) vertex $\pi_{x_i}(A)$ of the polyhedron $N_i$ satisfies the assumption of Theorem \ref{sandbox3} (for any triangulation of compact faces of $N_i$).

(2) Assume there is an $\{i,j\}$-piece in the link such that at least one of its edges is not contained in the boundary of $T_A$ and is disjoint from its vertices, then the two-dimensional face $F\subset\Gamma_+(f)$, corresponding to this edge, satisfies the assumptions of Theorem \ref{nearby} by Remark \ref{remnearby}.

(3) Assume there is an $\{i,j\}$-piece in the link such that none of its edges satisfies the condition requested in (2). If it is the unique two-dimensional piece in the link $L_A$, then we arrive at the situation (0), otherwise we arrive at the situation (1) (possibly with $j$ instead of $i$).
\end{proof}

\subsection{Dimension 4 : Triangulating.}

In this subsection we continue exploiting Proposition \ref{NDPP} to get further information on the combinatorial structure of the link $L_A$.
We shall thus investigate in particular the cases when $t_0$ is a monodromy eigenvalue of the singularity of $f$ at some point of the $i$-th coordinate axis (although not at the origin). The Newton polyhedron of such singularity equals $\RR_+\times N_i$ (see Remark \ref{remlink} for notation), so we could apply the 3-dimensional Theorem \ref{sandbox3} to its analysis. However, for this theorem, we need triangulations of the bounded faces of the polyhedra $N_i$.

The preceding Lemma \ref{lnoij} will play a crucial role.
Indeed, under its assumption, every triangulation $\mathcal{T}$ of the compact faces of the Newton polyhedron $\Gamma_+(f)$ `naturally'  (see Remark \ref{remtrianglink}) induces a triangulation of the link $L_A$ of a $V$-vertex $A$, contributing to the eigenvalue $t_0$:

1. Assuming $A=(*,0,0,0)$, take the isomorphic images of simplices of $\mathcal{T}$ intersected with the hyperplane $x_2+x_3+x_4=\varepsilon$ under the projection $\pi_{x_1}$ for small $\varepsilon$.

2. Subdivide every $i$-piece (Definition \ref{defiface}) of the link $L_A$ by several segments from the vertex $v_i$, so that the resulting pieces together with the ones from (1) form a triangulation of $T_A$.

This triangulation will be denoted by $\tilde L_A$. (If the link $L_A$ contains a unique two-dimensional piece, corresponding to an $\{i,j\}$-facet, we triangulate it trivially.)

\begin{remark}\label{remtrianglink} 1. The triangulation $\tilde L_A$ is natural in the sense that, in the notation of Remark \ref{remlink}, it agrees with the corresponding triangulations of the projection polyhedra $N_i$. More specifically, every compact face of $N_i$ is the projection of one compact face of $\Gamma_+(f)$, so the triangulation $\mathcal{T}$ of $\Gamma_+(f)$ induces a triangulation $\mathcal{T}_i$ of the compact faces of $N_i$. Assume that the link $L_A$ contains an $i$-piece, then the $\mathcal{T}_i$-triangulated link of $A_i$ in $N_i$ is affinely isomorphic to the triangulated link of $v_i$ in $\tilde L_A$. This is an important refinement of Remark \ref{remlink}, as we shall see later in Lemma \ref{lhexagon}.

2. No triangulation of the link $L_A$ may be natural in the above sense in the presence of $\{i,j\}$-pieces with edges 
outside the boundary of $T_A$. So we really have to work under the assumptions of Lemma \ref{lnoij} in this subsection. In particular, we have no natural notion of a link triangulation $\tilde L_A$ associated to $\mathcal{T}$ for a $V$-vertex $A$ that does not contribute to the eigenvalue $t_0$. 
\end{remark}
In what follows, we refer to the simplices in the triangulation $\mathcal{T}$ as $V$-simplices or just $V$-faces, because we shall not be interested in faces of $\Gamma_+(f)$ in the usual sense anymore.

We will now continue to study how the combinatorics of $\Gamma_+(f)$ assures the existence of a nearby monodromy eigenvalue $t_0$ outside the origin, but, this time, taking into account the chosen triangulation $\mathcal{T}$.

We will choose once and for all a triangulation $\mathcal{T}$ and corresponding link triangulations $\tilde L_A$ according to the following lemma.

\begin{lemma} \label{lgoodtriang} If $s_0$ is a pole of the topological zeta function, then there exists a triangulation $\mathcal{T}$ (with no new vertices) of the Newton polyhedron $\Gamma_+(f)$, such that either it has a non-$B$-simplex (see Definition \ref{B-F}) contributing to the eigenvalue $t_0=\exp{2\pi i  s_0}$, or a $B$-border, whose $V$-edge contributes to the eigenvalue $t_0$. 

\end{lemma}

\begin{proof} Since $\Gamma_+(f)$ does not satisfy the assumptions of Theorem \ref{mainB} for $s_0$, either it contains a $B$-border contributing to $t_0$, or a non-$B$-facet, contributing $s_0$. In the first case, the $V$-edge of the border contributes to the sought eigenvalue (see the proof of Lemma \ref{otherconfigB1}). 
In the second case, triangulate the contributing non-$B$-facet by Lemma \ref{ALEX} (so that one of the resulting non-$B$-simplices contributes to the sought eigenvalue) and extend this triangulation arbitrarily to the whole $\Gamma_+(f)$.
\end{proof}

The absence of the nearby monodromy eigenvalue $t_0$ outside the origin imposes lots of restrictions on combinatorics of the Newton polyhedron $\Gamma_+(f)$ and the triangulation $\mathcal{T}$. Let us use these restrictions to make some crucial conclusions about the combinatorics of the triangulated links of the $V$-vertices.

\begin{definition} The vertex $v_i$ of the triangle $T_A$ opposite to the edge $x_i=0$ is said to be the \emph{$i$-th corner of the link $L_A$}, if the link triangulation $\tilde L_A$ contains a piece of the form $v_iBC$, where $B$ and $C$ are points on the two edges of $T_A$ containing $v_i$. Its star is the set of four pieces $v_iBC, v_iB, v_iC, v_i$.
\end{definition}

\begin{lemma} \label{lhexagon} If $A$ contributes to the eigenvalue $t_0$ and $t_0$ is not a nearby eigenvalue outside the origin, then for every $i \in \{2,3,4\}$, either the edge of $\Gamma_+(f)$ corresponding to the vertex $v_i\in T_A$ is bounded and is not a corner of a $V$-facet, or the link $L_A$ has the $i$-th corner (which may correspond to an $i$-edge or to a bounded $V$-face of $\Gamma_+(f)$). In particular, there are six alternatives for the vertex $A$ in this situation:

-- The vertex $A$ is a corner of a $V$-facet;

-- The vertex $A$ is a corner of a $V$-edge and is contained in a unique facet, which is an $\{i,j\}$-facet.

-- The vertex $A$ is contained in a $V$-edge that is not a corner of a $V$-facet, and the link $L_A$ has two corners.

-- The vertex $A$ is contained in two $V$-edges that are not corners of $V$-facets, and the link $L_A$ has one corner.

-- The vertex $L_A$ is contained in three $V$-edges that are not corners of $V$-facets.

-- The link $L_A$ has three corners.

\end{lemma}
\begin{proof} 1. If we exclude the first two cases from our consideration, then, by Lemma \ref{lnoij}, there is at most one two-dimensional $i$-piece in the link $L_A$ for $i\in\{2,3,4\}$. However, it still could contain more than one two-dimensional $i$-piece of the triangulated link $\tilde L_A$. Once we exclude this possibility, we prove the lemma.

2. So assume towards the contradiction that the unique $i$-piece of the link $L_A$ (corresponding to an $i$-facet $\tau\ni A$) contains 

a) no $(i,j)$-pieces and 

b) at least two $i$-pieces of the triangulated link $\tilde L_A$ (corresponding to compact triangles $\tau_k\ni A$ in the boundary of $\tau$). 

From this, we shall conclude that Theorem \ref{sandbox3} is applicable to the family $F=\{A_i\}$ in the polyhedron $N_i$ with the triangulation $\mathcal{T}_i$ (in the notation of Remarks \ref{remlink} and \ref{remtrianglink}), because this family cannot fall within the ``unless'' case that we exclude in the statement of Theorem \ref{sandbox3}. Once we prove this, Theorem \ref{sandbox3} assures that $t_0$ is a nearby eigenvalue on the $i$-th coordinate axis outside the origin, which contradicts our assumption.

3. In remains to deduce from (a) and (b) above that the family $F$ does not fall within the case that we exclude in the statement of Theorem \ref{sandbox3}. Indeed, (a) implies that all faces of $N_i$, which contain the vertex $A_i$ and are not contained in a coordinate plane, are compact. And (b) assures that such faces contain at least two triangles with the vertex $A_i$ in the triangulation $\mathcal{T}_i$: these are the projections of $\tau_k$ along the $i$-th coordinate axis.

Thus all pieces of the triangulation $\mathcal{T}_i$ containing $A_i$ cannot belong to the same family (because every family contains at most one triangle).
%Then the projections of these $\tau_j$ along the $i$-th axis are triangles $\tilde\tau_j$ of the triangulation $\mathcal{T}_i$ of the polyhedron $N_i$. 
%Since each of the triangles $\tilde\tau_j$ is not contained in a coordinate plane and contains the vertex $A_i$, we can apply Theorem \ref{sandbox3} to the (family of the) vertex $A_i$ in the polyhedron $N_i$, because this family cannot fall within the case that we exclude in the statement of Theorem \ref{sandbox3}. 
%To exclude this possibility, note that, in the latter case, by Remark \ref{remtrianglink}(1) and using its notation, the (family of the) vertex $A_i$ of the polyhedron $N_i$ with the triangulation $\mathcal{T}_i$ satisfies the assumption of Theorem \ref{sandbox3}.
\end{proof}

We will need the following combinatorial observation, applicable to the conclusion of Lemma \ref{lhexagon}. Let $\tilde L$ be a triangulation of a triangle $T$.
\begin{definition}\label{definscr} A triangle of $\tilde L$ is said to be \emph{interior}, if none of its edges is contained in the edges of $T$, and no one of its vertices is contained in the vertices of $T$. An interior triangle of $\tilde L$ is said to be \emph{inscribed}, if its three vertices are in the interior of the three edges of $T$. A vertex of $T$ is called a \emph{corner of the triangulation $\tilde L$}, if it is a vertex of only one of the triangles in the triangulation.
\end{definition}
\begin{lemma}\label{ltriangles} \begin{enumerate}
\item If the triangulation $\tilde L$ of a triangle $T$ has three corners, then either it has an inscribed triangle, or it has at least three interior triangles.
\item If the triangulation $\tilde L$ has a corner, and no edge of the triangulation connects a vertex of $T$ with its opposite edge, then $T$ has an interior triangle.
\end{enumerate}
\end{lemma}
Informally speaking, this means that, in the setting of Lemma \ref{lhexagon}, the link looks similarly to one of the six examples on Figure \ref{pichex}. The pieces of the link that may correspond to $i$-facets for some $i$ are shown in grey; the white area may be subdivided into pieces in a more complicated way than the one shown on the picture.
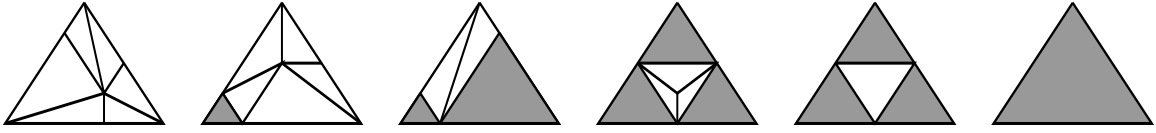
\begin{figure}
    \centering
\skipfig{\psscalebox{0.65 1.0} 
{
\begin{pspicture}(0,-0.82121325)(23.296568,0.82121325)
\definecolor{colour0}{rgb}{0.6,0.6,0.6}
\psline[linecolor=black, linewidth=0.04](5.6482844,0.8070711)(4.0482845,-0.79292893)(7.2482843,-0.79292893)(5.6482844,0.8070711)
\psline[linecolor=black, linewidth=0.04](4.448284,-0.39292893)(4.8482842,-0.79292893)(5.6482844,0.0070710755)(4.448284,-0.39292893)
\psline[linecolor=black, linewidth=0.04](5.6482844,0.8070711)(5.6482844,0.0070710755)(6.448284,0.0070710755)
\psline[linecolor=black, linewidth=0.04](5.6482844,0.0070710755)(7.2482843,-0.79292893)
\psline[linecolor=black, linewidth=0.04, fillstyle=solid,fillcolor=colour0](4.448284,-0.39292893)(4.0482845,-0.79292893)(4.8482842,-0.79292893)(4.448284,-0.39292893)
\psline[linecolor=black, linewidth=0.04](9.648284,0.8070711)(8.048285,-0.79292893)(11.248284,-0.79292893)(9.648284,0.8070711)
\psline[linecolor=black, linewidth=0.04, fillstyle=solid,fillcolor=colour0](8.448284,-0.39292893)(8.048285,-0.79292893)(8.848285,-0.79292893)(8.448284,-0.39292893)
\psline[linecolor=black, linewidth=0.04, fillstyle=solid,fillcolor=colour0](10.048285,0.40707108)(8.848285,-0.79292893)(11.248284,-0.79292893)(10.048285,0.40707108)
\psline[linecolor=black, linewidth=0.04](9.648284,0.8070711)(8.848285,-0.79292893)
\psline[linecolor=black, linewidth=0.04, fillstyle=solid,fillcolor=colour0](13.648284,0.8070711)(12.048285,-0.79292893)(15.248284,-0.79292893)(13.648284,0.8070711)
\psline[linecolor=black, linewidth=0.04, fillstyle=solid](12.848285,0.0070710755)(13.648284,-0.79292893)(14.448284,0.0070710755)(12.848285,0.0070710755)
\psline[linecolor=black, linewidth=0.04, fillstyle=solid,fillcolor=colour0](17.648285,0.8070711)(16.048285,-0.79292893)(19.248283,-0.79292893)(17.648285,0.8070711)
\psline[linecolor=black, linewidth=0.04, fillstyle=solid](16.848284,0.0070710755)(17.648285,-0.79292893)(18.448284,0.0070710755)(16.848284,0.0070710755)
\psline[linecolor=black, linewidth=0.04](12.848285,0.0070710755)(13.648284,-0.39292893)(14.448284,0.0070710755)
\psline[linecolor=black, linewidth=0.04](13.648284,-0.39292893)(13.648284,-0.79292893)
\psline[linecolor=black, linewidth=0.04](1.6482843,0.8070711)(0.048284303,-0.79292893)(3.2482843,-0.79292893)(1.6482843,0.8070711)
\psline[linecolor=black, linewidth=0.04](1.6482843,0.8070711)(2.0482843,-0.39292893)(2.4482844,0.0070710755)
\psline[linecolor=black, linewidth=0.04](2.0482843,-0.39292893)(3.2482843,-0.79292893)
\psline[linecolor=black, linewidth=0.04](0.048284303,-0.79292893)(2.0482843,-0.39292893)(2.0482843,-0.79292893)
\psline[linecolor=black, linewidth=0.04](2.0482843,-0.39292893)(1.2482843,0.40707108)
\psline[linecolor=black, linewidth=0.04, fillstyle=solid,fillcolor=colour0](21.648285,0.8070711)(20.048285,-0.79292893)(23.248283,-0.79292893)(21.648285,0.8070711)
\end{pspicture}
}}
    \caption{examples of links of vertices in the setting of Lemma \ref{lhexagon}}
    \label{pichex}
    
\end{figure}

\begin{proof} The proof of Part 1 proceeds by induction on the number of triangles in the triangulation. If any corner triangle can be glued with its (unique) adjacent triangle into a larger triangle, then glue and apply the induction hypothesis. Otherwise each of the three corner triangles are adjacent to some interior triangles $T_i$. If $T_i=T_j$, then its vertices are in the interior of the three edges of $T$, and otherwise $T_i$ are three different interior triangles.

Part 2 is proved in the same way.
\end{proof}

From Proposition \ref{NDPP} we can also deduce the following result.

\begin{lemma} \label{lchasm} If a $V$-triangle $F\subset\Gamma_+(f)$ is not contained in a $V$-facet, and its family (see Definition \ref{deffamily}) contributes to the eigenvalue $t_0$, then $t_0$ is a nearby monodromy eigenvalue outside the origin.
\end{lemma}
\begin{proof} Under these assumptions, if $F$ is in the $i$-th coordinate hyperplane, then $\pi_{x_i}F$ is a facet of $N_i=\pi_{x_i}\Gamma_+(f)$, whose family contributes to $t_0$, so $t_0$ is a nearby monodromy eigenvalue at a point of the $i$-th coordinate hyperplane by Theorem \ref{sandbox3}.
\end{proof}

\subsection{Monodromy conjecture for non-degenerate singularities of four variables}\label{Shappyend}

We now prove the monodromy conjecture for non-degenerate singularities of four variables similarly to Theorem \ref{sandbox3}. Recall our setting.

Let $f\in{\mathbb C}[x_1,x_2,x_3,x_4]$ be non-degenerate at the origin. Let $s_0$ be a pole of the topological zeta function of $f$ and set $t_0=\exp(2\pi i  s_0)$.
Analogously to the proof of Theorem \ref{sandbox3}, we choose once and for all a triangulation $\mathcal{T}$ of the Newton polyhedron $\Gamma_+(f)$ in accordance with Lemma \ref{lgoodtriang}, and we will define the register of $\Gamma_+(f)$ as a disjoint union of certain groups of families. It will be practical to work with what we call extended families.

\begin{definition} The \emph{extended family of a $V$-simplex $F$} is the set of all $V$-faces from the family of $F$ and, in the case of $\dim\, F=3$, also the $V$-vertex of $F$, if it is a maximal by inclusion $V$-face of $F$. In this case, $F$ has no other $V$-faces except for, maybe, the facet of $F$ opposite to the $V$-vertex (notice that if $A$ and $PQ$ are V-faces in a facet $APQR$, then also $APQ$ is a V-face).
\end{definition}
\begin{remark} A $V$-vertex $A$ is in the extended family of a $V$-tetrahedron $\tau$ if and only if the image of $\tau$ in the triangulated link $\tilde L_A$ is an interior triangle or coincides with $T_A$.
\end{remark}
\begin{definition} The \emph{register $\R$ of $\Gamma_+(f)$} (depending on the chosen triangulation) is the set of the following extended families:

-- extended families of all 3-dimensional $V$-simplices;

-- extended families of all 2-dimensional $V$-simplices that do not enter the aforementioned extended families;

-- extended families of all 1-dimensional $V$-simplices that do not enter the aforementioned extended families.
\end{definition}

\begin{remark}\label{rem1to1} Notice that every positive-dimensional $V$-simplex enters exactly one extended family of the register. This is obvious for simplices of dimensions 2 and 3, and a $V$-edge $E$ may enter families of two different $V$-triangles, but in this case, by the subsequent Lemma \ref{lvvv}, both of these triangles are themselves in families of $V$-tetrahedra $\tau_1$ and $\tau_2$, so their own families are not in the register. Thus the extended family of $E$ is itself in the register for $\tau_1\ne\tau_2$ (because $E$ is not a corner of any tetrahedron in this case), and $E$ is in the extended family of $\tau_1=\tau_2$ otherwise.
\end{remark}

\begin{lemma}\label{lvvv} If a bounded $(n-3)$-dimensional $V$-face $\nu$ of a Newton polyhedron in $\RR^n$ is contained in two bounded $(n-2)$-dimensional $V$-faces $\nu_1$ and $\nu_2$, then each of these faces is contained in a bounded $V$-facet.
\end{lemma}
\begin{proof} This is obvious for $n=3$, and the general case reduced to $n=3$ by taking the projection of the Newton polyhedron along the affine span of $\nu$.
\end{proof}

We will now prove the monodromy conjecture for non-degenerate singularities of four variables, modulo several lemmas in the next subsection regarding certain exotic families.

\begin{theorem}
Let $f\in{\mathbb C}[x_1,x_2,x_3,x_4]$ be non-degenerate at the origin. Let $s_0 \neq -1$ be a pole of the topological zeta function of $f$ and set $t_0=\exp(2\pi i  s_0)$. If $t_0$ is not a (tropical) nearby monodromy eigenvalue outside the origin, then $t_0$ is a root of the monodromy zeta function of $f$ at the origin, and hence a monodromy eigenvalue. 
\end{theorem}

\begin{proof}
Let $\mathcal{T}$ be a triangulation of the Newton polyhedron $\Gamma_+(f)$, according to Lemma \ref{lgoodtriang}, and 
induce from $\mathcal{T}$ the corresponding link triangulations $\tilde L_v$ for every $V$-vertex $v$, contributing to the eigenvalue $t_0$ (see the beginning of the preceding subsection).

For every $V$-vertex $v$, denote by $r(v)$ the number of extended families containing $v$ in the register $\R$.

By Remark \ref{rem1to1}, we represent the multiplicity of the candidate root $t_0=\exp(2 \pi i s_0)$ of the monodromy zeta function as 
$$\sum_{\mathcal F} (\mbox{contribution of }{\mathcal F}\mbox{ to the multiplicity of }t_0)+\sum_v \bigl(r(v)-1\bigr), \eqno{(*)}$$
where ${\mathcal F}$ runs over the register, and $v$ runs over $V$-vertices, contributing to $t_0$.

We will prove that every term in every sum of $(*)$ is non-negative, and, moreover, at least one term in $(*)$ is strictly positive.

In our setting we have :

{\bf 1.} The contribution of every extended family ${\mathcal F}\in\R$ to the multiplicity of $t_0$ is non-negative. For $3$-dimensional extended families, this follows from Theorem \ref{Key-5} and Lemma \ref{lemmatau5} in the next subsection. 
Note that there are no $2$-dimensional families contributing to $t_0$ by Lemma \ref{lchasm}. As $s_0 \neq -1$, the contribution of every $1$-dimensional family to $t_0$ is also non-negative (see for example the proof of \cite[Prop. 5]{L-V}).

{\bf 2.} For every $V$-vertex $A$, contributing to the eigenvalue $t_0$, we have $r(A)\geqslant 1$. This is because, by Lemmas \ref{lhexagon} and \ref{ltriangles}(1), the vertex $A$ is contained in one of the following:

-- a $V$-tetrahedron, for which $A$ is a corner;

-- a $V$-edge, whose extended family is in the register;

-- a $V$-tetrahedron, whose image in the triangulated link $\tilde L_A$ is interior, and whose extended family thus contains $A$.

{\bf 3.} The contribution of at least one odd-dimensional extended family ${\mathcal F}\in\R$ is positive, or $r(v)>1$ for some $V$-vertex $v$. To see this, consider the following possible cases:

-- If Lemma \ref{lgoodtriang} provides a $B$-border $AA_1A_2$ contained in two $V$-tetrahedra, whose candidate pole of the topological zeta function equals $s_0$, then we have %three
two subcases for its $V$-edge $A_1A_2$:

%-- -- The contribution of the extended family of the edge $A_1A_2$ is positive.

-- -- Assume for every $A_i$ the following: if it is a $V$-vertex, contributing to $t_0$, then the edge $AA_i$ is in a coordinate plane. Under this assumption, the contribution of the family of the $V$-edge $A_1A_2$ of this border is positive by Lemma \ref{otherconfigB1} in the next subsection.

-- -- Assume that some $A_i$ (say, $A_1$) breaks the preceding assumption, then, in the triangulated link $\tilde L_{A_1}$, exactly one interior segment contains the vertex $v$ corresponding to the $V$-edge $A_1A_2$: this segment corresponds to the border triangle $AA_1A_2$ and does not split the link by our assumption. Now we have again two subcases:

-- -- -- The triangulated link $\tilde L_{A_1}$ has exactly two corners, then Lemma \ref{ltriangles}(2) applies to $\tilde L_{A_1}$ and provides an interior triangle. The $V$-vertex $A_1$ is then contained in the extended families of the $V$-edge $A_1A_2$ and the $V$-tetrahedron, corresponding to the interior triangle, so $r(A_1)>1$  (both of these families are on the register, since $A_1A_2$ is not a corner).

-- -- -- The triangulated link $\tilde L_{A_1}$ has at most one corner, then $A_1$ is contained by Lemma \ref{lhexagon} in the extended families of two $V$-edges, which are not corners, hence on the register, so $r(A_1)>1$. 

-- If Lemma \ref{lgoodtriang} provides a contributing $V$-tetrahedron, whose extended family coincides with the usual family, then the contribution of this family is positive by Theorem \ref{Key-5}.

-- If Lemma \ref{lgoodtriang} provides a $V$-tetrahedron $F$, whose candidate pole of the topological zeta function is $s_0$, and whose extended family consists of the family of $F$ and one additional contributing $V$-vertex $A$, then we have two subcases for the triangulated link $\tilde L_A$ by Lemma \ref{lhexagon}:

-- -- One of the vertices of $\tilde L_A$ corresponds to a $V$-edge, whose extended family contains $A$ and is contained in the register. In this case $r(A)\geqslant 2$, because $A$ is also in the extended family of $F$.

-- -- There are three corners in $\tilde L_A$. This case subdivides into the following subcases by Lemma \ref{ltriangles}(1):

-- -- -- There are three interior triangles in $\tilde L_A$. Then $r(A)>2>1$.

-- -- -- There is an inscribed triangle in $\tilde L_A$, and it is not the image of $F$. Then $r(A)>1$.

-- -- -- There is an inscribed triangle in $\tilde L_A$, and it is the image of $F$. Then, Lemma \ref{lemmalinalg1} or \ref{FITH} from the next subsection applies to $F$, so its extended family has a positive contribution to the multiplicity of $t_0$.

We conclude that the number $t_0$ is a root of the monodromy zeta function and, in particular, a monodromy eigenvalue of the singularity $f$ at the origin. 
\end{proof}

\subsection{Exotic families}

In the course of the proof of the monodromy conjecture for $n=4$, we encountered certain exotic families of $V$-faces, whose contributions to the multipliciy of the corresponding monodromy eigenvalue ought to be non-zero.
Their contributions are estimated in this subsection.

For the most part (Lemmas \ref{lemmalinalg1}--\ref{FITH}), we will study the extended family of a facet $\tau=ABCD$ if it does not coincide with the family of $\tau$ (i. e. consists of a $V$-vertex $A\prec\tau$ and possibly its opposite triangular face whenever it is a $V$-face), and often moreover assume that $\tau$ defines an inscribed triangle (Definition \ref{definscr}) in the link of $A$.

%We now start to study the contribution of the extended family of a $V$-tetrahedron $\tau =ABCD$ if it does not coincide with the family of $\tau$ (i.e. if the only proper $V$-faces of $\tau$ are a vertex $A$ and possibly its opposite triangle).

\begin{lemma} \label{lemmalinalg1}
Let $\tau =ABCD$ be a 
$3$-dimensional lattice 
simplex in a compact facet of 
$\Gamma_+(f)$ 
such that 
$v=A(0,0,0,\alpha)$ is its only proper V-face contributing to $t_0=e^{-2\pi i \nu(\tau)/N(\tau)}$.
If $B(\beta_1,\beta_2,\beta_3,\beta_4), C(\gamma_1,\gamma_2,\gamma_3,\gamma_4)$ and $D(\delta_1,\delta_2,\delta_3,\delta_4)$ with $\beta_2=\gamma_3=\delta_1=0$,
then $\Vol_{\ZZ}(\tau) \neq 1$.  
\end{lemma}

\begin{proof}
Assuming to the contrary that $\Vol_{\ZZ}(\tau) = 1$, the vector product of $AB, AC, AD$ is a primitive vector $(a,b,c,d)$, normal to $\tau$:
$$ \mbox{aff}(\tau)  \leftrightarrow ax_1+bx_2+cx_3+dx_4=N, \mbox{ with}$$
\begin{eqnarray*}
a & = & \gamma_2\delta_3(\alpha-\beta_4) + \delta_2\beta_3(\alpha-\gamma_4) - \beta_3\gamma_2(\alpha-\delta_4) \\
b & = & \delta_3\beta_1(\alpha-\gamma_4) + \beta_3\gamma_1(\alpha-\delta_4) - \gamma_1\delta_3(\alpha-\beta_4) \\
c & = & \beta_1\gamma_2(\alpha-\delta_4) + \gamma_1\delta_2(\alpha-\beta_4) - \delta_2\beta_1(\alpha-\gamma_4) \\
d & = &  \beta_1\gamma_2\delta_3 + \gamma_1\delta_2 \beta_3 \\
N & = & \alpha d.
\end{eqnarray*}
Since $A$ contributes to $t_0$, we have $d \mid a+b+c+d$. Let $k \in \ZZ$ be such that $a+b+c+d=(-k+1)d$ and let $M$ be the matrix
\begin{equation*}
\left( \begin{array}{cccc}
\beta_1 &  0 &\beta_3 & \beta_4- \alpha\\
\gamma_1 &\gamma_2 & 0 & \gamma_4- \alpha\\
0 & \delta_2  &\delta_3 & \delta_4-\alpha\\
1 & 1 & 1 & k
 \end{array} \right). 
\end{equation*}
Then $M ^t (a,b,c,d)=~^t0$. By Lemma \ref{b} it follows that $d$ divides the minors $M^i_4$, $1 \leq i \leq 4$, which are
 \begin{eqnarray*}
M^1_4 & = & \gamma_1\delta_2 + \gamma_2\delta_3 - \gamma_1\delta_3 \\
M^2_4 & = & \delta_3\beta_1+\delta_2\beta_3 - \delta_2\beta_1\\
M^3_4 & = & \beta_1\gamma_2+ \beta_3\gamma_1 - \beta_3\gamma_2\\
M^4_4 & = & d.
\end{eqnarray*}
As $A$ is a maximal (by inclusion) proper V-face of $\tau$, none of the numbers $\beta_1,\beta_3,\gamma_1,\gamma_2, \delta_2,\delta_3$ equals $0$, and then obviously $M^1_4, M^2_4,$ and $M^3_4$ are strictly less than $d$. 
As 
$$\frac{M^1_4}{\gamma_1\delta_3}+ \frac{M^2_4}{\delta_2\beta_1}+\frac{M^3_4}{\beta_3\gamma_2} = \left(\frac{\delta_2}{\delta_3}+\frac{\delta_3}{\delta_2} \right)+
\left(\frac{\gamma_2}{\gamma_1}+\frac{\gamma_1}{\gamma_2}\right)+\left(\frac{\beta_3}{\beta_1}+\frac{\beta_1}{\beta_3}\right) - 3 >0, $$ 
it follows that at least one of the minors $M^1_4, M^2_4,$ or $M^3_4$ is strictly positive, which contradicts the fact that $d$ divides the minors $M^i_4$, $1 \leq i \leq 4$.
\end{proof}

\begin{lemma}\label{lemmalinalg2}
Let $\tau=APQR$ with $A(2,0,0,0), P(0,0,p_2,p_3), 
Q(0,q_1,0,q_3), R(0,r_1,r_2,0)$ 
be a $3$-dimensional lattice 
simplex in a compact facet of 
$\Gamma_+(f)$ contributing to $t_0=e^{-2 \pi i \nu(\tau)/N(\tau)} \neq 1$. 
Assume that $v=A$ 
and $\sigma =PQR$ are the only proper
contributing $V$-faces to $t_0$ in $\tau$.
Then $(\Vol_{\ZZ}( \sigma ),\Vol_{\ZZ}( \tau )) \neq (1,2)$. 
\end{lemma}

\begin{proof}
Suppose that $\Vol_{\ZZ}( \tau )=2$ and 
$\Vol_{\ZZ}( \sigma )=1$. 
Let 
\begin{equation*}
{\rm aff} (\tau) : \  ax_1+bx_2+cx_3+dx_4
=N(\tau)=2a 
\end{equation*}
be the equation of ${\rm aff} (\tau)$ 
with ${\rm gcd}(a,b,c,d)=1$. 
One has $\frac{N(\tau)}{{\rm gcd}(b,c,d)}=N(\sigma)$ and as $\sigma$ contributes to $t_0$, we have 
$$\frac{\nu(\tau)}{N(\tau)}N(\sigma)=\frac{a+b+c+d}{N(\tau)}N(\sigma) \in \mathbb{Z}.$$
This implies that ${\rm gcd}(b,c,d)$ divides $a$. As ${\rm gcd}(a,b,c,d)=1$, we get that ${\rm gcd}(b,c,d)=1$.
Since the $V$-face $v=A$ 
contributes to $t_0$, we have also 
$a|(b+c+d)$. Then by $N( \tau )=2a$, we obtain 
$t_0=1$ (which is excluded) or 
$t_0=-1$.
We study what happens when 
$t_0=-1$. This implies that 
$2a | (b+c+d)$. 
As $\Vol_{\ZZ}(\tau)=2$, by Proposition \ref{PLT} (2) 
the even integers $2a,2b,2c$ and $2d$ are the 
$3 \times 3$ minors of the matrix
 $$ \left( \begin{array}{c}
 PA \\
 QA \\
 RA
 \end{array} \right) = \left( \begin{array}{cccc}
 -2 & 0 &p_2 &p_3 \\
 -2 &q_1 &0 &q_3 \\
 -2 &r_1 &r_2 &0
 \end{array} \right),$$
 and hence the expressions for $a,b,c$ and $d$ become
 \begin{eqnarray*}
 a & =& \frac{q_1r_2p_3+r_1p_2q_3}{2}, 
\qquad b  = r_2p_3+p_2q_3-q_3r_2, 
 \end{eqnarray*}
 \vspace*{-1cm}
 \begin{eqnarray*}
 c &= & p_3q_1+q_3r_1-r_1p_3, \qquad d  = q_1r_2+r_1p_2-p_2q_1.
 \end{eqnarray*}
%As in the proof of Lemma 
%\ref{lemmalinalg1}, 
For the 
integer $k=(b+c+d)/a$ the vector $(-k,1,1,1)$ is a rational 
linear combination of 
$\overrightarrow{AP}, \overrightarrow{AQ}$ 
and $\overrightarrow{AR}$, because $(-k,1,1,1)\cdot(a,b,c,d)=0$, and $\overrightarrow{AP}, \overrightarrow{AQ}$ 
and $\overrightarrow{AR}$ generate the orthogonal complement to $(a,b,c,d)$.

Let $H \simeq \RR^3$ be the affine hyperplane in $\RR^4$ 
containing $\tau$ and consider the lattice 
$L:=H \cap \ZZ^4 \simeq \ZZ^3$ in it. For the 
integer $k=(b+c+d)/a$, we have $(-k,1,1,1) \in L$. Since 
the normalized volume of $\tau$ is $2$, the 
sublattice $K$ of $L$ generated by the three 
vectors $\overrightarrow{AP}, \overrightarrow{AQ}$ 
and $\overrightarrow{AR}$ is of index $2$ in $L$, i.e.
$[L:K]=2$.

%As $\Vol_{\ZZ}(\tau)= 2$, the cone generated by $\overrightarrow{AP}, \overrightarrow{AQ}$ and $\overrightarrow{AR}$ has 
%multiplicity two and so in particular
%there should be an integer solution 
%$(x,y,z) \in \mathbb{Z}^3$ for 
%the equation 
This means there exist integers $x,y,z$ such that
 \begin{equation}\label{eqM}
 \left( \begin{array}{ccc}
 -2& -2& -2 \\
 0 &q_1 &r_1 \\
p_2 &0 &r_2 \\
p_3 &q_3 &0
 \end{array} \right)\left( \begin{array}{c}
 x \\y \\z
 \end{array} \right)=\left( \begin{array}{c}
-2k \\
2 \\
2 \\
2
\end{array} \right).
\end{equation}
We define a matrix $M$ by 
\begin{equation*}
M= \left( \begin{array}{ccc}
 0 &q_1 &r_1 \\
p_2 &0 &r_2 \\
p_3 &q_3 &0
 \end{array} \right). 
\end{equation*}
Then by Cramer's rule we find that
\begin{equation*}
x= \frac{ 2(r_2 q_1 + r_1 q_3 - r_2 q_3)}{ {\rm det} 
(M)} \quad \mbox{,} 
\quad y=  \frac{ 2(r_2 p_3 + r_1 p_2 - 
r_1 p_3)}{  {\rm det} (M)},
\end{equation*}
\begin{equation*}
\mbox{and} \quad z= \frac{ 2(q_1p_3+p_2q_3-
p_2q_1)}{  {\rm det} (M)}.
\end{equation*}
%Since the lattice volume of the tetrahedron $APQR$ equals 2, we have $|\det M|=1$ or 2, i.e. $(x,y,z)\in\ZZ^3$.

Now we study the possible signs of $x,y$ and $z$. 
If $p_2 \geq p_3$ and $q_1 \geq q_3$, then $y > 0$ and $x > 0$.
If $p_2 \geq p_3$ and $q_1 \leq q_3$, then $y>0$ and $z > 0$.
If $p_2 \leq p_3$  and $r_1 \leq r_2$, then $z > 0$ and $y > 0$ 
and so on. Thus we find that at least two of the integers 
$x,y$ and $z$ are always positive. By permuting them, 
we may assume that $x > 0$ and $y > 0$. 
As none of $p_2,p_3,q_1,q_3,r_1,r_2$ is equal to $0$, 
the equation $p_3x+q_3y=2$ obtained by \eqref{eqM} 
implies that $p_3=q_3=1$ and $x=y=1$. 
Consequently we get 
\begin{equation*}
a = \frac{q_1r_2+r_1p_2}{2}, 
\ b  = p_2, \ c =  q_1, \ d  
= q_1r_2+r_1p_2-p_2q_1
\end{equation*}
and ${\rm det} (M)= q_1r_2 +r_1p_2$. 
As we supposed that $2a |(b+c+d)$, we have 
\begin{equation*}
(q_1r_2 +r_1p_2) | (p_2+q_1+q_1r_2+r_1p_2-p_2q_1) 
 \ \Longleftrightarrow \ 
{\rm det} (M) | (p_2+q_1-p_2q_1)
\end{equation*}
and $z$ is an even integer. 
Hence, again by \eqref{eqM} and by using that $x=y=1$, 
we find that $p_2$ and $q_1$ should be even. 
Then we have 
\begin{equation*}
{\rm gcd} (b,c,d) = {\rm gcd} 
(p_2,q_1,q_1r_2+r_1p_2-p_2q_1) \geq 2. 
\end{equation*}
However, it contradicts  
${\rm gcd} (b,c,d) =1$. 
This completes the proof. 
\end{proof}

Recall that, for a $V$-face $\tau$, we define
\begin{equation*}
\zeta_{\tau} (t) = 
\Bigl( 1-t^{N( \tau )} \Bigr)^{\Vol_{\ZZ}( \tau )} 
\in \CC [t]. 
\end{equation*}

\begin{lemma} \label{lemmatau5}
1. Let $\tau =APQR$ be a 
$3$-dimensional lattice 
simplex in a compact facet of 
$\Gamma_+(f)$ 
such that 
$v=A(\alpha,0,0,0)$ and $\sigma=PQR$ are 
V-faces. 
Then 
\begin{equation*}
\frac{\zeta_{\tau}(t) \cdot (1-t)}{\zeta_v(t) \cdot 
\zeta_{\sigma}(t)} 
\end{equation*}
is a polynomial.

2. If $\Vol_{\ZZ}(\tau) = \Vol_{\ZZ}(\sigma)$, then 1 is the only common root of the polynomials in the denominator. 
\end{lemma}

\begin{proof}
If $\Vol_{\ZZ}(\tau) > \Vol_{\ZZ}(\sigma)$ 
then the assertion is obvious. 
So suppose that $\Vol_{\ZZ}(\tau)= \Vol_{\ZZ}(\sigma)$. 
Let 
\begin{equation*}
{\rm aff} (\tau) : \  
a x_1 + b x_2 + cx_3 + dx_4 = N(\tau)
\end{equation*}
be the equation of ${\rm aff} (\tau)$ 
with ${\rm gcd}(a,b,c,d)=1$. 
Since we have 
\begin{equation*}
N(v)= \alpha = \frac{N( \tau )}{a} , \quad 
N( \sigma )= \frac{N( \tau )}{{\rm gcd}(b,c,d)}
\end{equation*}
and ${\rm gcd}(a,b,c,d)=1$, the equivalent (by Proposition \ref{PLT} and Lemma \ref{LLT}) conditions 
\begin{equation*}
\Vol_{\ZZ}(\tau)= \Vol_{\ZZ}(\sigma) \ 
\Longleftrightarrow \ 
{\rm gcd}( N(v), N( \sigma ))=1 
\end{equation*}
imply that $N( \tau )= a \cdot 
{\rm gcd}(b,c,d)$, $N( \sigma )=a$ 
and $N(v)= \alpha = {\rm gcd}(b,c,d)$. 
Hence we get 
\begin{equation*}
\frac{\zeta_{\tau}(t)}{\zeta_v(t) \cdot 
\zeta_{\sigma}(t)} 
=
\frac{(1-t^{ a \cdot \alpha }
)^{\Vol_{\ZZ}(\sigma)}}{
(1-t^{ \alpha }) \cdot (1-t^{a}
)^{\Vol_{\ZZ}(\sigma)
}}.
\end{equation*}
The only common zero of $\zeta_v(t)$ and 
$\zeta_{\sigma}(t)$ is equal to $1$, because ${\rm gcd}(a,\alpha)=1$. Thus the denominator divides the numerator.
\end{proof}

\begin{lemma}\label{FITH}
Let $\tau=APQR$ with $A(\alpha,0,0,0), P(0,0,p_2,p_3), 
Q(0,q_1,0,q_3), R(0,r_1,r_2,0)$ 
be a $3$-dimensional lattice 
simplex in a compact facet of 
$\Gamma_+(f)$. %that is not of type $B_1$ ($\Longleftrightarrow$ $\alpha \geq 2$). 
Assume that $v=A$ 
and $\sigma =PQR$ are its proper
contributing $V$-faces to $
t_0 = e^{-2 \pi i \nu ( \tau )/N ( \tau )} \not= 1$. 
Then for the polynomial 
\begin{equation*}
F(t)= \frac{\zeta_{\tau}(t) \cdot 
(1-t)}{\zeta_v(t) \cdot 
\zeta_{\sigma}(t)} \in \CC [t] 
\end{equation*}
(see Lemma \ref{lemmatau5}) we have 
$F( t_0 )=0$. 
\end{lemma}
\begin{example}
The only $V$-facet of the function $f(x_1,x_2,x_3,x_4)=x_1^2+x_2^3x_3^3+x_2^3x_4^3+x_3^3x_4^3$ 
has the extended family as in Lemma \ref{FITH} and contributes the pole $s_0=-3/4$. 
One easily checks that the multiplicity of the corresponding monodromy eigenvalue $t_0=e^{-8\pi i/3}$ equals $18-9+1=10\ne 0$.
\end{example}
\begin{proof} Writing $F(t)$ by the definition as
\begin{equation*}
F(t)
=
\frac{(1-t^{ N(\tau) }
)^{ \Vol_{\ZZ}(\tau)} 
\cdot (1-t)}{
(1-t^{ N(v) })^1 \cdot (1-t^{N(\sigma)}
)^{\Vol_{\ZZ}(\sigma)
}},
\end{equation*}
the statement is obvious if $\Vol_{\ZZ}(\tau) > 
\Vol_{\ZZ}(\sigma)+1$.
Since $\Vol_{\ZZ}(\sigma)$ divides $\Vol_{\ZZ}(\tau)$ by Proposition \ref{PLT}, 
the only exceptions from the first inequality would be the cases

1) $\Vol_{\ZZ}(\tau) = 
\Vol_{\ZZ}(\sigma )$ or

2) $\Vol_{\ZZ}(\tau) = 2,\, 
\Vol_{\ZZ}(\sigma)=1$.

%However, the second case is excluded by Lemma \ref{lemmalinalg2}, and in 

In the first case the sought statement follows from Lemma \ref{lemmatau5}.2. Suppose now that $\Vol_{\ZZ}(\tau) = 2,\, 
\Vol_{\ZZ}(\sigma)=1$. Let 
\begin{equation*}
{\rm aff} (\tau) : \  ax_1+bx_2+cx_3+dx_4
=N(\tau)=a \cdot \alpha
\end{equation*}
be the equation of ${\rm aff} (\tau)$ 
with ${\rm gcd}(a,b,c,d)=1$. If the $V$-face $\sigma$ 
does not contribute to $t_0$, then obviously $F(t_0)=0$. If the $V$-face $\sigma$ 
does contribute to $t_0$, then
$$N(\sigma)\frac{(a+b+c+d)}{N(\tau)} \in \ZZ.$$ As $N(\sigma)=N(\tau)/{\rm gcd}(b,c,d)$, this would imply that
$$\frac{(a+b+c+d)}{{\rm gcd}(b,c,d)} \in \ZZ.$$
As ${\rm gcd}(a,b,c,d)=1$,
one gets that ${\rm gcd}(b,c,d)=1$ and hence $N(\sigma)=N(\tau)$.
As $$\Vol_{\ZZ}(\tau) =\frac{\alpha \cdot {\rm det} (M)}{N(\tau)}=2 
\quad \mbox{ and } \quad \Vol_{\ZZ}(\sigma) =\frac{{\rm det} (M)}{N(\sigma)}=1,$$
where
\begin{equation*}
M= \left( \begin{array}{ccc}
 0 &q_1 &r_1 \\
p_2 &0 &r_2 \\
p_3 &q_3 &0
 \end{array} \right),
\end{equation*} we find $\alpha=2$.
This case is excluded by Lemma \ref{lemmalinalg2}.

\begin{comment}
Since the $V$-face $\sigma$ 
contributes to $t_0$, by the 
proof of Proposition \ref{Key-1} we have 
$\Vol_{\ZZ}(\tau) = \alpha \cdot 
\Vol_{\ZZ}(\sigma )$. 
Let 
\begin{equation*}
{\rm aff} (\tau) : \  ax_1+bx_2+cx_3+dx_4
=N(\tau)=a \cdot \alpha
\end{equation*}
be the equation of ${\rm aff} (\tau)$ 
with ${\rm gcd}(a,b,c,d)=1$. Then %by the proof of Lemma \ref{lemmatau5} 
the equivalent (by Proposition \ref{PLT} and Lemma \ref{LLT}) conditions 
\begin{equation*}
\Vol_{\ZZ}(\tau)= \alpha \cdot \Vol_{\ZZ}(\sigma) \ 
\Longleftrightarrow \ 
{\rm gcd}( N(v), N( \sigma ))= \alpha  
\end{equation*}
imply that $N( \tau )= a \cdot \alpha \cdot 
{\rm gcd}(b,c,d)$, $N( \sigma )= a \cdot \alpha$ 
and $N(v)= \alpha = \alpha \cdot {\rm gcd}(b,c,d)$. 
Hence we get ${\rm gcd}(b,c,d)=1$, 
$N( \tau )=N( \sigma )=a \cdot \alpha$ and 
\begin{equation*}
F(t)
=
\frac{(1-t^{ a \cdot \alpha }
)^{\alpha \cdot \Vol_{\ZZ}(\sigma)} 
\cdot (1-t)}{
(1-t^{ \alpha }) \cdot (1-t^{a \cdot \alpha}
)^{\Vol_{\ZZ}(\sigma)
}}.
\end{equation*}
It follows from the assumption $\alpha \geq 2$ 
that we have $F( t_0 )=0$ unless $\alpha =2$ 
and $\Vol_{\ZZ}(\sigma )=1$. 
As shown in Lemma \ref{lemmalinalg2} this can never occur.
\end{comment}
\end{proof}

It now remains to study contributions of $B$-borders (Definition \ref{defborder}), generalizing
the proof of Theorem 15 of \cite{L-V}.

\begin{lemma}\label{otherconfigB1}
Assume that $f(x) \in \mathbb{C}[x_1,\ldots,x_4]$ is 
non-degenerate at the origin $0 \in \mathbb{C}^4$.
Let $\tau_1$ and 
$\tau_2$ be compact $B$-facets in $\Gamma_+(f)$ 
intersecting in a 
$B$-border $AA_1A_2$ with a $V$-edge $A_1A_2$, and contributing the same candidate pole 
$s_0 \neq 1$.

Assume additionally for every $A_i$ the following: if it is a $V$-vertex, contributing to the eigenvalue $t_0=\exp(-2\pi i s_0)$, then the edge $AA_i$ is in a coordinate plane.

Then the contribution of the family of the $V$-edge $A_1A_2$ of this border to the multiplicity of the corresponding eigenvalue $t_0=\exp(-2\pi i s_0)$ is positive, i.e.
for the polynomial 
\begin{equation*}
F(t)= \frac{\zeta_{A_1A_2}(t)\cdot (1-t)}{\zeta_{A_1}(t) \cdot 
\zeta_{A_2}(t)} \in \CC [t] 
\end{equation*}
we have $F( t_0 )=0$.

\end{lemma}

 \begin{proof}
Without loss of generality, we can assume 
that both $\tau_1$ and $\tau_2$ are compact 
simplicial $B_1$-facets: indeed, if $\tau_i$ 
is a $B_2$-facet, then it contains a 
$B_1$-tetrahedron with the same $B$-border, 
and we can consider this $B_1$-tetrahedron instead of $\tau_i$.

We redenote $\tau_1=ABCD$ and $\tau_2=ABCE$ such that $A=(1,1,\alpha_3,\alpha_4), 
B=(0,0,\beta_3,\beta_4), C=(0,0,\gamma_3,\gamma_4), 
D=(0,\delta_2,\delta_3,\delta_4)$ and $E=(\e_1, 0,
\e_3, \e_4)$. By computing the equation of the affine space passing through $A,B,C$ and $D$, we compute the candidate pole $s_0$ contributed by $\tau_1$ and $\tau_2$ :
$$s_0 =  \frac{(\gamma_4-\beta_4)(\alpha_3-\beta_3-1)+(\beta_3-\gamma_3)(\alpha_4-\beta_4-1)}{\beta_3(\gamma_4-\beta_4)+\beta_4(\beta_3-\gamma_3)}.$$
As the lattice index of the V-segment $BC$ is equal to the absolute value of
$$\frac{\beta_3(\gamma_4-\beta_4)+\beta_4(\beta_3-\gamma_3)}{{\rm gcd}(\gamma_4-\beta_4,\beta_3-\gamma_3)},$$
the V-face $BC$ contributes to $t_0$.
If neither $B$ nor $C$ contribute to the eigenvalue $t_0$, then the statement is obvious by projecting along $BC$.

So first suppose that $B$ is a contributing V-vertex 
and $C$ is not. Say $\beta_3=0$.
Then, by the additional assumption in the statement of the lemma, we have $A=(1,1,0,\alpha_4)$, so the affine space 
passing through the facet $ABCD$ has as equation:
\[x_1\left[\gamma_3(\delta_4-\beta_4) - \delta_3(\gamma_4-\beta_4) 
- \delta_2\gamma_3(\alpha_4-\beta_4)\right]+x_2\left[
\delta_3(\gamma_4-\beta_4)-
\gamma_3(\delta_4-\beta_4)\right]+\]\[x_3\left[\delta_2
(\beta_4-\gamma_4)\right]+
x_4(\gamma_3\delta_2)=\beta_4\gamma_3\delta_2.\]
The corresponding candidate pole is then
\[\frac{\gamma_3(\alpha_4-\beta_4)+\gamma_4-\beta_4-
\gamma_3}{\beta_4\gamma_3}.\]
As $B$ also contributes to $t_0$, one has 
that $\gamma_3$ divides $\beta_4-\gamma_4$. Notice that $\mbox{Vol}_{\mathbb{Z}}(BC)=\gamma_3$ and that $\gamma_3>1$ (otherwise we have a $B^2$-border).

We now suppose that both $B$ and $C$ are contributing V-vertices to $t_0$. Then, by the additional assumption in the statement of the lemma, we have $A=(1,1,0,0)$, so as before one gets 
$\mbox{Vol}_{\mathbb{Z}}(BC)>1$. As $C$ contributes to $t_0$, 
one would have a cancelation if $\mbox{Vol}_{\mathbb{Z}}(BC)=2$. Now
\[s_0=\frac{-\gamma_3\beta_4-\beta_4-\gamma_3}{
\beta_4 \gamma_3}\] and
so $\beta_4$ should divide $\gamma_3$. We have 
$\mbox{Vol}_{\mathbb{Z}}(BC)=gcd(\beta_4,\gamma_3)=\gamma_3$, 
because $B$ contributes to $t_0$. If $\mbox{Vol}_{\mathbb{Z}}(BC)=2$, 
then $\gamma_3=2$ and $\beta_4 = 2$ (otherwise we have a $B^2$-border), so $t_0=1$.

\end{proof}

\section{Appendix: some elements of lattice geometry}\label{appendix}
We recall some basic notions and facts about the geometry of ${\mathbb Z}^n$ that we use throughout the paper.
A latticed space is a real affine space $A$ with an integer lattice $L\subset A$ such that $\dim L=\dim A$. For instance:

-- ${\mathbb R}^n$ will be always considered a latticed space with the lattice ${\mathbb Z}^n$;

-- A rational affine subspace $A\subset {\mathbb R}^n$ will be always considered a latticed space with the lattice $A\cap{\mathbb Z}^n$;

-- The quotient space of ${\mathbb R}^n$ along its rational affine subspace $A$ (i.e. ${\mathbb R}^n/(A-a)$ for $a\in A$) will be always considered a latticed space, whose lattice is the image of ${\mathbb Z}^n$ under the quotient map.

A lattice polytope in a latticed space $A$ is a polytope all of whose vertices belong to the lattice.

The lattice volume form on a latticed space $A$ with the lattice $L$ is the volume form such that the volume of $A/L$ equals $(\dim A)!$, or, equivalently, the minimal positive volume form such that the volume of every lattice polytope in $A$ is integer.

A segment in a latticed space $A$ is said to be primitive, if its end points are the only lattice points that it contains. The lattice distance between two lattice points $a$ and $b$ is the number of primitive segments into which the lattice points subdivide the segment $ab$. In coordinates, the lattice distance between $a$ and $b\in{\mathbb Z}^n$ is the GCD of the coordinates of the difference $b-a$.

The lattice distance from a lattice affine subspace $A\subset{\mathbb R}^n$ to the origin can be defined in one of the following equivalent ways:
 
I) it is the lattice distance between 0 and the image of $A$ under the projection $p$ of ${\mathbb R}^n$ along $A$; 
in particular, if $A$ is a hypersurface given by an equation $a_1v_1+\cdots+a_nv_n=q$ with coprime integer coefficients $a_i$ and $q$, then the lattice distance from $A$ to 0 equals $|q|$.

II) It is the maximum of lattice distances between the points of $A$ and 0.

\begin{remark}\label{remaffdivide}
By the definition, the lattice distance from 0 to any point of $A$ divides the distance from $0$ to $A$. As a consequence, the distance from any affine subspace $A'\subset A$ divides the distance from $0$ to $A$. 
\end{remark}

The lattice distance of a lattice polytope $P$ in ${\mathbb R}^n$ to the a lattice point $a$ is defined as the lattice distance from the affine hull of the shifted polytope $P-a$ to the origin.

%, denoted by $N(P)$, is the lattice distance from its affine hull to the origin.

Metric computations with lattice length and distances naturally translate into the lattice setting. For instance, the following statements directly follows from their well known metric versions: 

\begin{proposition}\label{PLT}
1) The lattice volume of a lattice pyramid in ${\mathbb R}^n$ equals the lattice volume of its base times the lattice distance from the base to the apex.

2) The lattice volume of an $(n-1)$-dimensional simplex in $\ZZ^n$ generated by vectors $v_1\ldots,v_{n-1}$ is equal to the lattice length of the vector product $v_1\wedge\cdots\wedge v_{n-1}$.
\end{proposition}

Denote the lattice distance from a lattice polytope $P\in{\mathbb R}^n$ to the origin by $N(P)$.
\begin{remark}\label{remaffdivide2}
1) If $\pi$ is the projection of ${\mathbb R}^n$ along an affine subspace of the affine hull of the polytope $P$, then $N(P)=N(\pi(P))$ by the  definition of the lattice distance.

2) If $j$ is the embedding ${\mathbb R}^n\to{\mathbb R}^n\oplus{\mathbb R}^m$, then $N(P)=N(j(P))$.
\end{remark}

\begin{lemma}\label{LLT} 
%For two faces $\tau, \gamma \prec \Gamma_+(f)$ 
%such that $\gamma \prec \tau$, 
1) For a lattice polytope $\tau\subset{\mathbb R}^n$ and its face $\gamma$,
we have $N( \gamma ) | N( \tau )$. 

2) For lattice polytopes $\gamma\subset{\mathbb R}^k$ and $\gamma'\subset{\mathbb R}^{n-k}$ and their affine hull $\tau\subset {\mathbb R}^k\oplus{\mathbb R}^{n-k}$, we have $N(\tau)=LCM(N(\gamma),N(\gamma'))$.

3) If $\gamma'$ is a point in the setting of Part 2, then the lattice distance from $\gamma'$ to $\gamma$ equals $GCD(N(\gamma),N(\gamma'))$.
\end{lemma} 

\begin{proof} Part 1 rephrases Remark \ref{remaffdivide}. Parts 2 and 3 are obvious if $\gamma$ and $\gamma'$ are points and $k=n-k=1$. The general case reduces to this one by Remark \ref{remaffdivide2} for the projections of ${\mathbb R}^k$ and ${\mathbb R}^{n-k}$ along the affine hulls of $\gamma$ and $\gamma'$.
%We may assume that $\gamma$ is a facet of $\tau$. Let $\Sigma_0$ be the dual fan of $\Gamma_+(f)$ and $\Sigma$ a smooth subdivision of $\Sigma_0$. Then there exists an $( n- \dim \gamma )$-dimensional smooth cone $\Delta \in \Sigma$ contained in $\gamma^{\circ}$ such that $\dim ( \Delta \cap \tau^{\circ}) = \dim \tau^{\circ} = n- \dim \tau =(n- \dim \gamma )-1$. Set $l= n- \dim \gamma$ and let $a(1), \ldots, a(l) \in \ZZ_+^n$ be the primitive vectors on the edges of $\Delta$. Then it is easy to see that for any $v \in \gamma$ we have \begin{equation*} N( \gamma ) = {\rm gcd} \Bigl( \langle a(1), v \rangle, \ldots, \langle a(l), v \rangle \Bigr). \end{equation*}
%We have also a similar description of $N( \tau )$ in terms of the primitive vectors on the edges of $\Delta \cap \tau^{\circ} \prec \Delta$. If we use a point $v \in \gamma \prec \tau$ to express $N( \gamma )$ and $N( \tau )$ simultaneously, we find $N( \gamma ) | N( \tau )$. This completes the proof. 
\end{proof}

\begin{comment}
\begin{remark}
For two $V$-faces $\tau, \gamma \prec \Gamma_+(f)$ 
such that $\gamma \prec \tau$ we can prove Lemma 
\ref{LLT} more easily as follows. 
Let $\RR^{S_{\tau}}$ (resp. $\RR^{S_{\gamma}}$) 
be the minimal coordinate subspace of $\RR^n$ 
containing $\tau$ (resp. $\gamma$) and 
\begin{equation*}
{\rm aff}( \tau ) : \  
\langle a( \tau ), v \rangle =N( \tau )  
\end{equation*}
the equation of the affine hyperplane 
${\rm aff}( \tau ) \subset \RR^{S_{\tau}}$. 
Then by restricting it to 
$\RR^{S_{\gamma}} \subset \RR^{S_{\tau}}$ 
to find that of ${\rm aff}( \gamma )$ 
we obtain the equality 
\begin{equation*}
{\rm gcd}  \Bigl\{ a( \tau )_i \ 
(i \in S_{\gamma})  \Bigr\} 
\cdot N( \gamma )  =N( \tau ).  
\end{equation*}
\end{remark} 
\end{comment}
Besides these geometric observations, in the study of exotic families, we use the following observation from the integer linear algebra. Let us first fix notation. For a square matrix $A$ of size $n \times n$ and for $I, J \subset \{1,\ldots,n\}$ with same cardinality, we denote $A^I_J$ for the minor of $A$ removing the rows with index in the set $I$ and 
removing the columns with index in the set $J$. 
When $I=\{x_i\}$ is a singleton, we will also write $A^{x_i}_J$ and idem for $J$.
\begin{lemma}\label{b}
Let $A \in M_{n}(\ZZ)$ be a square matrix of size $n \times n$ with integer entries.  Let $P=(P_1,\ldots,P_n)^t$ be a primitive vector of size $n$ such that $AP=0$.
Then $P_n$ divides the minors $A^{i}_{n}$, for $  1 \leq  i \leq n$.
\end{lemma}

\begin{proof}
We proceed by induction on $n$. One easily verifies that the statement is true when $n=2$.
We now take a matrix $A  \in M_{n}(\ZZ)$ of size $n > 2$. By elementary row operations over $\ZZ$, we transform the matrix $A$ in a matrix $B$ having first column $(k,0,\ldots,0)^t,$ with $k \in \ZZ$.
Then it still holds that $BP=0$. We set $d={\rm gcd}(P_2,\ldots,P_n)$. As $P$ is a primitive vector, we deduce that $d$ divides $k$.
By the induction hypothesis we have that $P_n/d$ divides the minors $B^{\{1,i\}}_{\{1,n\}}$, for all $i \in \{2,\ldots,n\}$. Hence $P_n$ divides $kB^{\{1,i\}}_{\{1,n\}}=B^{i}_{n}$, for all $i \in \{2,\ldots,n\}$. As $B^{1}_{n}=0$, we get
that $P_n$ divides $B^{i}_{n}$, for all $i \in \{1,\ldots,n\}$ and so $P_n$ also divides the original minors $A^{i}_{n}$, for all $i \in \{1,\ldots,n\}$.
\end{proof}

\noindent
\textsc{A. Esterov\\ National Research University 
Higher School of Economics \\
Faculty of Mathematics NRU HSE, 6 Usacheva, 119048,
Moscow, Russia} \\

\vspace{-0.4cm}
\noindent \emph{E-mail address}: aesterov@hse.ru
\\ \\
\noindent
\textsc{A. Lemahieu\\ Universit\'e C\^ote d'Azur, CNRS, Laboratoire J.-A. Dieudonn\'e \\UMR CNRS 7351, Parc Valrose, 06108 Nice Cedex 02, France
}\\

\vspace{-0.4cm}
\noindent \emph{E-mail address}: ann.lemahieu@unice.fr
\\ \\
\noindent
\textsc{K. Takeuchi\\ 
Mathematical Institute, Tohoku University\\
Aramaki Aza-Aoba 6-3, Aobaku, Sendai, 980-8578, Japan}\\

\vspace{-0.4cm}
\noindent \emph{E-mail address}: takemicro@nifty.com 

\end{document}